\documentclass[10pt]{article}

\usepackage[OT2,OT1]{fontenc}
\usepackage{amsfonts,amsmath,amssymb,latexsym,mathrsfs,eucal,upgreek}
\usepackage[all,cmtip]{xy}
\usepackage{hyperref}
\hypersetup{colorlinks=false}
\usepackage{float}
\usepackage[margin=0.9in]{geometry}

\usepackage{amsthm}
\usepackage{mathtools}
\usepackage{manfnt}
\usepackage{enumitem}
\usepackage{setspace}
\usepackage{calc,footnote}
\usepackage{graphicx}
\usepackage{tikz}
\usepackage{tikz-cd}
\usepackage{comment}
\usepackage[textwidth=2cm,textsize=tiny]{todonotes}

\newtheorem{theorem}{Theorem}[section]
\newtheorem{proposition}[theorem]{Proposition}
\newtheorem{lemma}[theorem]{Lemma}
\newtheorem{corollary}[theorem]{Corollary}
\newtheorem{conjecture}[theorem]{Conjecture}

\theoremstyle{definition}

\renewenvironment{proof}{\noindent {\bf Proof:}}{$\Box$ \vspace{2 ex}}
\newtheorem{remark}[theorem]{Remark}

\newtheorem*{theorem*}{Theorem}
\newtheorem*{proposition*}{Proposition}
\newtheorem*{lemma*}{Lemma}

\theoremstyle{remark}

\newcommand{\defeq}{\vcentcolon=}

\newcommand\nc{\newcommand}
\nc{\on}{\operatorname}
\nc\renc{\renewcommand}
\nc{\BR}{\mathbb R}
\nc{\BC}{\mathbb C}
\nc{\BQ}{\mathbb Q}
\nc{\BZ}{\mathbb Z}
\nc{\BN}{\mathbb N}
\nc{\BP}{\mathbb P}
\nc{\BA}{\mathbb A}
\nc{\Hom}{\on{Hom}}
\nc{\wt}{\widetilde}
\nc{\vspan}{\on{span}}
\nc{\ord}{\on{ord}}
\nc{\im}{\on{im}}
\nc{\Mat}{\on{Mat}}
\nc{\can}{\on{can}}
\nc{\coker}{\on{coker}}
\nc{\ev}{\on{ev}}
\nc{\Tr}{\on{Tr}}
\nc{\End}{\on{End}}
\nc{\swap}{\on{swap}}
\nc{\Set}{\on{Set}}
\nc{\bC}{{\bm C}}
\nc{\bc}{{\bm c}}
\nc{\bD}{{\bm D}}
\nc{\bd}{{\bm d}}
\nc{\bE}{{\bm E}}
\nc{\be}{{\bm e}}
\nc{\bff}{{\bm f}}
\nc{\Bf}{\mathbb{f}}

\nc{\adj}{\on{adj}}
\nc{\tensor}[3]{#1 \underset{#2}\otimes #3}
\nc{\Nat}{\on{Nat}}
\nc{\op}{\on{op}}
\nc{\funct}{\on{funct}}
\nc{\Ob}{\on{Ob}}
\nc{\fR}{\mathfrak{R}}
\nc{\Vect}{\on{Vect}}
\nc{\ns}{\on{non-spec}}
\nc{\ol}{\overline}
\nc{\ul}{\underline}
\nc{\w}{\omega}
\nc{\nlog}{\on{nlog}}
\nc{\Aut}{\on{Aut}}
\nc{\Gal}{\on{Gal}}
\nc{\Poss}{\on{Poss}}
\nc{\ksep}{\on{sep}}
\nc{\low}{\on{low}}
\nc{\Stab}{\on{Stab}}
\nc{\pp}{\mathfrak{p}}
\nc{\OO}{\mathcal{O}}
\nc{\mm}{\mathfrak{m}}
\nc{\qq}{\mathfrak{q}}
\nc{\Nm}{\on{Nm}}
\nc{\Ann}{\on{Ann}}
\nc{\gug}{\mathfrak{g}}
\nc{\hug}{\mathfrak{h}}
\nc{\mf}{\mathfrak}
\nc{\mc}{\mathcal}
\nc{\Sym}{\on{Sym}}

\def\Z{{\mathbb Z}}
\def\I{{\mathbb I}}
\def\End{{\rm End}}
\def\Sym{{\rm Sym}}
\def\Sel{{\rm Sel}}
\def\SL{{\rm SL}}
\def\GL{{\rm GL}}
\def\PGL{{\rm PGL}}

\def\var{{\rm var}}
\def\mon{{\rm mon}}
\def\Stab{{\rm Stab}}
\def\Sym{{\rm Sym}}
\def\Jac{{\rm Jac}}

\def\O{{\mathcal O}}
\def\cO{{\mathcal O}}
\def\SO{{\rm SO}}
\def\PSO{{\rm PSO}}
\def\P{{\mathbb P}}

\def\adj{{\rm adj}}
\def\Aut{{\rm Aut}}
\def\irr{{\rm irr}}
\def\gen{{\rm gen}}
\def\ngen{{\rm ngen}}
\def\inv{{\rm inv}}
\def\Inv{{\rm Inv}}

\def\Vol{{\rm Vol}}
\def\R{{\mathbb R}}
\def\F{{\mathbb F}}
\def\FF{{\mathcal F}}
\def\cN{{\mathcal N}}
\def\cM{{\mathcal M}}

\def\cS{{\mathcal S}}

\def\RR{{\mathcal R}}

\def\Q{{\mathbb Q}}
\def\H{{\mathcal H}}

\def\J{{\mathcal J}}
\def\cA{{\mathcal A}}
\def\DA{{\mathcal A}}
\def\cB{{\mathcal B}}
\def\cD{{\mathcal D}}

\def\cI{{\mathcal I}}

\def\Z{{\mathbb Z}}
\def\P{{\mathbb P}}
\def\F{{\mathbb F}}
\def\Q{{\mathbb Q}}
\def\C{{\mathbb C}}

\def\L{{\mathcal L}}

\def\six{{\upkappa}}

\makeatletter
\def\@tocline#1#2#3#4#5#6#7{\relax
  \ifnum #1>\c@tocdepth
  \else
    \par \addpenalty\@secpenalty\addvspace{#2}
    \begingroup \hyphenpenalty\@M
    \@ifempty{#4}{
      \@tempdima\csname r@tocindent\number#1\endcsname\relax
    }{
      \@tempdima#4\relax
    }
    \parindent\z@ \leftskip#3\relax \advance\leftskip\@tempdima\relax
    \rightskip\@pnumwidth plus4em \parfillskip-\@pnumwidth
    #5\leavevmode\hskip-\@tempdima
      \ifcase #1
       \or\or \hskip 1em \or \hskip 2em \else \hskip 3em \fi
      #6\nobreak\relax
    \dotfill\hbox to\@pnumwidth{\@tocpagenum{#7}}\par
    \nobreak
    \endgroup
  \fi}
\makeatother

 \allowdisplaybreaks

\begin{document}

\title{The second moment of the size of the $2$-Selmer group \\ of elliptic curves}

\author{Manjul Bhargava, Arul Shankar, and Ashvin A.~Swaminathan}

\maketitle

\begin{abstract}
In this paper, we prove that when elliptic curves over $\Q$ are ordered by height, the second moment of the size of the $2$-Selmer group is at most $15$. This confirms a conjecture of Poonen and Rains.
\end{abstract}

\setcounter{tocdepth}{2}

\tableofcontents

\section{Introduction}

\subsection{Statements of main results}

Recall that every elliptic curve over $\Q$ is isomorphic to a unique curve of the form $$E^{I,J} \colon y^2 = x^3 -\tfrac{I}{3}x-\tfrac{J}{27},$$
where $I,J$ are integers satisfying the following three conditions: $3 \mid I$, $27 \mid J$, and for every prime $p$, we have that $3p^4 \nmid I$ if $27p^6 \mid J$. The height of $E^{I,J}$ is defined to be $H(E^{I,J}) \defeq H(I,J) \defeq \frac{4}{27}\max\{|I|^3,J^2/4\}$.

In the series of papers~\cite{MR3272925,MR3275847,4sel,5sel}, the first- and second-named authors determined, for each $n \in \{2,3,4,5\}$, the average size of the $n$-Selmer group in the universal family of elliptic curves over $\Q$ ordered by height. The purpose of this article is to prove an analogous result for the second moment of the size of the $2$-Selmer group. Specifically, we prove the following theorem:

\begin{theorem} \label{thm-main}
When elliptic curves over $\Q$ are ordered by height, the second moment of the size of the $2$-Selmer group is at most $15$.
\end{theorem}

In fact, we prove that Theorem~\ref{thm-main} holds for any subfamily of elliptic curves defined by any finite set or any acceptable infinite set of congruence conditions, and further that the upper bound of $15$ is tight if one assumes a certain plausible tail estimate.

A key ingredient in the proof of Theorem~\ref{thm-main} is the determination of the average size of the $2$-Selmer group of the Jacobian of locally soluble genus-$1$ curves:

\begin{theorem} \label{thm-main2}
Consider the family of locally soluble genus-$1$~curves of the form $z^2 = f(x,y)$, where $f$ is an irreducible integral binary quartic form up to $\on{PGL}_2(\Q)$-action. When these curves are ordered by the heights of their Jacobian elliptic curves, the average size of the $2$-Selmer group of the Jacobian is at most $6$.
\end{theorem}

Theorems~\ref{thm-main} and~\ref{thm-main2} confirm remarkable heuristics due to Poonen and Rains (see~\cite{MR2833483}, cf.~\cite{MR3393023}), who modeled the $p$-Selmer group, where $p$ is a prime, as a random intersection of a discrete maximal isotropic subspace with a compact open maximal isotropic subspace in a locally compact quadratic space over the finite field $\mathbb{F}_p$. Using this model, they arrived at a conjecture for the distribution of the $p$-Selmer group of elliptic curves. As far as the moments of the distribution are concerned, Poonen and Rains made the following precise prediction:
\begin{conjecture}[\protect{\cite[Conjecture~1.1(c)]{MR2833483}}] \label{conj-poonenrains}
Let $p$ be prime. The $m^{\text{th}}$ moment of the size of the $p$-Selmer group of elliptic curves over $\Q$ is $\prod_{i = 1}^m (p^i+1)$.
\end{conjecture}
The only cases of Conjecture~\ref{conj-poonenrains} that have been proven in the literature are the cases where $m = 1$ and $p \in \{2,3,5\}$ (see~\cite{MR3272925,MR3275847,5sel}). Theorem~\ref{thm-main} constitutes the very first result toward Conjecture~\ref{conj-poonenrains} in the case where $m > 1$. We note that the distribution of the $p$-Selmer group conjectured by Poonen and Rains is also expected to hold for various special families of elliptic curves, and for $p = 2$, significant progress has been made toward proving the conjecture for such families: In~\cite{MR1193603}, Heath-Brown determined the distribution of the $2$-Selmer group for the family of congruent number curves, and Swinnerton-Dyer (see~\cite{MR2464773}) and Kane (see~\cite{MR3101079}) showed that the same distribution holds for any family of quadratic twists of a single elliptic curve with full rational $2$-torsion. Considerably more is known about $2$-Selmer groups in quadratic twist families; see, for example, the work of Klagsbrun, Mazur, and Rubin~\cite{MR3043582} computing the proportion of twists of a given elliptic curve with even $2$-Selmer rank, and the very recent work of Smith~\cite{alex} determining (under certain technical conditions) the distribution of the $2^k$-Selmer group in quadratic twist families of elliptic curves for every $k \geq 1$.

Poonen and Rains also proposed a conjecture (see~\cite[Conjecture~1.8]{MR2833483}) for the distribution of the $2$-Selmer group of the Jacobian of even-degree hyperelliptic curves of even genus; in particular, they predict that the average size of the $2$-Selmer group is $6$. In the forthcoming work~\cite{BSSpreprint2}, we extend the methods of the present article to prove, as predicted, that the average is at most $6$ for the universal family of locally soluble hyperelliptic curves of any given genus.

\subsection{Background on hyperelliptic curves and their Selmer groups}

Let $C/\Q$ be a hyperelliptic curve of positive genus $g$. Recall that $C$ is isomorphic to a curve of the form
$$C_f \colon z^2 = f(x,y) = f_0x^n + f_1x^{n-1}y + \cdots +  f_ny^n,$$
where $n = 2g+2$ and the binary form $f = f(x,y)$ is separable with integer coefficients $f_i$, and where we regard $C_f$ as a curve embedded in the weighted projective plane $\mathbb{P}_{\Q}^2(1,1,g+1)$ with coordinates $[x : y : z]$.

Let $J = \on{Jac}(C)$ be the Jacobian of $C$; if $C = C_f$, we write $J_f = \on{Jac}(C_f)$. Recall that an \emph{$n$-covering of $J$} is an unramified covering $\pi \colon Y \to J$ by a principal homogeneous space $Y$ for $J$, all defined over $\Q$, with the property that $\pi(y + a) = \pi(y) + na$ for any $y \in Y$ and $a \in J$. The degree of any $n$-covering is $n^{2g}$, and the fiber over the origin gives a principal homogeneous space for the $n$-torsion subgroup $J[n]$. The class of this homogeneous space in the Galois cohomology group $H^1(\Q, J[n])$ determines the $n$-covering $\pi$ up to isomorphism.

We say that a variety over $\Q$ is \emph{soluble} if it has a point defined over $\Q$ and \emph{locally soluble} if it has points defined over $\R$ and over $\Q_p$ for every prime $p$. The elements of the $n$\emph{-Selmer group} $\on{Sel}_n(J) \subset H^1(\Q, J[n])$ are in bijection with the isomorphism classes of locally soluble $n$-coverings of $J$; under this correspondence, the elements of the subgroup $J(\Q)/nJ(\Q) \subset \on{Sel}_n(J)$ map to the isomorphism classes of soluble $n$-coverings of $J$.

 An \emph{$n$-covering of $C$} is an unramified covering $\pi \colon D \to C$ defined over $\Q$ that is Galois over $\ol{\Q}$ and whose Galois group is isomorphic to $J[n](\ol{\Q})$ as a $\on{Gal}(\ol{\Q}/\Q)$-module. If $C$ is soluble, then the Abel-Jacobi map realizes $C$ as a subvariety of $J$, and pullback under the inclusion $C \hookrightarrow J$ defines a bijection between the $n$-coverings of $J$ and those of $C$. When $g = 1$, this bijection preserves both solubility and local solubility.

\subsection{Summary of the proof}

Computing the second moment of $|\on{Sel}_2(E)|$ amounts to determining the average size of the set $\on{Sel}_2(E)^2$ of pairs $(\pi_1, \pi_2)$ of isomorphism classes of locally soluble $2$-coverings of $E$. We start by giving a simple heuristic for why the second moment of $|\on{Sel}_2(E)|$ should be $15$. To do this, fix a locally soluble $2$-covering $\pi_1$ of an elliptic curve $E$, and consider the following two cases:

\vspace*{0.2cm}
\noindent \emph{Case 1}:
Suppose $\pi_1$ corresponds to the trivial element of $\on{Sel}_2(E)$. Then $E$ can be arbitrary, and~\cite[Theorem~1.1]{MR3272925} implies that there are $3$ choices for $\pi_2$ on average.

\vspace*{0.2cm}
\noindent \emph{Case 2}:
Suppose $\pi_1$ corresponds to a non-trivial element of $\on{Sel}_2(E)$. Then~\cite[Theorem~1.1]{MR3272925} implies that there are $2$ choices for $\pi_1$ on average. Moreover, the elliptic curve $E$ cannot be arbitrary --- it has a marked non-trivial $2$-Selmer element, namely the one corresponding to $\pi_1$. We expect that this marked element gives an extra generator for the $2$-Selmer group $100\%$ of the time, and so, there should be $2 \times 3 = 6$ choices for $\pi_2$ on average, twice as many choices as in case 1.

\vspace*{0.2cm}
\noindent Combining the results of cases $1$ and $2$, we conclude that there should be a total of $1 \times 3 + 2 \times 6 = 15$ choices for the pair $(\pi_1, \pi_2)$ on average, as desired.

Our proof of Theorem~\ref{thm-main} essentially makes the above heuristic rigorous. To compute the average number of pairs of locally soluble $2$-coverings of elliptic curves, we apply an enhanced version of the parametrize-and-count strategy used in~\cite{MR3272925,MR3275847,4sel,5sel} to determine the average value of $|\on{Sel}_n(E)|$ for $n \in \{2,3,4,5\}$. Our strategy consists of two pieces: an algebraic piece, in which one parametrizes the objects of interest in terms of the rational orbits of a certain coregular representation and constructs integral representatives for these orbits; and an analytic piece, in which one combines geometry-of-numbers and equidistribution techniques to count these integral representatives.

To parametrize pairs of locally soluble $2$-coverings, we start by recalling the parametrization of individual locally soluble $2$-coverings in terms of $\on{PGL}_2(\Q)$-orbits of binary quartic forms over $\Q$. For a binary quartic form $f \in \Q[x,y]$, we call $f$ (locally) soluble if $C_f$ is. Letting $E = E^{I,J}$ for integers $I,J$, the parametrization states that the map $f \mapsto C_f$ defines a bijection between the set of $\on{PGL}_2(\Q)$-orbits of (locally) soluble binary quartic forms with $\on{PGL}_2$-invariants $I,J$ and the set of isomorphism classes of (locally) soluble $2$-coverings of $E$.

 Our task is then to parametrize pairs $(f,f')$ of locally soluble binary quartic forms that have the same invariants $I,J$. We do this by fibering over the first element of the pair. Fix a locally soluble binary quartic form $f \in \Q[x,y]$ with invariants $I,J$. Let us first consider the case where $f$ is in fact soluble. Here, the set of locally soluble binary quartic forms $f' \in \Q[x,y]$ with invariants $I,J$ is in bijection with the locally soluble $2$-coverings of $C_f$ (i.e., elements of $\on{Sel}_2(J_f)$), or equivalently, locally soluble $4$-coverings of $E$ (i.e., elements of $\on{Sel}_4(E)$). As shown in~\cite{MR3800357}, such coverings naturally correspond to certain orbits of the group $(\on{SL}_4/\mu_2)(\Q)$ acting via change-of-basis on pairs $(A,B)$ of quaternary quadratic forms over $\Q$ such that, when $A$ and $B$ are expressed in terms of their Gram matrices, we have $\det(xA + yB) = f(x,y)$. Thus, we may parametrize pairs $(f,f')$ of binary quartic forms, where $f$ is soluble and $f'$ is locally soluble, in terms of $(\on{SL}_4/\mu_2)(\Q)$-orbits of pairs $(A,B)$ of quaternary quadratic forms such that $\det(xA + yB) = f(x,y)$.

The case where $f$ is not soluble is far trickier, because locally soluble binary quartic forms $f' \in \Q[x,y]$ with the same invariants as $f$ are not always realizable as pairs $(A,B)$ of quaternary quadratic forms over $\Q$ satisfying $\det(xA+yB) = f(x,y)$. Indeed, as shown in~\cite{MR3600041}, such a pair $(A,B)$ might not even exist if $f$ is insoluble over $\Q$! To circumvent the possibility that $f$ might not be soluble, we simply manufacture a $\Q$-rational point by replacing $f$ with its \emph{monicization} $f^{\mathrm{mon}}$, which is defined as follows: letting $f_0 \neq 0$ be the leading coefficient of $f$, we write $f^{\mathrm{mon}} \defeq f_0^{-1} \times f(x, f_0y)$. The soluble genus-$1$ curve $C_{f^{\mathrm{mon}}}$ is a quadratic twist of $C_f$ by $\Q(\sqrt{f_0})$, and its Jacobian is the elliptic curve $E^{f_0^2I, f_0^3J}$; we thus obtain identifications $E^{I,J}[2] \simeq E^{f_0^2I, f_0^3J}[2]$ and $J_f[2] \simeq J_{f^{\mathrm{mon}}}[2]$ of group schemes over $\Q$. These identifications induce isomorphisms
\begin{align}
H^1(\Q, E^{I,J}[2]) & \simeq H^1(\Q, E^{f_0^2I, f_0^3J}[2]), \quad \text{and} \label{eq-iso1h1}\\
H^1(\Q, J_f[2]) & \simeq H^1(\Q, J_{f^{\mathrm{mon}}}[2]) \label{eq-iso2h1}
\end{align}
in Galois cohomology. The isomorphism~\eqref{eq-iso1h1} maps the classes of the locally soluble $2$-coverings defined by $f$ and $f^{\mathrm{mon}}$ to each other. The isomorphism~\eqref{eq-iso2h1} allows us to identify locally soluble $2$-coverings of $J_f$ with certain $2$-coverings of $J_{f^{\mathrm{mon}}}$; crucially, we show that the $2$-coverings of $J_{f^{\mathrm{mon}}}$ arising in this way correspond to $(\on{SL}_4/\mu_2)(\Q)$-orbits of pairs $(A,B)$ of quaternary quadratic forms over $\Q$ such that $\det(xA + yB) = f^{\mathrm{mon}}(x,y)$, and further that these orbits have representatives defined over $\Z$ (up to a bounded factor). More precisely, we have the following result, which holds for binary forms of any even degree $n$:
\begin{theorem} \label{thm-2selparam}
Let $f \in \Z[x,y]$ be a separable binary form of even degree $n$, and suppose that $C_f$ is locally soluble if $n \equiv 0 \pmod 4$. There is a natural injective map of sets
\begin{equation} \label{eq-selmap}
\begin{tikzcd}
\on{Sel}_2(J_f) \arrow[hook]{r} & \left\{\begin{array}{c}
\text{$(\on{SL}_n/\mu_2)(\BQ)$-orbits of $(A,B)$ over $\Q$}
\\ \text{such that $\det(x A + y B) =  f^{\operatorname{mon}}(x,y)$} \end{array}\right\}
\end{tikzcd}
\end{equation}
Moreover, there exists an integer $\upkappa \geq 1$ depending only on $n$ such that if we replace the binary $n$-ic form $f(x,y)$ with $g(x,y) = f(x, \upkappa y)$, then the $(\on{SL}_n/\mu_2)(\BQ)$-orbit arising from an element of $\on{Sel}_2(J_g) \simeq \on{Sel}_2(J_f)$ contains an integral representative of the form $(\mc{A},B)$, where $\mc{A}$ denotes the anti-diagonal $n \times n$ matrix with all antidiagonal entries equal to $1$.
\end{theorem}
Let $f$ be as in Theorem~\ref{thm-2selparam}, and let $R_f$ denote the ring of global sections of the closed subscheme of $\mathbb{P}^1_{\Z}$ cut out by $f$. To prove the theorem, we first show that the elements of $\on{Sel}_2(J_f)$ correspond to square roots of the ideal class of the inverse different of $R_f$. We then apply a parametrization first introduced by the third-named author in~\cite{swacl}, which relates square roots of the class of the inverse different of $R_f$ to pairs $(A,B)$ of $n$-ary quadratic forms over $\Z$ such that $\det(xA + yB) = f^{\mathrm{mon}}(x,y)$. This parametrization is inspired by a related parametrization due to Wood
(see~\cite{MR3187931}). Wood's parametrization has been used to great effect in the study of rational points on hyperelliptic curves; for example, it was used in~\cite{thesource} (cf.~\cite{MR3600041}) to show that most hyperelliptic curves over $\BQ$ have no rational points, and in~\cite{MR3156850} (resp.,~\cite{MR3719247}) to show that the average size of the $2$-Selmer groups of hyperelliptic curves with a rational Weierstrass (resp., non-Weierstrass) point is at most $3$ (resp., $6$). In addition, Wood's parametrization has also been used to study average $2$-torsion in the class groups of rings defined by binary forms; see~\cite{MR3782066,Siadthesis1,Siadthesis2}.

\medskip

Once the parametrization has been established, and pairs $(f,f')$ of $2$-Selmer elements of elliptic curves have been mapped into $(\on{SL}_4/\mu_2)(\Z)$-orbits (equivalently, $\PSO_\cA(\Z)$-orbits; see \S\ref{sec-orth}) of pairs $(\cA,B)$ with $\det(x\cA+yB) = f^\mon(x,\upkappa y)$, we must then estimate the number of these orbits having bounded height. The analysis required to do this is considerably more complicated than that developed in previous works, such as~\cite{4sel} or~\cite{5sel}, since the current parametrization is not in terms of orbits for a group acting on a single lattice, but rather for a group acting on a family of increasingly sparse subsets of a lattice. Indeed, fix an integer $a \neq 0$. Then the set of forms $f^\mon$, as $f$ varies across binary quartic forms $ax^4+\cdots$, has density $a^{-6}$ in the set of all monic binary quartic forms. Hence, we expect that the set of pairs $(\cA,B)$ satisfying $\det(x\cA+yB) = f^\mon(x,y)$ for some $f(x,y)=ax^4+\cdots$ also has density around $a^{-6}$ in the set of all pairs $(\cA,B)$.

The leading coefficient $a$ of $f$ grows with the height of the elliptic curve in question. In more concrete terms, let $E = E^{I,J}$ be an elliptic curve such that $H(E)=X^6$. (For the purpose of developing heuristics, this happens to be the most convenient height normalization to take; there are $\asymp X^5$ elliptic curves with height $\leq X^6$.) The set of non-trivial elements in the $2$-Selmer group of $E$ (up to some slight subtleties at the prime $2$ that we will ignore) admits a surjective map from the set of $\PGL_2(\Z)$-orbits on the set of locally soluble binary quartic forms $f$ having invariants $I$ and $J$, such that $f$ has no rational linear factor. By definition, the height of each such form $f$ is also $X^6$. Reduction theory then implies that the $\PGL_2(\Z)$-orbit of $f$ has a representative $g(x,y)=ax^4+bx^3y+cx^2y^2+dxy^3+ey^4$, the coefficients of which satisfy the bounds
\begin{equation*}
|a|\ll \frac{X}{t^4},\quad
|b|\ll \frac{X}{t^2},\quad
|c|\ll X,\quad
|d|\ll t^2X,\quad
|e|\ll t^4X,
\end{equation*}
where $t\geq \sqrt[4]{3}/\sqrt{2}$ is necessarily $\ll X^{1/4}$, for otherwise $a$ is forced to be $0$.
In the balance of this section, we will assume that $t=1$ for the following reasons: this will simplify the notation considerably; the number of points when $t$ is large is negligible; and for small $t \neq 1$, the proof is identical to the proof when $t=1$.

The monicization $f^\mon$ has coefficients $1$, $b$, $ac$, $a^2d$, and $a^3d$, which satisfy the bounds
\begin{equation*}
|b|\ll X,\quad
|ac|\ll X^2,\quad
|a^2d|\ll X^3,\quad
|a^3e|\ll X^4.
\end{equation*}
Therefore, any lift $(\cA,B)$ with resolvent $f^\mon$ has a representative $(\cA,B')$ in its $\PSO_\cA(\Z)$-orbit such that the coefficients $b_{ij}'$ of $B$ are
all bounded by some constant times $X$, multiplied by some cuspidal factors that are easily bounded. Furthermore, the element $B'$, like $f$, satisfies congruence conditions having density around $a^{-6}$. (This density is asymptotic to $a^{-6}$ when $a$ is squarefree, but the situation is more complicated when $a$ is divisible by a high power of a prime and is addressed in Lemma~\ref{lemspecialden}.) One key feature of the parametrization is that these congruence conditions are actually defined modulo $a$, and not some higher power thereof. A quaternary quadratic form $B$ satisfying these congruence conditions modulo $a$ is said to be {\it special at $a$}. The condition of being special at a squarefree integer $a$ is beautifully explicit: namely, $B$ must be of rank $\leq 1$ modulo $a$ (i.e., $B$ must be of rank $\leq 1$ modulo $p$ for every prime $p$ dividing $a$). On the contrary, the specialness condition at prime powers is significantly more involved.

Our task is now to determine, for each nonzero $a\asymp X$, an asymptotic for the number of $\PSO_\cA(\Z)$-orbits on pairs $(\cA,B)$, where the ``height'' of $B$ is bounded by $X$ and $B$ is special at $a$. We must then sum this asymptotic over $a$. Heuristically, the total sum is
\begin{equation*}
\asymp\sum_{a\asymp X}a^{-6} X^{10}\asymp X^5.
\end{equation*}
This is exactly what we expect, and if made rigorous, could be used to show that the second moment of the size of the $2$-Selmer groups of elliptic curves is finite.

However, there are several difficulties in carrying out this strategy. The first is that we need to keep precise control over the error dependence on both $X$ and $a$. When $a$ is around the size of $X$, this requires some additional methods going beyond the use of Davenport's lemma. We achieve  this for squarefree $a$ by proving that special points are equidistributed modulo $a$; this is carried out via  twisted Poisson summation together with a computation of the Fourier coefficients of the set of points special at $a$. This method proves the correct asymptotic along with a sufficiently small error term that is explicit in both $X$ and $a$.

Other difficulties arise in the cases not covered above. First, and less seriously, equidistribution methods are not as readily available when $a$ is a multiple of $p^2$ for some prime $p$, since in this case we do not have sufficiently explicit descriptions of special elements. While this difficulty can be circumvented using the scale invariance of specialness, we choose to resolve it in a way that simultaneously handles the next difficulty.

The second difficulty, which is far more serious, is the following: the ``height'' function by which we need to count the elements $B$ is not a natural one (hence the scare quotes around the word ``height''). Essentially, we must take the naive height on elliptic curves, which yields a $\PGL_2(\Z)$-invariant height on the space $V(\R)$ of binary quartic forms. We must then restrict this height, appropriately scaled, to the space $V_1(\R)$ of monic binary quartic forms. Finally, we must lift the height from $V_1(\R)$ to the space $W_\cA(\R)$ of pairs $(\cA,B)$, where $B$ has real coefficients.

To see why this produces an unnatural height on $W_\cA(\R)$, consider the situation when $a=1$. In this case, elements in a fundamental set for the action of $\on{PGL}_2(\R)$ on $V(\R)$ with leading coefficient $1$ and height less than~$X$ consist of polynomials $f(x)=x^4+bx^3+cx^2+dx+e$, the coefficients of which satisfy the bounds
\begin{equation*}
b\ll X, \quad c\ll X, \quad d\ll X, \quad e\ll X.
\end{equation*}
When such a region is lifted to $W_\cA(\R)$, we obtain a set of elements $B$, the anti-traces of which can be as large as their determinants. Such a region is necessarily quite skewed, and neither Davenport's lemma nor our Fourier analytic methods suffice to count integer points within it while obtaining a sufficiently tight error term. It is not necessary to lower $a$ all the way down to $1$: even when $a$ is around $X^{1-\kappa}$, with $\kappa>1/12$, the associated region in $W_\cA(\R)$ is far too skewed, and we are unable to efficiently count integer points in it.

We resolve this problem as follows. Let $a$ be an integer that is small compared to $X$, say $a\asymp X^{1-\kappa}$ for some $\kappa>0$.
Then the set of elements in $V_1(\R)$ with ``height'' less than $X$ consists of polynomials $f(x)=x^4+bx^3+cx^2+dx+e$, where the coefficients of $f$ satisfy the bounds
\begin{equation*}
b\ll X, \quad
c\ll X^{2-\kappa}, \quad d\ll X^{3-2\kappa}, \quad e\ll X^{4-3\kappa}.
\end{equation*}
These bounds, as explained before, lead to skewed regions upon lifting to $W_\cA(\R)$. Instead, we consider the region inside $V(\R)$ consisting of elements $f(x,y)=a'x^4+bx^3y+cx^2y^2+dxy^3+ey^4$, where the coefficients satisfy the bounds
\begin{equation}\label{equnskewed}
a'\ll X^\kappa, \quad
b\ll X, \quad
c\ll X^{2-\kappa}, \quad d\ll X^{3-2\kappa}, \quad e\ll X^{4-3\kappa}.
\end{equation}
Our goal is now to bound $\SL_4(\Z)$-orbits on pairs $(A,B)\in W(\Z)$ such that: the form $\det(xA+yB)$ satisfies the coefficient bounds~\eqref{equnskewed}; $\det A =1$; and $B$ is special at $a$. This is possible using traditional geometry-of-numbers methods, since the lift of the set of elements satisfying \eqref{equnskewed} to $W(\R)$ is not skewed. Moreover, since we only require upper bounds in this case, it suffices to obtain a minimal saving from the condition $\det A=1$.

We note that this unskewing process, wherein we expand the set of elements in $V_1(\R)$ having bounded height, actually requires the action of $\SL_4$ on $W$; just having an action of $\PSO_\cA$ on $W_\cA$ would not suffice. When the powerful part of $a$ (the product of prime powers dividing $a$ with multiplicity greater than $1$) is large, we similarly use the action of $\SL_4(\Z)$ to obtain upper bounds on the number of $\PSO_\cA(\Z)$-orbits on $W_\cA(\Z)$ with bounded height that are special at $a$. These upper bounds, both when $a$ is small and when the powerful part of $a$ is large, are not optimal. However, we show that because such integers $a$ appear as leading coefficients relatively rarely, their contribution to the second moment is negligible. Consequently, for the purpose of computing the second moment, it is enough to determine asymptotics for the number of $\PSO_\cA(\Z)$-orbits on $W_\cA(\Z)$ of bounded height that are special at $a$ for those leading coefficients $a$ that are both close to $X$ and close to squarefree. As described above, we carry this out  using equidistribution methods.

This results in an upper bound on the second moment of the size of the $2$-Selmer group of elliptic curves in terms of a product of local integrals. We finish by evaluating these integrals using a mass formula.

\section{Algebraic preliminaries}
 \label{sec-alltheconstructions}

 In this section, we recall various algebraic constructions from the literature that feature in our parametrization of $2$-Selmer elements of hyperelliptic Jacobians. We start in \S\ref{sec-reps} by introducing the representations of $\on{SL}_n^{\pm}$ and $\on{SL}_n/\mu_2$ on pairs of $n$-ary quadratic forms whose orbits we used to parametrize Selmer elements. Then, in \S\ref{sec-orth}, we study orbits of $\on{SL}_n/\mu_2$ on pairs of $n$-ary quadratic forms $(A,B)$ for fixed $A$. In \S\ref{sec-ringsbins}, we define the ring $R_f$ associated to a binary $n$-ic form $f$, and in \S\ref{sec-smallconstruct}, we recall a parametrization from~\cite{swacl} of square roots of the ideal class of the inverse different of $R_f$ in terms of orbits of the aforementioned representations. We conclude in \S\ref{sec-maxranks} by describing the image of the parametrization discussed in \S\ref{sec-smallconstruct}.

\subsection{Representations of $\on{SL}_n^{\pm}$ and $\on{SL}_n/\mu_2$ on pairs of $n$-ary quadratic forms} \label{sec-reps}

For an integer $n > 0$, let $\on{Mat}_{n}$ denote the affine scheme over $\Z$ whose $R$-points are
given by the set of $n \times n$ matrices with entries in $R$ for any $\Z$-algebra $R$. Let $U \subset \on{Mat}_n$ denote the subscheme over $\Z$ whose $R$-points are given by $\Sym_2 R^n$ for any $\Z$-algebra $R$ (i.e., $n \times n$ symmetric matrices with entries in $R$, or equivalently, classically integral $n$-ary quadratic forms over $R$), and let $W = U \times U$ \mbox{be the space of pairs of $n$-ary quadratic forms.}

Let $\on{SL}_n^{\pm} \subset \on{GL}_n$ be the subgroup consisting of elements with determinant $\pm 1$. Then $W$ has a natural structure of $\on{SL}_n^{\pm}$-representation (and hence of $\on{SL}_n$-representation) given as follows: for any element $g \in \on{SL}_n^{\pm}(R)$ and any pair $(A,B)
\in W(R)$, let
$$g \cdot (A,B) = \big(g A  g^T, g B g^T\big) \in W(R),$$
where for a matrix $M$, we denote by $M^T$ its transpose.

 Given a pair $(A,B) \in W(R)$, we define the \emph{resolvent} of $(A,B)$ to be the binary $n$-ic form $\det(x A + y B) \in V(R)$. Let $\iota \colon W \to V$ be the function that takes a pair of $n$-ary quadratic forms to its resolvent, and consider the $n+1$ homogeneous polynomial functions $f_0 \circ \iota, \dots, f_n \circ \iota \colon W \to \BA_{\BZ}^1$, each of degree $n$. One readily verifies that these functions are $\on{SL}_n^{\pm}$-invariant; in fact, they freely generate the ring of polynomial invariants for the action of $\on{SL}_n^{\pm}$ on $W$.

Within $\on{SL}_n$, we have the subgroup scheme $\mu_2$, with the inclusion $\mu_2 \hookrightarrow
 \on{SL}_n$ defined on $R$-points by sending $r \in \mu_2(R)$ to $r \times \on{id}
 \in \on{SL}_n(R)$. We now use the action of $\on{SL}_n$ on $W$ to obtain an action
 of $(\on{SL}_n/\mu_2)(R)$ on $W(R)$ in the case where $R$ is a PID. To do this, we must first explicitly characterize the elements of $(\on{SL}_n/\mu_2)(R)$. Upon inspecting the first few terms of the long exact sequence in flat
 cohomology arising from the short exact sequence $1 \to \mu_2 \to \on{SL}_n \to \on{SL}_n/\mu_2 \to 1$, we obtain the following short exact sequence:
 \begin{equation} \label{eq-defquotact}
     \begin{tikzcd}
1 \arrow{r} & \on{SL}_n(R)/\mu_2(R) \arrow{r} & (\on{SL}_n/\mu_2)(R)
\arrow{r} & H^1(R,\mu_2) \arrow{r} & H^1(R,\on{SL}_n).
\end{tikzcd}
 \end{equation}
  We may identify $H^1(R,\mu_2)$ with
 the group $R^\times/R^{\times 2}$ via Kummer theory and $H^1(R,\on{SL}_n)$ with the trivial group via Hilbert's Theorem 90. Let $K_{\on{sep}}$ denote the separable closure of $K$, and for each $r \in R^\times/R^{\times 2}$, let $\varepsilon_r \in \on{SL}_n(K_{\on{sep}})$ be the diagonal matrix with diagonal entries given by $\frac{n}{2}$ copies of $\sqrt{r}$ followed by $\frac{n}{2}$ copies of $\frac{1}{\sqrt{r}}$. We then have the following characterization of the group $(\on{SL}_n/\mu_2)(R)$ in terms of the matrices $\varepsilon_r$:
 \begin{lemma} \label{lem-newdiag}
 Suppose that $R$ is a PID. The group $(\on{SL}_n/\mu_2)(R)$ is the subgroup of $\on{SL}_n(K)/\{\pm \on{id}\}$ consisting of $($classes of$)$ matrices of the form $g \varepsilon_r$, where $g \in \on{SL}_n(R)$ and $r \in R^\times/R^{\times 2}$.
 \end{lemma}
 \begin{proof}
When $R = K$ is a field, the exact sequence~\eqref{eq-defquotact} may be viewed as a sequence in Galois cohomology, in which case the lemma is proven in~\cite[\S4.4]{MR3495795}. For a general PID $R$, applying~\eqref{eq-defquotact} to both $R$ and its field of fractions $K$ yields the following commutative diagram with exact rows, in which all vertical maps are injective:
 \begin{equation} \label{eq-defquotact2}
     \begin{tikzcd}
1 \arrow{r} & \on{SL}_n(R)/\mu_2(R) \arrow{r} \arrow[hook]{d} & (\on{SL}_n/\mu_2)(R)
\arrow{r} \arrow[hook]{d} & R^\times/R^{\times 2} \arrow{r} \arrow[hook]{d} & 1 \\
1 \arrow{r} & \on{SL}_n(K)/\mu_2(K) \arrow{r} & (\on{SL}_n/\mu_2)(K)
\arrow{r} & K^\times/K^{\times 2} \arrow{r} & 1.
\end{tikzcd}
 \end{equation}
 Because the lemma holds for $K$, it follows from~\eqref{eq-defquotact2} that $(\on{SL}_n/\mu_2)(R)$ is contained in the subgroup of $\on{SL}_n(K)/\{\pm \on{id}\}$ consisting of $($classes of$)$ matrices of the form  $g  \varepsilon_r$, where $g \in \on{SL}_n(K)$ and $r \in R^\times/R^{\times 2}$. Now, choose an element $g \varepsilon_r \in (\on{SL}_n/\mu_2)(R)$ of this form. Since $\on{SL}_n(R)/\mu_2(R)$ is a normal subgroup of $(\on{SL}_n/\mu_2)(R)$, the element $g  \varepsilon_r$ normalizes $\on{SL}_n(R)$. But one easily checks that $\varepsilon_r$ normalizes $\on{SL}_n(R)$, so $g$ must normalize $\on{SL}_n(R)$ too. Because $\on{SL}_n(R)$ is its own normalizer in $\on{SL}_n(K)$, we conclude that $g \in \on{SL}_n(R)$, as desired.
 \end{proof}

 We are now ready to define the action of $(\on{SL}_n/\mu_2)(R)$ on $W(R)$. Given (the class of) an element $g \varepsilon_r \in (\on{SL}_n/\mu_2)(R)$, where $g \in \on{SL}_n(R)$ and $r \in R^\times/R^{\times 2}$, we let
 $$(g \varepsilon_r) \cdot (A,B) = \big((g \varepsilon_r) A (g  \varepsilon_r)^T, (g  \varepsilon_r)  B  (g \varepsilon_r)^T\big) \in W(R)$$
 for any pair $(A,B) \in W(R)$. This action is well-defined because the subgroup $\mu_2(R) \subset \on{SL}_n(R)$ acts trivially on $W(R)$, and it also preserves the resolvent of $(A,B)$.

 \subsection{Representation of $\on{PSO}_A$ on $n$-ary quadratic forms} \label{sec-orth}

Let $U_1 \subset U$ be the subscheme consisting of $n$-ary quadratic forms having determinant $1$. Let $\mc{A} \in U_1(\Z)$ denote the antidiagonal matrix with all antidiagonal entries equal to $1$. As claimed in Theorem~\ref{thm-2selparam}, the $(\on{SL}_n/\mu_2)(R)$-orbit on $W(R)$ arising from a $2$-Selmer element of a hyperelliptic Jacobian over $R$ has a representative of the form $(\mc{A},B)$, where $R = \Z$, $\Q$, or a completion thereof. It is thus sometimes convenient to fix an $n$-ary quadratic form $A \in U_1(\Z)$ and to consider the orbits of $(\on{SL}_n/\mu_2)(R)$ on pairs of the form $(A,B) \in W(R)$.

Fix $A \in U_1(\Z)$. Let $\on{PSO}_A = \on{SO}_A/\mu_2$ be the projective orthogonal group associated to the form $A$, and let $W_A$ denote the space of self-adjoint operators for $A$. Observe that the map $B \mapsto -A^{-1}B$ defines an isomorphism of $U \overset{\sim}\longrightarrow W_A$ of affine spaces over $\Z$; as shorthand, we often write $B \in W_A(R)$ to denote the self-adjoint operator $-A^{-1}B \in W_A(R)$. Then for any PID $R$, the space $W_A(R)$ has a natural structure of $\on{PSO}_A(R)$-representation, given as follows: for any element $g \in \on{PSO}_A(R)$, regard $g$ as an element of $(\on{SL}_n/\mu_2)(R)$ via the inclusion $\on{PSO}_A(R) \hookrightarrow (\on{SL}_n/\mu_2)(R)$ induced by the inclusion $\on{SO}_A \hookrightarrow \on{SL}_n$, and let
$$g \cdot B = g B g^T \in W_A(R)$$
for any $B \in W_A(R)$. We then have the following result, the proof of which is immediate:
\begin{proposition} \label{prop-connectiontopso}
Let notation be as above. The map $(A,B) \mapsto B$ defines a natural bijection from the set of $(\on{SL}_n/\mu_2)(R)$-orbits on pairs of the form $(A,B) \in W(R)$ with resolvent $f \in V(R)$ to the set of $\on{PSO}_A(R)$-orbits of forms $B \in W_A(R)$ with $\det(x A + y B) = f(x,y)$. Under this bijection, we have that $\on{Stab}_{(\on{SL}_n/\mu_2)(R)}(A,B) = \on{Stab}_{\on{PSO}_A(R)}(B)$ for any $B \in W_A(R)$.
\end{proposition}

\noindent By abuse of notation, we write $\iota \colon W_A \to V$ for the function that sends an $n$-ary quadratic form to its resolvent.

\subsection{Rings and ideals associated to binary $n$-ic forms} \label{sec-ringsbins}

Let $R$ be a principal ideal domain (PID) with field of fractions $K$. Let $f_0 \in R \smallsetminus \{0\}$, and fix a binary $n$-ic form $f(x,y) = \sum_{i = 0}^n f_ix^{n-i}y^i \in V(R)$ that is separable over $K$. Consider the \'{e}tale $K$-algebra $$K_f \defeq K[x]/(f(x,1)),$$
and let $\theta$ denote the image of $x$ in $K_f$. For each $i \in \{1, \dots, n-1\}$, let
$$ \zeta_i
\defeq p_i(\theta) \in K_f, \quad \text{where} \quad p_i(t) \defeq \sum_{j = 0}^{i-1} f_jt^{i-j} \in R[t].$$
To the form $f$, we associate the free $R$-submodule $R_f \subset K_f$ of rank $n$ and with $R$-basis given by
\begin{equation} \label{eq-rfbasis}
R_f \defeq R \langle 1, \zeta_1, \zeta_2, \dots, \zeta_{n-1} \rangle.
\end{equation}
Much is known about the structure of the module $R_f$. For instance, in~\cite[Proof of
  Lemma~3]{MR0306119}), Birch and Merriman proved that the discriminant
of the \emph{form} $f$ is equal to that of the \emph{module} $R_f$ with respect to the basis~\eqref{eq-rfbasis}; furthermore, in~\cite[Proposition
  1.1]{MR1001839} (cf.~\cite[\S2.1]{MR2763952}), Nakagawa proved that $R_f$ is closed under multiplication (as a subset of $K_f$) and is thus a ring.

Also contained in $K_f$ is a natural family of free $R$-submodules
$I_f^k$ of rank $n$ for each integer $i \in \{0, 1,\dots,n-1\}$, having
$R$-basis given by
\begin{equation} \label{eq-idealdef}
I_f^k \defeq R\langle1,\theta, \dots, \theta^k,\zeta_{k+1}, \dots,
\zeta_{n-1} \rangle.
\end{equation}
For concreteness, we note that $I_f^0 = R_f$ is the unit ideal and that $I_f^{n-1} = R[\theta]$. By~\cite[Proposition A.1]{MR2763952}, each $I_f^k$ is an $R_f$-module and hence a fractional ideal of $R_f$; moreover, the notation $I_f^k$ makes sense, because $I_f^k$ is equal to the $k^{\text{th}}$ power of $I_f^1$ for each $k$. By~\cite[Corollary~2.5 and Prop.~A.4]{MR2763952}, the ideals $I_f^k$ are
invertible precisely when the ring $R_f$ is Gorenstein, which happens precisely when the form $f$ is primitive (i.e., $\gcd(f_0, \dots, f_n) = 1$). By~\cite[Prop.~14]{MR2523319}, the ideal $I_f^{n-2}$ is an explicit representative of the ideal class of the inverse different of the ring $R_f$ and plays a central role in the orbit
parametrization that we introduce in
\S\ref{sec-bigconstruct}.

Given a fractional ideal $I$ of $R_f$ having a chosen
basis --- i.e., a ``based'' fractional ideal --- the
\emph{norm} of $I$, denoted by $\on{N}(I)$, is defined to be the
determinant of the $K$-linear transformation taking the basis of $I$
to the basis of $R_f$ in~\eqref{eq-rfbasis}. It is easy to check that
$\on{N}(I_f^k) = f_0^{-k}$ for each $k \in \{0,1,\dots,n-1\}$ with respect to the basis of $I_f^k$
in~\eqref{eq-idealdef}. The norm of $\kappa \in K_f^\times$ is the
determinant of the $K$-linear transformation taking the basis $\langle
1, \zeta_1, \dots, \zeta_{n-1}\rangle$ to the basis $\langle \kappa,
\kappa  \zeta_1, \dots, \kappa  \zeta_{n-1}\rangle$. Observe that we have the multiplicativity relation
\begin{equation} \label{eq-normsmult}
\on{N}(\kappa I) = \on{N}(\kappa) \on{N}(I)
\end{equation}
for any $\kappa \in K_f^\times$ and based fractional ideal $I$ of $R_f$, where we regard $\kappa \cdot I$ as having the basis given by
multiplying each basis element of $I$ by $\kappa$.

By taking $\wt{\gamma}$ in~\eqref{eq-pgl2acts} to be an element of $\on{SL}_2(R)$, we obtain an action of $\on{SL}_2$ on $V$. We now briefly discuss how the objects $K_f$, $R_f$, and $I_f^k$
transform under this action of $\on{SL}_2$. If $f' = \gamma \cdot f$ for some $\gamma \in \on{SL}_2(R)$, then $K_{f'}$ and
$K_{f}$ are naturally isomorphic, and the rings $R_{f'}$ and $R_f$ are identified under this isomorphism (see~\cite[Proposition 1.2]{MR1001839} for a direct proof and~\cite[\S2.3]{MR2763952} for a geometric argument). On the other hand, the ideals $I_f^k$ and $I_{f'}^k$ are isomorphic as $R_f$-modules but may \emph{not necessarily} be identified under the isomorphism $K_{f'} \simeq K_{f}$. Indeed, as explained in~\cite[(7)]{thesource}, these ideals are related as follows: if $\gamma = \left[\begin{smallmatrix} s & t \\ u &
    v\end{smallmatrix}\right]$ is such that $-t\theta +s \neq 0$ (equivalently, $f'$ has nonzero leading coefficient), then
for each $k \in \{0,1,\dots,n-1\}$, the composition
\begin{equation} \label{eq-idealident}
\begin{tikzcd}
I_f^k \arrow[hook]{r}{\phi_{k,\gamma}} & K_f \arrow{r}{\sim} & K_{f'}
\end{tikzcd}
\end{equation}
is an injective map of $R_f$-modules having image equal to $I_{f'}^{k}$, where $\phi_{k,\gamma}$ sends each $\delta \in I_f^k$ to $(-t\theta+s)^{-k} \delta \in K_f$. When $k = 0$, we recover the identification $R_f = I_f^0$ with $R_{f'} = I_{f'}^0$ from~\eqref{eq-idealident}.

\subsection{Parametrization of square roots of the ideal class of $I_f^{n-2}$} \label{sec-smallconstruct}

Fix a PID $R$ and a binary $n$-ic form \mbox{$f(x,y) = \sum_{i = 0}^n f_ix^{n-i}y^i \in V(R)$} with $f_0\in R \smallsetminus \{0\}$ that is separable over $K$ for the rest of this section.

Up to some factors of $6$ (see~\S\ref{sec-smallprimeconstruct} for details), we show in \S\ref{sec-2selhypjac} that when $n \equiv 2 \pmod 4$ or when $f$ is locally soluble, the $2$-Selmer
elements of $\on{Jac}(C_f)$ naturally give rise to square roots of the ideal class of the fractional ideal $I_f^{n-2}$ (where $R = \Z$, $\Q$, or a completion thereof). In this section, we recall
a parametrization introduced in~\cite[\S2.2]{swacl} that takes such square roots and produces pairs $(A,B) \in W(R)$ with resolvent $f^{\on{mon}}$,
where $f^{\operatorname{mon}}$ denotes the \emph{monicization} of $f$ and is
defined by
$$f^{\operatorname{mon}}(x,y) \defeq f_0^{-1} \times f(x, f_0 y).$$
Note that $f^{\on{mon}}$ is a monic binary $n$-ic form and is an element of $V(R)$, for its coefficients are integral. Given $a \in \Z$, let $V_a \subset V$ to be the subscheme consisting of binary $n$-ic forms having leading coefficient $a$. The map $\on{mon} \colon V_{f_0} \to V_1$ that sends a binary $n$-ic form to its monicization is injective; we call the left-inverse to this map \emph{demonicization at $f_0$}, and we denote the demonicization of a form $g$ at $f_0$ by $g^{\on{dem}}_{f_0}$.

We now describe the parametrization. Let $I$ be a based fractional ideal of $R_f$, and suppose that there
exists $\alpha \in K_f^\times$ such that
\begin{equation} \label{eq-newconds}
I^2 \subset \alpha I_f^{n-2} \quad \text{and} \quad \on{N}(I)^2
= \on{N}(\alpha) f_0^{2-n}.
\end{equation}
Consider the symmetric bilinear form
\begin{equation} \label{eq-defbilin}
\langle-,-\rangle \colon I \times I \to K_f, \quad (\beta, \gamma)
\mapsto \langle \beta, \gamma \rangle = \alpha^{-1} \beta\gamma.
\end{equation}
By assumption, $\langle -, - \rangle$ has image contained in $I_f^{n-2}$. Let $\pi_{n-2},\pi_{n-1} \in \Hom_{R}(I_f^{n-2},R)$ be
the maps defined on the $R$-basis~\eqref{eq-idealdef} of $I_f^{n-2}$ by
\begin{align*}
& \pi_{n-2}(\theta^{n-2}) - 1 = \pi_{n-2}(\zeta_{n-1}) = \pi_{n-2}(\theta^i) = 0 \quad\text{for each } i \in \{0, \dots, n-3\} , \text{ and}\\
& \pi_{n-1}(\zeta_{n-1}) + 1 = \pi_{n-1}(\theta^i) = 0 \quad\text{for each } i \in \{0, \dots, n-2\}.
\end{align*}
Let $A$ and $B$ be the Gram matrices of the symmetric bilinear forms \mbox{$\pi_{n-1} \circ \langle -, - \rangle  \colon I \times I \to R$} and $\pi_{n-2} \circ \langle -, - \rangle
\colon I \times I \to R$, respectively, with respect to the chosen basis of
$I$. Observe that $A$ and $B$ are symmetric $n \times n$ matrices with
entries in $R$, and so the pair $(A,B)$ is an element of $W(R)$.

The group $\on{SL}_n^{\pm}(R)$ acts naturally on pairs $(I,\alpha)$ of the form~\eqref{eq-newconds} by changing the basis of the based fractional ideal $I$. While the pair $(A,B) \in W(R)$ depends on the choice of basis of $I$,
the $\on{SL}_n^{\pm}(R)$-orbit of $(A,B)$ does not depend on this
choice. We have thus constructed an $\on{SL}_n^{\pm}(R)$-orbit on
$W(R)$ from a pair $(I,\alpha)$ of the form~\eqref{eq-newconds}. The following key parametrization result gives a complete characterization of the orbits arising via this construction:
\begin{theorem}[\protect{\cite[Theorems~10 and~14, and Propositions~13 and~15]{swacl}}] \label{thm-theconstruction}
The pair $(A,B) \in W(R)$ constructed above satisfies the following
two $\on{SL}_n^{\pm}(R)$-invariant properties:
\begin{enumerate}
\item[$(\mathrm{a})$] $\det(x A + y B) =
  f^{\operatorname{mon}}(x,y)$; and
\item[$(\mathrm{b})$] $p_i\big(f_0^{-1} \times -A^{-1}B\big)
  \in\on{Mat}_n(R)$ for each $i \in \{1, \dots,n-1\}$.
\end{enumerate}
Moreover, the above construction defines a natural bijection between
$\on{SL}_n^{\pm}(R)$-orbits on $W(R)$ satisfying the conditions $(\mathrm{a})$ and $(\mathrm{b})$ with equivalence classes of pairs $(I,\alpha)$ of the
form~\eqref{eq-newconds}, where we declare two pairs $(I,\alpha)$ and
$(I',\alpha')$ equivalent if there exists $\kappa \in K_f^\times$ such
that the based fractional ideals $I'$ and $\kappa I$ are equal up to an $\on{SL}_n^{\pm}(R)$-change-of-basis and $\alpha' = \kappa^2 \alpha$. The stabilizer in $\on{SL}_n^{\pm}(R)$ of the orbit corresponding to $(I,\alpha)$ is isomorphic to $\on{End}_{R_f}(I)^\times[2]$, where
$\on{End}_{R_f}(I)$ denotes the ring of $R_f$-module endomorphisms of $I$.
\end{theorem}
\begin{remark}
When $f$ is separable, every $R_f$-module endomorphism of a based fractional ideal $I$ is
given by multiplication by an element of $K_f$. Thus, we may view the
stabilizer $\on{End}_{R_f}(I)^\times[2]$ of the $\on{SL}_n^{\pm}(R)$-orbit corresponding to a pair $(I,\alpha)$ via Theorem~\ref{thm-theconstruction} as a subgroup of
$K_f^\times[2]$.
\end{remark}

We say that (the $\on{SL}_n^{\pm}(R)$-orbit of) a pair $(A,B) \in W(R)$ \emph{arises for $f$} if it arises from a pair $(I,\alpha)$ of the form~\eqref{eq-newconds} via
Theorem~\ref{thm-theconstruction}, or equivalently, if it satisfies conditions (a) and (b) in the theorem. If $f_0 \not\in R^\times$, then it is not necessarily the case that every pair $(A,B) \in W(R)$ with resolvent $f^{\on{mon}}$ arises for $f$. Indeed, conditions (a) and (b) in Theorem~\ref{thm-theconstruction} imply that the set of pairs in $W(R)$ that arise for $f$ is defined by congruence conditions modulo $f_0^{n-1}$; for a detailed analysis of these conditions when $n = 4$, see \S\ref{sec-interpret}.

In \S\ref{sec-idealconstruct}--\ref{sec-smallprimeconstruct}, we often find it convenient to replace the binary $n$-ic form $f$ with an $\on{SL}_2(R)$-translate thereof. We must therefore understand how the parametrization in Theorem~\ref{thm-theconstruction} interacts with the action of $\on{SL}_2$ on $V$. The following lemma establishes that the parametrization is covariant with respect to this action:

\begin{lemma} \label{lem-sl2transform}
Let $\gamma = \left(\begin{smallmatrix} s & t \\ u &
    v \end{smallmatrix}\right) \in \on{SL}_2(R)$ be such
that $-t\theta+s \neq 0$, and let $f' = \gamma \cdot f$. Then $\gamma$
induces a bijection between the
$\on{SL}_n^{\pm}(R)$-orbits of pairs $(A,B) \in W(R)$ with resolvent $f^{\on{mon}}$ that arise for $f$ and the $\on{SL}_n^{\pm}(R)$-orbits of pairs $(A,B) \in W(R)$ with resolvent ${f'}^{\on{mon}}$ that arise for $f'$.
\end{lemma}
\begin{proof}
Let $(A,B) \in W(R)$ with resolvent $f^{\on{mon}}$, and suppose that $(A,B)$ arises
from a pair $(I,\alpha)$ of the form~\eqref{eq-newconds}. Then it
follows from~\eqref{eq-idealident} that the pair $\big(\phi_{\frac{n}{2},\gamma}(I), \phi_{2,\gamma}(\alpha)\big)$ is of the
form~\eqref{eq-newconds}, albeit with $f$ replaced by $f'$. The
desired bijection is given by $(I,\alpha) \mapsto \big(\phi_{\frac{n}{2},\gamma}(I), \phi_{2,\gamma}(\alpha)\big)$.
\end{proof}

When $R = K$ is a field, the parametrization in
Theorem~\ref{thm-theconstruction} can be stated in simpler terms. Indeed, over
$K$, the data of a based fractional ideal of $R_f = K_f$ is encoded in
its basis, so the condition that $I^2 \subset \alpha I_f^{n-2}$ is a tautology in this setting. Since the parametrization is only concerned with based fractional
ideals up to change-of-basis by an element of $\on{SL}_n^{\pm}(K)$, we
may replace every instance of the based fractional ideal $I$ with $|\on{N}(I)|$. But the value of $\alpha \in K_f^\times$ determines $|\on{N}(I)|$ via the condition that $\on{N}(I)^2 = \on{N}(\alpha) f_0^{2-n}$, and so we may forget about $I$ altogether. Moreover, condition (b) in Theorem~\ref{thm-theconstruction} holds trivially because $f_0$ is a unit in $K$. We thus obtain the following immediate consequence of Theorem~\ref{thm-theconstruction}:

\begin{corollary}[\protect{\cite[Prop.~26]{swacl}}] \label{cor-fieldparamsl}
 The $\on{SL}_n^{\pm}(K)$-orbits on $W(K)$ with resolvent $f^{\on{mon}}$ are in natural bijection with the elements of the set
 $$(K_f^\times/K_f^{\times2})_{\on{N}\equiv 1} \defeq \{\alpha
 \in K_f^\times/K_f^{\times 2} : \on{N}(\alpha) \equiv 1 \in
 K^\times/K^{\times2}\}.$$
 The stabilizer in $\on{SL}_n^{\pm}(K)$ of any such orbit is isomorphic to $K_f^\times[2]$.
\end{corollary}

We say that (the $\on{SL}_n^{\pm}(K)$-orbit of) a pair $(A,B) \in W(K)$ is \emph{distinguished} if $A$ has a maximal isotropic space defined over $K$ containing an $(\tfrac{n}{2}-1)$-dimensional isotropic space for $B$. Under the bijection in Corollary~\ref{cor-fieldparamsl}, the distinguished orbits correspond to the subset $K^\times \subset (K_f^\times/K_f^{\times2})_{\on{N}\equiv 1}$.

 We now translate Theorem~\ref{thm-theconstruction} and Corollary~\ref{cor-fieldparamsl} into statements about $(\on{SL}_n/\mu_2)(R)$-orbits on $W(R)$. Conditions (a) and (b) in Theorem~\ref{thm-theconstruction} are evidently invariant under the action of $(\on{SL}_n/\mu_2)(R)$; we say that the $(\on{SL}_n/\mu_2)(R)$-orbit of a pair $(A,B) \in W(R)$ \emph{arises for $f$} if $(A,B)$ does. By imitating the argument used to deduce~\cite[Theorem~17]{MR3600041} from~\cite[Theorem~16]{MR3600041}, we obtain the following parametrization of $(\on{SL}_n/\mu_2)(R)$-orbits on $W(R)$ from Theorem~\ref{thm-theconstruction}:
 \begin{theorem}\label{thm-slnparam}
The $(\on{SL}_n/\mu_2)(R)$-orbits on $W(R)$
that arise for $f$ are in natural bijection with equivalence classes of pairs $(I,\alpha)$ of the form~\eqref{eq-newconds}, where
we declare two pairs $(I,\alpha)$ and $(I',\alpha')$ to be equivalent if
there exists $\kappa \in K_f^\times$ and $r \in R^\times$ such that the based fractional ideals $I'$ and $\kappa I$ are equal up to
a $\on{GL}_n(R)$-change-of-basis, $\on{N}(I') = \on{N}(\kappa) r^2 \on{N}(I)$, and $\alpha' = \kappa^2  r \alpha$. The stabilizer in $(\on{SL}_n/\mu_2)(R)$ of the orbit
corresponding to the pair $(I,\alpha)$ is an elementary abelian
$2$-group that contains the group
$(\on{End}_{R_f}(I)^\times[2])_{\on{N}\equiv1}/R^\times[2]$.
\end{theorem}

Next, by imitating the argument used to deduce~\cite[Corollary~20]{MR3600041} from (the proof of)~\cite[Theorem~17]{MR3600041}, we obtain the
following corollary of (the proof of) Theorem~\ref{thm-slnparam} parametrizing the orbits of $(\on{SL}_n/\mu_2)(K)$ on $W(K)$:
\begin{corollary} \label{cor-fieldparam}
 The $(\on{SL}_n/\mu_2)(K)$-orbits on $W(K)$ with resolvent $f^{\on{mon}}$
 are in natural bijection with equivalence classes of pairs
 $(\varepsilon, \alpha)$ consisting of numbers $\varepsilon \in
 K^\times$ and $\alpha \in K_f^\times$ such that $\varepsilon^2 =
 \on{N}(\alpha) f_0^{2-n}$, where we declare two such pairs
 $(\varepsilon,\alpha)$ and $(\varepsilon',\alpha')$ equivalent if
 there exist numbers $\kappa \in K_f^\times$ and $r \in K^\times$ such
 that $\varepsilon' = \on{N}(\kappa) r^2
 \varepsilon$ and $\alpha' = \kappa^2 r \alpha$. The
 stabilizer subscheme in $\on{SL}_n/\mu_2$, viewed as a $K$-scheme, of
 the orbit corresponding to $(\varepsilon,\alpha)$ is
 isomorphic to $(\on{Res}_{K_f/K} \mu_2)_{\on{N}\equiv1}/\mu_2$.

 Moreover, the map from the set of $(\on{SL}_n/\mu_2)(K)$-orbits on $W(K)$ with resolvent $f^{\on{mon}}$ to $(K_f^\times/K_f^{\times 2}K^\times)_{\on{N}\equiv 1}$
 defined by sending $(\varepsilon,\alpha) \mapsto \alpha$ is an at
 most two-to-one surjection.
\end{corollary}

In \S\ref{sec-alltheconstructions}, we often find it convenient to reexpress the
parametrization in Corollary~\ref{cor-fieldparam} in the language of
Galois cohomology. For this purpose, we assume that there exists a
pair \mbox{$(A,B_0) \in W(K)$} with resolvent $f^{\on{mon}}$. Applying
Corollary~\ref{cor-fieldparam} to the separable closure $K_{\ksep}$ of $K$, we see that there is a
unique $(\on{SL}_n/\mu_2)(K_{\ksep})$-orbit on $W(K_{\ksep})$ with resolvent $f^{\on{mon}}$. Let
$(A_f,B_f) \in W(K_{\ksep})$ be a representative of this unique
$(\on{SL}_n/\mu_2)(K_{\ksep})$-orbit. Then there exists $g \in
(\on{SL}_n/\mu_2)(K_{\ksep})$ such that $(A,B_0) = g \cdot (A_f,
B_f)$. For any $\sigma \in \on{Gal}(K_{\ksep}/K)$, we have that
$(A,B_0) = {^\sigma} g \cdot (A_f, B_f)$, so $g^{-1}{^\sigma g} \in
\on{Stab}_{(\on{SL}_n/\mu_2)(K_{\ksep})}(A_f,B_f)$, and the map $\sigma
\mapsto ^{-1}{^\sigma g}$ defines a $1$-cocycle on
$\on{Gal}(K_{\ksep}/K)$ with values in
$\on{Stab}_{(\on{SL}_n/\mu_2)(K_{\ksep})}(A_f,B_f) \simeq
\big((\on{Res}_{K_f/K} \mu_2)_{\on{N}\equiv1}/\mu_2\big)(K_{\ksep})$,
the cohomology class of which depends only on the
$(\on{SL}_n/\mu_2)(K)$-orbit of $(A,B)$ and has trivial image under the
map $\delta_K \colon H^1(K,(\on{Res}_{K_f/K}
\mu_2)_{\on{N}\equiv1}/\mu_2)\to H^1(K, \on{SL}_n/\mu_2)$. This
construction defines a bijection
\begin{equation} \label{eq-fieldcohparam}
\begin{tikzcd}
\left\{\begin{array}{c} \text{($\on{SL}_n/\mu_2)(K)$-orbits on $W(K)$ with resolvent $f^{\on{mon}}$} \end{array}\right\} \arrow[leftrightarrow]{r} & \ker
\delta_K.
\end{tikzcd}
\end{equation}
From the long exact sequence in cohomology of the short exact sequence
\begin{equation*}
\begin{tikzcd}
1 \arrow{r} & (\on{Res}_{K_f/K} \mu_2)_{\on{N}\equiv1}/\mu_2 \arrow{r}
& \on{Res}_{K_f/K} \mu_2/\mu_2 \arrow{r}{\on{N}} & \mu_2 \arrow{r} & 1
\end{tikzcd}
\end{equation*}
we obtain the following surjective map:
\begin{equation} \label{eq-surje}
\begin{tikzcd}
\phi \colon H^1(K, (\on{Res}_{K_f/K} \mu_2)_{\on{N}\equiv1}/\mu_2)
\arrow[two heads]{r} & \ker\big(\on{N}\colon H^1(K,\on{Res}_{K_f/K}
\mu_2/\mu_2) \to H^1(K,\mu_2)\big)
\end{tikzcd}
\end{equation}
As shown in~\cite[\S4.4]{MR3495795}, the subset
$\phi^{-1}((K_f^\times/K_f^{\times 2}K^\times)_{\on{N}\equiv1})
\subset H^1(K,(\on{Res}_{K_f/K} \mu_2)_{\on{N}\equiv1}/\mu_2)$ lies in
$\ker \delta_K$, so to prove that a class in $H^1(K, (\on{Res}_{K_f/K}
\mu_2)_{\on{N}\equiv1}/\mu_2)$ corresponds to an
$(\on{SL}_n/\mu_2)(K)$-orbit on $W(K)$ with resolvent $f^{\on{mon}}$ under the
bijection~\eqref{eq-fieldcohparam}, it suffices to show that its image
under $\phi$ lies in $(K_f^\times/K_f^{\times
  2}K^\times)_{\on{N}\equiv1}$.

  The parametrization in Theorem~\ref{thm-slnparam} takes on a similarly simple form when $R = \Z_p$ for a prime $p$ and $R_f$ is the maximal order in its algebra of fractions $K_f$. Indeed, by imitating the proof of~\cite[Theorem~21]{swacl}, we deduce the following consequence of Theorem~\ref{thm-slnparam}:
\begin{corollary}\label{cor-zparam}
 Let $R = \Z_p$ for a prime $p$, and suppose that $R_f$ is the maximal order in $K_f$. Then the $(\on{SL}_n/\mu_2)(\BZ_p)$-orbits on
 $W(\BZ_p)$ that arise for $f$ are in natural bijection with equivalence classes of pairs $(\varepsilon, \alpha)$ consisting of numbers
 $\varepsilon \in \BZ_p^\times$ and $\alpha \in K_f^\times$ such that
 $\varepsilon^2 = \on{N}(\alpha) f_0^{2-n}$, where we declare two such pairs $(\varepsilon,\alpha)$ and $(\varepsilon',\alpha')$
 equivalent if there exist numbers $\kappa \in R_f^\times$ and $r \in
 \BZ_p^\times$ such that $\varepsilon' = \on{N}(\kappa) r^2  \varepsilon$ and $\alpha' = \kappa^2 r \alpha$. The stabilizer in $(\on{SL}_n/\mu_2)(\BZ_p)$ of any
 such orbit is an elementary abelian $2$-group that contains the group
 $(R_f^\times[2])_{\on{N}\equiv1}/\{\pm1\}$.

  Moreover, the map from the set of $(\on{SL}_n/\mu_2)(\BZ_p)$-orbits on $W(\BZ_p)$ that arise for $f$ to
  $(K_f^\times/K_f^{\times 2}\BQ_p^\times)_{\on{N}\equiv1}$ defined by sending $(\varepsilon,\alpha) \mapsto \alpha$ is at most two-to-one,
  and its image is a torsor for the group $(R_f^\times/R_f^{\times 2}\BZ_p^\times)_{\on{N}\equiv1}$.
\end{corollary}

We now describe how an $\on{SL}_n^{\pm}(R)$- or $(\on{SL}_n/\mu_2)$-orbit on $W(R)$ arising for $f$ transforms when $f$ is replaced with a scalar multiple of itself by an element of $R \smallsetminus\{0\}$ (not necessarily a unit):

\begin{lemma} \label{lem-canscale}
Suppose that $(A,B) \in W(R)$ arises for $f$, and let $r \in R
\smallsetminus \{0\}$. Then the pair $(A, r  B) \in W(R)$ arises for
$\wt{f} = rf$, and the $\on{SL}_n^{\pm}(K)$-orbit of $(A, r B)$ corresponds via the bijection in Corollary~\ref{cor-fieldparamsl} to the class of $r^{-1} \alpha\in (K_f^\times/K_f^{\times 2})_{\on{N}\equiv1}$.
\end{lemma}
\begin{proof}
One verifies by inspection that the pair $(A, r B) \in W(R)$
satisfies the conditions (a) and (b) in Theorem~\ref{thm-theconstruction},
with $f$ replaced by $\wt{f} \defeq r f$. To compute the
$\on{SL}_n^{\pm}(K)$-orbit of $(A, r B)$, we work over $K$,
where based fractional ideals are just bases. The operation of
replacing the pair $(A,B)$ with $(A, r  B)$ transforms the resolvent
form $\det(x A +y B) = f^{\operatorname{mon}}(x,y)$ into the
new form $\det(x A + y (r B)) = \wt{f}^{\mathrm{mon}}(x,y)$. For each of
$I_f^{n-2}$ and $I_{\wt{f}}^{n-2}$, the basis~\eqref{eq-idealdef}
is given by
\begin{equation}
I_{f}^{n-2} = K \langle 1, \theta, \theta^2,\dots,\theta^{n-2}, \zeta_{n-1}\rangle \quad \text{and} \quad I_{\wt{f}}^{n-2} = K\langle 1, \theta, \theta^2 ,\dots, \theta^{n-2}, r \zeta_{n-1}\rangle. \label{eq-base1}
\end{equation}
Suppose $(A,B)$ corresponds to a pair $(I,\alpha)$ via
Theorem~\ref{thm-theconstruction}. We claim that the pair $(I, r^{-1}
\alpha)$ is of the form~\eqref{eq-newconds}, with $f$ replaced by $\wt{f}$, and that it corresponds to $(A, r B)$ via
Theorem~\ref{thm-theconstruction}. Upon comparing the bases in~\eqref{eq-base1}, it is evident that the bilinear forms $\pi_{n-1} \circ \langle -,-\rangle$ and $\pi_{n-2}\circ \langle -,-\rangle$
arising from the pair $(I, r^{-1} \alpha)$ are represented by the matrices $A$ and $r B$ with respect to the chosen basis of $I$. It thus remains to verify that the pair $(I, r^{-1} \alpha)$ is of the form~\eqref{eq-newconds}. Let $\on{N}$ denote the based fractional ideal norm with respect to the basis~\eqref{eq-rfbasis} of $R_f$, and let $\wt{\on{N}}$ denote this norm with respect to the basis~\eqref{eq-rfbasis} of $R_{\wt{f}}$. Then a calculation reveals that $\wt{\on{N}}(I) = r^{1-n} \on{N}(I)$ for any basis $I$. Note
that this does \emph{not} extend to norms of numbers: it follows from the definition of the norm of a number in \S\ref{sec-ringsbins} that
$\wt{\on{N}}(\alpha) = \on{N}(\alpha)$. It now suffices to show that
$\wt{\on{N}}(I)^2 = \wt{\on{N}}\big(r^{-1} \alpha\big)
(r f_0)^{2-n}$, which is equivalent to showing that $\big(r^{1-n} \on{N}(I)\big)^2 = \big(r^{-n} \on{N}(\alpha)\big) (r f_0)^{2-n}$, but this holds since $\on{N}(I)^2 = \on{N}(\alpha) f_0^{2-n}$.
\end{proof}

\subsection{Preliminary results on the ranks of orbits arising for $f$} \label{sec-maxranks}

Let $f$ be as in \S\ref{sec-smallconstruct}, and let $R = \Z_p$ for a prime $p$. A pair $(A,B) \in W(\Z_p)$ with resolvent $f^{\on{mon}}$ arises for $f$ if certain complicated congruence conditions modulo $f_0^{n-1}$, given by condition (b) in the statement of Theorem~\ref{thm-theconstruction}, are satisfied. Each of these congruence conditions is of the following shape: for some $k \in \{1, \dots, n\}$, a certain linear combination of the $k \times k$ minors of $-A^{-1}B$ must be divisible by $f_0^{k-1}$. For these conditions to hold, it is evidently sufficient that the $n$-ary quadratic form $B$ be of \emph{rank $\leq 1$ modulo $f_0$}, by which we mean that $B$ is a constant multiple of the square of a linear form modulo $f_0$. The objective of this section is to prove the following result, which states that, \emph{most of the time}, the pairs $(A,B) \in W(R)$ with resolvent $f^{\on{mon}}$ that arise for $f$ are \emph{precisely} those pairs such that $B$ is of rank $\leq 1$ modulo $f_0$.

\begin{theorem} \label{thm-orbsmallrank}
Let $f \in \Z_p[x,y]$ be a separable binary form with leading coefficient $f_0$ divisible by $p$, and let $(A,B) \in W(\Z_p)$ be a pair arising for $f$. If $p \nmid f_1$ or if $p$ is sufficiently large and $R_f$ is the maximal order in its algebra of fractions $K_f$, then $B$ is of rank $\leq 1$ modulo $f_0$.
\end{theorem}
\begin{proof}
We split into two cases, according as $p \nmid f_1$ or $p$ is sufficiently large and $R_f$ is the maximal order in $K_f$:

\vspace*{0.2cm}
\noindent \emph{Case 1}: We first show that $B$ is of rank $1$ modulo $p$. Since $\det A = (-1)^{\lfloor \frac{n}{2} \rfloor}$, it suffices to prove that $-A^{-1}B$ is of rank $1$ modulo $p$. First, observe that the characteristic polynomial of $-A^{-1}B$ over $\F_p$ is equal to $f^{\on{mon}}(x,1) \equiv x^n + f_1x^{n-1} \pmod p$. Next, taking $i = 2$ in condition (b) in Theorem~\ref{thm-theconstruction} yields that
\begin{equation} \label{eq-takei=2}
p_2(\tfrac{1}{f_0} \times -A^{-1}B) = \tfrac{1}{f_0}\big((-A^{-1}B)^2+f_1(-A^{-1}B))\big) \in \on{Mat}_n(\Z_p),
\end{equation}
and so the minimal polynomial of $-A^{-1}B$ over $\F_p$ is equal to $x^2 + f_1x$. Comparing the characteristic and minimal polynomials of $-A^{-1}B$, we deduce that the rational canonical form of $-A^{-1}B$ over $\F_p$ must be diagonal with the first $n-1$ diagonal entries equal to $0$ and last diagonal entry nonzero, as desired.

By the previous paragraph along with the $p$-adic Jordan decomposition, we may assume that $B$ is diagonal with the first $n-1$ diagonal entries divisible by $p$ and last diagonal entry a $p$-adic unit. Multiplying~\eqref{eq-takei=2} through by $A^{-1}$, we find that
\begin{equation} \label{eq-secondi3}
    BA^{-1}B - f_1 B \equiv 0 \pmod{f_0}.
\end{equation}
Since $\nu_p(f_1) = 0$ and since $p$ divides each of the first $n-1$ diagonal entries of $B$, the $p$-adic valuation of each of the first $n-1$ diagonal entries of $BA^{-1}B$ is strictly larger than that of the corresponding diagonal entry of $-f_1B$. Thus, in order for~\eqref{eq-secondi3} to be satisfied, the $p$-adic valuation of each of the first $n-1$ diagonal entries of $B$ must be at least as big as $\nu_p(f_0)$, as necessary.

\vspace*{0.2cm}
\noindent \emph{Case 2}: Since $R_f$ is maximal, $f$ is primitive, so it factorizes over $\Z_p$ as $f = f_1 f_2$, where $f_1 \equiv y^e \pmod{p}$ for some $e \in \{1, \dots, n\}$ and $f_2$ has unit leading coefficient. It then follows from~\cite[Lemma~19]{swacl} that $R_f \simeq R_{f_1} \times R_{f_2}$. Thus, any fractional ideal $I$ of $R_f$ may be expressed as a product $I_1 \times I_2$ of fractional ideals $I_i$ of $R_{f_i}$. Moreover, it follows from Dedekind's criterion (see~\cite[Theorem~17 and Lemma~20]{swacl}) that $\min\{\nu_p(f_0),e\} = 1$. Having already handled case 1, we may assume that $\nu_p(f_0) = 1$ and that $e > 1$ (equivalently, that $\nu_p(f_0) \geq 1$).

Let $(I,\alpha)$ be a pair of the form~\eqref{eq-newconds} for $f$. Then expressing $I$ as a product $I_1 \times I_2$, we see that the orbit arising from $(I,\alpha)$ must contain an element $(A,B) \in W(\Z_p)$ such that both $A$ and $B$ are block-diagonal with one $e\times e$ block followed by one $(n-e) \times(n-e)$ block. Let $A_1,B_1$ denote the $e \times e$ blocks of $A,B$, and let $A_2,B_2$ denote their $(n-e) \times (n-e)$ blocks.

The proof of~\cite[Theorem~10]{swacl} shows that $-A^{-1}B$ is the matrix of multiplication by $f_0\theta \in R_f$ on the basis of $I$ obtained by concatenating bases of $I_1$ and $I_2$. It follows that $-A_i^{-1}B_i$ is the matrix of multiplication by $f_0\theta \in R_{f_i}$ on the basis of $I_i$. Consequently, we have that $\det(x A_1 + y B_1) = f_1^{\on{mon}}(x,y)$, since $f_1^{\on{mon}}(x,1)$ is the characteristic polynomial of multiplication by $f_0\theta$ on $K_{f_1}$. The $y^e$-coefficient of $f_1$ is a unit, so $\nu_p(\det B_1) = n-1$. Moreover, taking $i = n-1$ in condition (b) of Theorem~\ref{thm-theconstruction}, we find that $\wedge^{n-1} B_1 \equiv 0 \pmod{p^{n-2}}$ as long as $p$ is sufficiently large (cf.~proof of Theorem~\ref{thm-interpret}, to follow). It then follows from the $p$-adic Jordan decomposition (see~\cite[p.~104--107]{MR1478672}) that $B_1$ is of rank $1$ modulo $p$. That $B_2$ has rank $0$ modulo $p$ follows from the fact that $f_2$ has unit leading coefficient, which implies that $\theta \in R_{f_2}$ and hence that $f_0^{-1} B_2$ has integral entries.
\end{proof}

\section{Parametrization of $2$-Selmer groups of hyperelliptic Jacobians} \label{sec-2selhypjac}

The purpose of this section is to prove Theorem~\ref{thm-2selparam}. We begin in \S\ref{sec-rationalgalois} by constructing rational orbits from $2$-Selmer elements. Crucially, we show in \S\ref{sec-bigconstruct} that the rational orbits arising from $2$-Selmer elements have integral representatives, thus allowing us to use geometry-of-numbers methods to count $2$-Selmer elements.

\subsection{Construction of rational orbits via Galois cohomology} \label{sec-rationalgalois}

To construct the map~\eqref{eq-selmap}, we first work over a local field $K$ of characteristic not equal to $2$. Let $f \in V(K)$ be a separable binary $n$-ic form with nonzero leading coefficient, and assume that $C_f$ is soluble over $K$ if $n \equiv 0 \pmod 4$. We shall now construct an injective map of sets
\begin{equation} \label{eq-j2j}
    \begin{tikzcd}
J_f(K)/2J_f(K) \arrow[hook]{r} & \left\{\begin{array}{c}
\text{($\on{SL}_n/\mu_2)(K)$-orbits of $(A,B) \in W(K)$} \\ \text{such
  that $\det(x A + y B) =
  f^{\operatorname{mon}}(x,y)$} \end{array}\right\}
\end{tikzcd}
\end{equation}
By~\cite[part (2) of Prop. 22]{MR3600041},
we have $J_f[2] \simeq (\on{Res}_{K_f/K}
\mu_2)_{\on{N}\equiv1}/\mu_2 \simeq (\on{Res}_{K_{f^{\operatorname{mon}}}/K}
\mu_2)_{\on{N}\equiv1}/\mu_2$ as group schemes over $K$. Upon combining the
$2$-descent map $J_f(K)/2J_f(K) \hookrightarrow H^1(K, J_f[2])$ with the cohomological description of the
parametrization in Corollary~\ref{cor-fieldparam}, we obtain the following \mbox{commutative diagram:}
\begin{equation} \label{eq-square}
    \begin{tikzcd}
    J_f(K)/2J_f(K) \arrow{rr}{\text{``$x-\theta$''}} \arrow[hook]{d} &
    & (K_f^\times/K_f^{\times 2}K^\times)_{\on{N}\equiv1}
    \arrow[hook]{d} \\ H^1(K,J_f[2]) \arrow[swap]{r}{\sim} & H^1(K,
    (\on{Res}_{K_{f^{\operatorname{mon}}}/K} \mu_2)_{\on{N}\equiv1}/\mu_2)
    \arrow[swap, two heads]{r}{\phi} & \ker \on{N}
\end{tikzcd}
\end{equation}
where the homomorphism ``$x-\theta$'' in the top row
of~\eqref{eq-square} is defined in~\cite[\S5]{MR1465369} and
essentially evaluates the quantity $x - \theta$ on divisor
classes. That the ``$x - \theta$'' map is defined on all of
$J_f(K)/2J_f(K)$ follows from our assumption that $C_f$ is soluble when $n \equiv 0 \pmod 4$ (see~\cite[\S10]{MR1465369}). From~\eqref{eq-square}, we see that the
image of $J_f(K)/2J_f(K)$ in $\ker \on{N}$ lies in the subset
\mbox{$(K_f^\times/K_f^{\times2}K^\times)_{\on{N}\equiv1} \simeq
  (K_{f_\mathsf{mon}}^\times/K_{f_\mathsf{mon}}^{\times2}K^\times)_{\on{N}\equiv1}$,} so we conclude that the image of $J_f(K)/2J_f(K)$ in $H^1(K,
(\on{Res}_{K_{f^{\operatorname{mon}}}/K} \mu_2)_{\on{N}\equiv1}/\mu_2)$ lies
in $\ker \delta_K$. Since classes in $\ker \delta_K$ correspond to orbits, we obtain the desired map~\eqref{eq-j2j}.

Now let $K = \BQ$. Since the group $\on{SL}_n/\mu_2$ is an adjoint
group, we have from~\cite[Theorem~6.22]{MR1278263} that the map
\begin{equation*}
H^1(\BQ,\on{SL}_n/\mu_2) \to \prod_v H^1(\BQ_v,\on{SL}_n/\mu_2)
\end{equation*}
is injective, where the product on the right-hand side is taken over
all places $v$ of $\BQ$. That the local map $J_f(\BQ_v)/2J_f(\BQ_v)
\hookrightarrow H^1(\BQ_v, J_f[2])$ has image contained in $\ker \delta_{\BQ_v}$
for every place $v$ of $\BQ$ then implies that the global map
$\on{Sel}_2(J_f) \hookrightarrow H^1(\BQ,J_f[2])$ has image contained
in $\ker \delta_\BQ$, as desired.

 We now claim that the $(\on{SL}_n/\mu_2)(\Q_v)$-orbit (resp., $(\on{SL}_n/\mu_2)(\Q)$-orbit) arising from an element $\alpha_0 \in J_f(\Q_v)/2J_f(\Q_v)$ (resp., $\alpha_0 \in \on{Sel}_2(J_f)$) admits a representative of the form $(\mc{A},B)$ for some $n$-ary quadratic form $B \in W_{\mc{A}}(\Q_v)$ (resp., $B \in W_{\mc{A}}(\Q)$). Since there is only one split unimodular $n$-ary quadratic form over $\Q_v$ up to change-of-basis, it suffices to prove the following:
\begin{proposition} \label{prop-splits}
Let $v$ be a place of $\Q$, and let $(A,B) \in W(\Q_v)$ be a representative of the orbit arising from an element of $J_f(\BQ_v)/2J_f(\BQ_v)$ via~\eqref{eq-j2j}. Then $A$ is split over $\BQ_v$.
\end{proposition}
\begin{proof}
Choose an element $\alpha_0 \in J_f(\BQ_v)/2J_f(\BQ_v)$, and let
$\alpha \in K_f^\times$ be a representative of $(x - \theta)(\alpha_0)$. It suffices to show that the symmetric bilinear form $\pi_{n-1} \circ \langle -,-\rangle \colon K_f \times K_f \to \Q_v$ has a maximal isotropic subspace defined over $\Q_v$. In~\cite[proof of Theorem~10]{MR3719247}, a maximal isotropic subspace over $\Q_v(\sqrt{f_0})$ is constructed, and an explicit basis $\big(\Q_v(\sqrt{f_0})\big)\langle v_1, \dots, v_{\frac{n}{2}}\rangle$ is given. This basis has the property that, for some integer $i \in \{1, \dots, \frac{n}{2}\}$, we have $v_j \in K_f$ for each $j \leq i$ and $v_j \in \sqrt{f_0}K_f$ for each $j > i$. Letting $v_j' \defeq \sqrt{f_0^{-1}}v_j \in K_f$ for each $j > i$, we conclude that the basis $\Q_v\langle v_1, \dots, v_i, v_{i+1}', \dots, v_{\frac{n}{2}}' \rangle$ defines a maximal isotropic subspace over $\Q_v$.
\end{proof}

Combining Propositions~\ref{prop-connectiontopso} and~\ref{prop-splits}, we conclude that an element $\alpha_0\in J_f(\Q_v)/2J_f(\Q_v)$ (resp., $\alpha_0 \in \on{Sel}_2(J_f)$) gives rise to an orbit of $\on{PSO}_{\mc{A}}(\Q_v)$ on $W_{\mc{A}}(\Q_v)$ (resp., $\on{PSO}_{\mc{A}}(\Q)$ on $W_{\mc{A}}(\Q)$).

Let $K$ be any field. In what follows, when we consider binary forms with a fixed nonzero leading coefficient $f_0$, we say that an element $B \in W_{\DA}(K)$ is \emph{soluble} if it arises via the above construction from an element of $J_f(K)/2J_f(K)$ for some binary form $f \in V(K)$ with leading coefficient $f_0$. When $K = \Q$, we say that an element $B \in W_{\DA}(K)$ is \emph{locally soluble} if it is soluble over $\Q_v$ for every place $v$ of $\Q$.

\subsection{Construction of integral representatives} \label{sec-bigconstruct}

We now want to construct integral representatives for the rational orbits arising from $2$-Selmer elements. To do this, it suffices to work over $K = \Q_p$. Let $f \in V(\Z_p)$ be a binary $n$-ic form with nonzero leading coefficient $f_0$ that is separable over $\Q_p$, and assume that $C_f$ is soluble over $\Q_p$ if $n \equiv 0 \pmod 4$. Fix an element \mbox{$\alpha_0 \in J_f(\BQ_p)/2J_f(\BQ_p)$.} We then have the following result, the proof of which occupies the balance of this section:
\begin{theorem} \label{thm-15}
There exists an integer $\upkappa \geq 1$ depending only on $n$ such that if we replace the binary $n$-ic form $f(x,y)$ with $g(x,y) = f(x, \upkappa y)$, then the $(\on{SL}_n/\mu_2)(\BQ_p)$-orbit on $W(\Q_p)$ arising from $\alpha_0$, viewed as an element of $J_g(\Q_p)/2J_g(\Q_p)$, contains a representative of the form $(\mc{A},B) \in W(\Z_p)$.
\end{theorem}

We split the proof of Theorem~\ref{thm-15} into three successively more complicated cases, depending on whether $f$ is primitive and on whether $p > n$. We start in \S\ref{sec-idealconstruct} with the case where $f$ is primitive and $p > n$. In this case, it is not necessary to scale the variable $y$: we can construct an integral representative with resolvent $f$ itself. The construction is facilitated by the fact that we can replace a binary $n$-ic form with a suitable $\on{SL}_2(\BZ_p)$-translate whose leading coefficient is not
divisible by $p$. Then, in \S\ref{sec-idealconstruct2}, we modify the construction in \S\ref{sec-idealconstruct} to handle imprimitive $f$ when $p > n$; here too, we obtain an integral representative with resolvent $f$. We conclude in \S\ref{sec-smallprimeconstruct} by building upon the construction in \S\ref{sec-idealconstruct2} to handle the case where $p < n$; here, we construct an integral representative with resolvent $f(x,p^{\mc{D}}y)$. The desired integer $\upkappa$ may then be taken to be $\upkappa = \prod_{p < n} p^{\mc{D}}$.

\begin{remark}
The main results of the work~\cite{MR3600041} (namely, Theorems~1, 2, 5--7, and 10--12, as well as Corollaries~3, 8, and 9) rely on Theorem~15 in \emph{loc.\ cit.}, which is an analogue of Theorem~\ref{thm-15} for odd-degree divisor classes. The proof of Theorem~15 on pp.~478--480 in \emph{loc. cit.} is correct when $f$ has leading coefficient $f_0 \in \BZ_p^\times$. However, the proof has two gaps when $f_0 \notin \BZ_p^\times$:
\begin{itemize}[leftmargin=1cm,itemsep=0pt]
\item The displayed equations in the middle of p.~479 verifying that $I_D^2 \subset f_0^{2m}\alpha_0I_f^{n-3}$ are not always correct, because the numbers $a_1, \dots, a_m$ may fail to be integral when $f_0 \not\in \BZ_p^\times$.
\item The specialization argument on p.~479 used to compute the norm $\on{N}(I_D)$ is not always correct when $f_0 \not\in \BZ_p^\times$, as can be verified by directly calculating $\on{N}(I_D)$ in small cases (e.g., by taking $n = 4$, $m = 1$, and $f$ to be imprimitive).
\end{itemize}
The three-part  strategy that we use in this section to prove Theorem~\ref{thm-15} can also be easily modified to obtain a complete proof of~\cite[Theorem~15]{MR3600041}.
\end{remark}

\subsubsection{Construction of integral representative for $f$ primitive and $p > n$} \label{sec-idealconstruct}

Suppose that $f$ is primitive and that $p > n$. Then there exists
$\gamma = \left[\begin{smallmatrix} s & t \\ u & v \end{smallmatrix}\right] \in \on{SL}_2(\BZ_p)$ such that the binary $n$-ic form $f'(x,y) \defeq f((x,y) \cdot \gamma) = \sum_{i = 0}^n f_i'
x^{n-i}y^i$ has leading coefficient $f_0' \in \BZ_p^\times$. Let
$\theta' = \frac{-u + v \theta}{s - t\theta} \in K_{f'}$ be the
root of $f'$, let $R_{f'} = \BZ_p\langle\zeta_1', \dots,
\zeta_{n-1}'\rangle$ be the basis elements~\eqref{eq-rfbasis} of $R_{f'}$ defined with respect to $\theta'$, and let $\on{N}'$ denote the norm with respect to this basis. Let $\alpha'_0$ be
the image of $\alpha_0$ under the isomorphism $J_f(\BQ_p)/2J_f(\BQ_p)
\simeq J_{f'}(\BQ_p)/2J_{f'}(\BQ_p)$ induced by $\gamma$. We first
construct a pair $(I',\alpha')$, where $I'$ is a based fractional
ideal of $R_{f'}$ and $\alpha' \in K_{f'}$ is a representative of $(x-\theta')(\alpha_0')$ such that ${I'}^2 \subset \alpha'I_{f'}^{n-2}$ and $\on{N}'(I')^2 = \on{N}'(\alpha') {f_0'}^{2-n}$. Then, the orbit corresponding to $(I',\alpha')$ has resolvent $f'$, so by Lemma~\ref{lem-sl2transform}, we can transform $(I',\alpha')$ into a pair $(I,\alpha)$ for $f$.

Since we have assumed that either $n \equiv 0 \pmod 4$ and
$C_{f'}(\BQ_p) \neq \varnothing$ or $n \equiv 2 \pmod 4$, it follows from~\cite[Prop.~21]{MR3600041} that there exists a rational divisor $\wt{D} \in \on{Div}^0(C_{f'})(\BQ_p)$ such that $\alpha_0' = [\wt{D}] \in J_{f'}(\BQ_p)/2J_{f'}(\BQ_p)$. The next
lemma shows that we can replace the divisor $\wt{D}$ with a new
divisor $D$ satisfying properties that turn out to be crucial
in constructing the desired ideal $I'$. For monic forms $f'$, this
lemma was proven in~\cite[\S3]{MR3167287}. Here, we extend that argument to handle classes $\wt{D} \in \on{Div}^0(C_{f'})(\BQ_p)$,
where $f_0' \in \BZ_p^\times$ is not necessarily equal to $1$.
\begin{lemma} \label{lem-divtranslate}
There exists an effective divisor $D = \sum_{i = 1}^m (P_i') \in
\on{Div}(C_{f'})(\BQ_p)$ such that the following three properties are
satisfied:
\begin{itemize}[leftmargin=1cm,itemsep=0pt]
\item $m \leq \frac{n}{2}$;
\item $P_i' = (a_i', 1, c_i') \in C_{f'}(\ol{\BQ}_p)$ with $a_i' \in
  \mc{O}_{\ol{\BQ}_p}$ and $c_i' \neq 0$ for each $i \in \{1, \dots,
  m\}$; and
\item $(x-\theta')(D) = (x-\theta')(\wt{D}) \in
  (K_{f'}^\times/K_{f'}^{\times 2}\BQ_p^\times)_{\on{N}\equiv1}$.
\end{itemize}
\end{lemma}
\begin{proof}
 We construct the desired divisor $D$ in a sequence of three steps as
 follows:

\vspace*{0.2cm}
\noindent {\sf Step 1}: It is a standard fact
(see~\cite[\S5]{MR1465369}) that there exists a principal divisor
$D_1 \in \on{Div}^0(C_{f'})(\BQ_p)$ such that $D + D_1$ is supported
neither at the ramification points of $C_{f'}$ over $\BP_{\BQ_p}^1$,
nor at the points $(1,0,\pm \sqrt{f_0'}) \in C_{f'}(\ol{\BQ}_p)$ lying above infinity in $\BP_{\BQ_p}^1$. We define $D$ to be $\wt{D} + D_1$; by translating $D$ with divisors of the form $2 D'$, where $D'$ is the divisor corresponding to a closed point of $C_{f'}$ over $\BQ_p$, we may assume that each point in the support of $D$ occur with multiplicity equal to $1$. It is clear that $(x-\theta')(\wt{D}) = (x-\theta')(D)$.

\vspace*{0.2cm}
\noindent {\sf Step 2}: We now construct a divisor $D_2 \in
\on{Pic}^0(C_{f'})(\BQ_p)$ such that $D -D_2$ has at most
$\frac{n}{2}$ non-infinite points in its support. In the previous
paragraph, we arranged for $D$ to take the form $D = \sum_{i = 1}^k
(Q_i)$, where $Q_i = (\mathfrak{a}_i, 1, \mathfrak{c}_i) \in
C_{f'}(\ol{\BQ}_p)$ with $\mathfrak{c}_i \neq 0$ for each $i$. If $k >
\frac{n}{2}$, let $\mathfrak{P} \in \BQ_p[x]$ be the monic polynomial
of degree $k$ defined by $\mathfrak{P}(x) = \prod_{i = 1}^k (x -
\mathfrak{a}_i)$, and let $\mathfrak{R} \in \BQ_p[x]$ be the unique
(not necessarily monic) polynomial of minimal degree at most $k-1$
such that $\mathfrak{R}(\mathfrak{a}_i) = \mathfrak{c}_i$ for each $i
\in \{1, \dots, k\}$. Then the polynomial $\mathfrak{R}(x)^2 - f'(x,1)
\in \BQ_p[x]$ has a root $x = \mathfrak{a}_i$ for each $i \in \{1,
\dots, k\}$, so we have that $\mathfrak{R}(x)^2 - f'(x,1) =
\mathfrak{P}(x) \mathfrak{H}(x)$ for some polynomial
$\mathfrak{H} \in \BQ_p[x]$. Because $f'$ is separable over $\BQ_p$,
we have that $\mathfrak{H} \not\equiv 0$.

Consider the rational function $g \defeq y - \mathfrak{R}(x)$ on
$C_{f'}$. Note that $g$ is defined over $\BQ_p$, that the zeros of $g$
are precisely the zeros of $\mathfrak{P} \mathfrak{H}$, and that
the poles of $g$ lie over infinity. Thus, letting $\on{div}(g)$ be the
principal divisor associated to the function $g$, we have that $D -
\on{div}(g)$ is supported over the zeros of $\mathfrak{H}$ and over
the infinite points. But since $k > \frac{n}{2}$, we have that $$\deg
\mathfrak{H} = \deg(\mathfrak{R}(x)^2 - f'(x,1)) - \deg
\mathfrak{P}(x) \leq k-1,$$
so $D - \on{div}(g)$ has
strictly fewer than $k$ non-infinite points in its support. By
iterating the above process, we obtain a sequence of functions $g =
g_1, g_2, \dots, g_\ell$ until $D - \sum_{i = 1}^\ell \on{div}(g_i)$
has at most $\frac{n}{2}$ non-infinite points in its support. We let
$D_2 = \sum_{i = 1}^\ell \on{div}(g_i)$, and we redefine $D$ to be $D
- D_2$. Just as in step 1, we modify $D$ so as to ensure that each
point in the support of $D$ occurs with multiplicity $1$. Once again,
the above modifications leave $(x-\theta')(D)$ unchanged.

\vspace*{0.2cm}
\noindent {\sf Step 3}: We now claim that we can simply delete every
point in the support of $D$ that is not of the form $(\mathfrak{a}, 1,
\mathfrak{c})$, where $\mathfrak{a} \in \mc{O}_{\ol{\BQ}_p}$. Observe
that there are two ways for a point
$(\mathfrak{a},\mathfrak{b},\mathfrak{c}) \in C_{f'}(\ol{\BQ}_p)$ in
the support of $D$ to fail to be in this form: either $\mathfrak{b}
\neq 0$ and $\mathfrak{a}/\mathfrak{b} \not\in \mc{O}_{\ol{\BQ}_p}$,
or $(\mathfrak{a},\mathfrak{b},\mathfrak{c}) = (1,0,\pm \sqrt{f_0'})$
is one of the points at infinity. We handle each of these two cases as
follows:

\vspace*{0.2cm}
\noindent \emph{Case 1}: First suppose that we have a point $Q_1 =
(\mathfrak{a}_1,1,\mathfrak{c}_1) \in C_{f'}(\ol{\BQ}_p)$ in the
support of $D$ such that $\mathfrak{a}_1 \not\in
\mc{O}_{\ol{\BQ}_p}$. Let $\{Q_i = (\mathfrak{a}_i, 1, \mathfrak{c}_i)
: 1 \leq i \leq k\}$ be the set of
$\on{Gal}(\ol{\BQ}_p/\BQ_p)$-conjugates of $Q_1$; since $D$ is a
rational divisor, each $Q_i$ occurs with the same multiplicity
$\nu$. We claim that $\xi \defeq \prod_{i = 1}^k (\mathfrak{a}_i - \theta) \in
K_{f'}^{\times 2}\BQ_p^\times$. Indeed, since $\mathfrak{a}_i \not\in
\mc{O}_{\ol{\BQ}_p}$, the denominator of $\xi$ is $p^j$ for
some integer $j > 0$, and for a suitable $u \in
\BZ_p^\times$, we have $u p^j \xi \in K_{f}^{\times2}$. We
then redefine $D$ to be $D - \sum_{i = 1}^k Q_i$.

\vspace*{0.2cm}
\noindent \emph{Case 2}: Now suppose that $(1,0,\sqrt{f_0'}) \in
C_f(\ol{\BQ}_p)$ is in the support of $D$. If $f_0' \not\in
\BQ_p^{\times 2}$, it follows that $(1, 0, -\sqrt{f_0'})$ is also in
the support of $D$, and so the support of $D$ over the points at
infinity is a multiple of the hyperelliptic class, and it follows
from~\cite[Prop.~5.1]{MR1465369} that this class lies in the
kernel of the ``$x-\theta'$'' map. If $f_0' \in \BQ_p^{\times2}$, then~\cite[\S2]{MR2521292} implies that $(x - \theta')(1,0,\pm \sqrt{f_0'}) = f_0' \equiv 1$. We may thus simply delete all of the infinite points in the support of $D$.

The modifications performed in both cases $1$ and $2$ above leave $(x-\theta')(D)$ unchanged. This completes the proof of the proposition.
\end{proof}

Let $D$ be as in Lemma~\ref{lem-divtranslate}. We proceed with the
construction under the assumption that $m$ is even; at the end of this
section, we explain how to deal with the case where $m$ is odd.

If $m = 0$, then we simply take $(I',\alpha') =
\big(I_{f'}^{\frac{n-2}{2}},1\big)$. Now suppose that $m > 0$. We define
polynomials $\mc{P}$, $\mc{R}$, and $\mc{H}$, much as in step 2 of the
proof of Lemma~\ref{lem-divtranslate}, but this time, we strive
to ensure that these polynomials have integral coefficients. Let
$\mc{P} \in \BQ_p[x]$ be the monic polynomial of degree $m$ defined by $\mc{P}(x) = \prod_{i = 1}^{m}(x- a_i')$. By our construction of $D$
in Lemma~\ref{lem-divtranslate}, the class $\alpha' \defeq
\mc{P}(\theta') \in (K_f^\times/K_f^{\times
  2}\BQ_p^\times)_{\on{N}\equiv 1}$ is equal to $(x -
\theta')(\alpha_0')$. Next, let $\mc{R} \in \BQ_p[x]$ be the unique
(not necessarily monic) polynomial of minimal degree at most $m - 1$
such that $\mc{R}(a_i') = c_i'$ for each $i \in \{1, \dots,m\}$. Then the polynomial $\mc{R}(x)^2 - f'(x,1) \in \BQ_p[x]$ has a root at
$x = a_i'$ for each $i$, so there exists $\mc{H} \in \BQ_p[x]$ such that $\mc{R}(x)^2 - f'(x,1) =
\mc{P}(x) \mc{H}(x)$. Since we chose $D$ so that $a_i' \in
\mc{O}_{\ol{\BQ}_p}$ for each $i$, we have $\mc{P} \in \BZ_p[x]$. We proceed under the assumption that we also have $\mc{R} \in \BZ_p[x]$; at the end of this section, we explain how to handle
the case where $\mc{R} \not\in \BZ_p[x]$. Then, since
$\mc{P},\mc{R},f'(x,1) \in \BZ_p[x]$, it follows from the relation
$\mc{R}(x)^2 - f'(x,1) = \mc{P}(x) \mc{H}(x)$ that \mbox{$\mc{H} \in
\BZ_p[x]$ as well.}

Consider the following $R_{f'}$-submodule of $K_{f'}$:
$$I' = R_{f'}\langle \mc{R}(\theta'), \mc{P}(\theta') \rangle.$$ Note
that $\theta' \in R_{f'}$ because $f_0' \in \BZ_p^\times$, so since
$\mc{P},\mc{R} \in \BZ_p[x]$, the fractional ideal $I'$ is in fact an
integral ideal of $R_{f'}$. We now prove that the pair $(I',\alpha')$
is of the form~\eqref{eq-newconds}:
\begin{lemma}
 We have ${I'}^2 \subset \alpha' I_{f'}^{n-2}$, and for
 some choice of a basis of $I'$, we have $\on{N}'(I')^2 = \on{N}'(\alpha') {f_0'}^{2-n}$.
\end{lemma}
\begin{proof}
Computing the square of $I'$ yields the following:
\begin{align*}
{I'}^2 = R_{f'}\langle\mc{R}(\theta')^2, \mc{R}(\theta') \mc{P}(\theta'), \mc{P}(\theta')^2 \rangle = \alpha'
R_{f'}\langle\mc{H}(\theta'), \mc{R}(\theta'), \mc{P}(\theta') \rangle
\subset \alpha' R_{f'},
\end{align*}
where the last step above follows from the fact that $\theta' \in
R_{f'}$ and $\mc{P},\mc{R},\mc{H} \in \BZ_p[x]$. Since $p \nmid f_0'$,
we have that $R_{f'} = I_{f'}^{n-2}$, and the first assertion follows.

That $\on{N}'(I')^2$ and $\on{N}'(\alpha')
{f_0'}^{2-n}$ generate the same fractional ideal of $\BZ_p$
follows immediately from~\cite[proof of
  Prop.~8.5]{MR3156850}. Since both $\on{N}'(I')^2$ and
$\on{N}'(\alpha')
{f_0'}^{2-n}$ are elements of
$\BQ_p^{\times 2}$, we can choose a basis of $I'$ so as to ensure that
$\on{N}'(I')^2 = \on{N}'(\alpha')
{f_0'}^{2-n}$.
\end{proof}

Next suppose that $\mc{R} \not\in \BZ_p[x]$. Then since $f_0' \in
\BZ_p^\times$, it follows from~\cite[proof of
  Prop.~12]{MR3719247} (see also~\cite[proof of
  Prop.~8.5]{MR3156850}) that we can modify $D$ so that it is
supported on $m-2$ points while preserving the properties in
Lemma~\ref{lem-divtranslate}. If the value of the new polynomial $\mc{R}$
arising from the modified divisor $D$ fails to be integral, we induct downward, repeating the process until $\mc{R}$ becomes integral or $D$ is supported on at most
$1$ point.

We now handle the case where $m$ is odd. After applying the steps in
the proof of Lemma~\ref{lem-divtranslate}, the only way for this to
fail is if we performed a deletion of an odd-degree closed point of
$C_{f'}$ over $\BQ_p$ in step 3. Regardless of whether this deletion
occurred in case 1 or 2 of step 3, we must have $f_0' \in \BQ_p^{\times 2}$, so the points $(1,0,\pm \sqrt{f_0'})$ at infinity must be defined over $\BQ_p$. In
this case, if $m$ is odd, it follows from~\cite[proof of Prop.~34]{MR3600041} that there exists a fractional ideal $I'$
of $R_{f'}$ such that ${I'}^2 \subset \alpha I_{f'}^{n-3}$ and
$\on{N}'(I')^2 = \on{N}'(\alpha') {f_0'}^{3-n}$. But since $f_0'
\in \BZ_p^\times$, we have that $I_{f'}^{n-3} = I_{f'}^{n-2}$ as fractional ideals, and since $f_0' \in \BQ_p^{\times 2}$, we can find a basis of
$I'$ with respect to which $\on{N}'(I')^2 = \on{N}'(\alpha')
{f_0'}^{2-n}$.

\subsubsection{Construction of integral representative for $f$ imprimitive and $p > n$} \label{sec-idealconstruct2}

Suppose that $\gcd(f_0,\dots, f_n) = p^d$ for some $d > 0$ and that $p > n$. We first show that we can reduce the problem to the case where $d = 1$. Suppose that $d > 1$, and let $\wt{f} = p^{-2 \lfloor d/2\rfloor} f \in \BZ_p[x,y]$; notice that the greatest common divisor of the coefficients of $\wt{f}$ is either $1$ or $p$. Then $C_f \simeq C_{\wt{f}}$, and this isomorphism induces an identification $J_f(\BQ_p)/2J_f(\BQ_p) \simeq J_{\wt{f}}(\BQ_p)/2J_{\wt{f}}(\BQ_p)$; letting $\wt{\alpha}_0$ be the image of $\alpha_0$ under this identification, observe that $(x-\theta)(\alpha_0) = (x-\theta)(\wt{\alpha}_0)$.

Suppose that we
have constructed a pair $(\wt{I},\wt{\alpha})$ of the
form~\eqref{eq-newconds} for $\wt{f}$, where $\wt{\alpha} \in
K_{\wt{f}}^\times = K_f^\times$ is a representative of
$(x-\theta)(\wt{\alpha}_0)$, and let $(A,B) \in W(\BZ_p)$ be any
representative of the corresponding
$\on{SL}_n^{\pm}(\BZ_p)$-orbit. Then by Lemma~\ref{lem-canscale}, the
pair $(A,p^{2 \lfloor d/2 \rfloor} B) \in W(\BZ_p)$ arises
for $f$, and its $\on{SL}_n^{\pm}(\BQ_p)$-orbit corresponds via Corollary~\ref{cor-fieldparamsl} to $\wt{\alpha} = (x-\theta)(\alpha_0) \in K_f^\times/K_f^{\times 2}$, as desired. We have thus reduced the problem to the case where $d \in \{0,1\}$; if $d
= 0$, then the construction in \S\ref{sec-idealconstruct} applies, so
we assume that $d = 1$ in the remainder of this section.

 Just like in \S\ref{sec-idealconstruct}, we now replace $f$ with its
 translate $f'(x,y) = f((x,y) \gamma)$ by an element $\gamma \in
 \on{SL}_2(\BZ_p)$; this time, however, we choose $\gamma$ so that $p$
 divides the leading coefficient $f_0'$ of $f'$ exactly once (such a
 $\gamma$ exists because $p > n$). Recalling the notation from
 \S\ref{sec-idealconstruct}, it once again suffices by
 Lemma~\ref{lem-sl2transform} to construct a pair $(I',\alpha')$,
 where $I'$ is a based fractional ideal of $R_{f'}$ and $\alpha' \in
 K_f'$ is a representative of $(x-\theta')(\alpha_0')$
 such that ${I'}^2 \subset \alpha' I_{f'}^{n-2}$ and
 $\on{N}'(I')^2 = \on{N}'(\alpha') {f_0'}^{2-n}$.

Just like we argued in \S\ref{sec-idealconstruct},~\cite[Prop.~21]{MR3600041} implies that there exists a divisor $\wt{D} \in \on{Div}^0(C_{f'})(\BQ_p)$ such
that $\alpha_0' = [\wt{D}] \in J_{f'}(\BQ_p)/2J_{f'}(\BQ_p)$. The following
lemma gives the analogue of Lemma~\ref{lem-divtranslate} in the present setting:
\begin{lemma} \label{lem-divtranslate2}
There exists an effective divisor $D = \sum_{i = 1}^m (P_i') \in
\on{Div}(C_{f'})(\BQ_p)$ such that the following three properties are
satisfied:
\begin{itemize}[leftmargin=1cm,itemsep=0pt]
\item $m \leq \frac{n}{2}$ is even;
\item $P_i' = (a_i', 1, c_i') \in C_{f'}(\ol{\BQ}_p)$ with $a_i' \in
  \mc{O}_{\ol{\BQ}_p}$ and $c_i' \neq 0$ for each $i \in \{1, \dots,
  m\}$; and
\item $(x-\theta')(D) = (x-\theta')(\wt{D}) \in
  (K_{f'}^\times/K_{f'}^{\times 2}\BQ_p^\times)_{\on{N}\equiv1}$.
\end{itemize}
\end{lemma}
\begin{proof}
Steps $1$ and $2$ in the proof of Lemma~\ref{lem-divtranslate} go
through without change. As for step 3, let $f''(x,y) = p^{-1}
f'(x,y)$. Then the ``$x-\theta'$'' map for $f'$ is the same as that
for $f''$; thus, the argument in case 1 of step 3 implies that we can
delete all points of the form $(\mathfrak{a},1,\mathfrak{c})$ such
that $\mathfrak{a} \not\in \mc{O}_{\ol{\BQ}_p}$. Since the points $(1,0, \pm \sqrt{f_0'})$ are evidently not defined over $\BQ_p$, the number $m$ must be even.
\end{proof}

Let $D$ be as in Lemma~\ref{lem-divtranslate2}. If $m = 0$, then we simply take $(I',\alpha') = \big(I_{f'}^{\frac{n-2}{2}},1\big)$. Now suppose that $m > 0$. We define polynomials $\mc{P}$, $\mc{R}$, and $\mc{H}$, much as in
\S\ref{sec-idealconstruct}, but this time, we strive to ensure not only that these polynomials have integral coefficients, but also that the coefficients of $\mc{R}$ are divisible by $p$. Let $\mc{P} \in \BQ_p[x]$ be the monic polynomial of degree $m$ defined by \mbox{$\mc{P}(x) = \prod_{i = 1}^m (x- a_i')$.} By our construction of $D$ in
Lemma~\ref{lem-divtranslate2}, the class $\alpha' \defeq \mc{P}(\theta') \in (K_f^\times/K_f^{\times2}\BQ_p^\times)_{\on{N}\equiv 1}$ is equal to $(x-\theta')(\alpha_0')$. Next, let $\mc{R} \in \BQ_p[x]$ be the unique polynomial of minimal degree at most $m-1$ such that $\mc{R}(a_i') = c_i'$ for each $i \in \{1,\dots,m\}$. Then the polynomial \mbox{$\mc{R}(x)^2 - f'(x,1) \in \BQ_p[x]$} has a root at $x = a_i'$ for each $i$, so there exists $\mc{H} \in \BQ_p[x]$ such that $\mc{R}(x)^2 - f'(x,1) = \mc{P}(x) \mc{H}(x)$. Since we chose $D$ so that $a_i' \in \mc{O}_{\ol{\BQ}_p}$ for each $i$, we have that $\mc{P} \in \BZ_p[x]$. We proceed under the assumption that we also have $\mc{R} \in p\BZ_p[x]$; at the end of this section, we explain how to
handle the case where $\mc{R} \not\in p\BZ_p[x]$. Then, since $\mc{P} \in \BZ_p[x]$ is monic and $\mc{R},f'(x,1) \in p\BZ_p[x]$, it follows from the relation $\mc{R}(x)^2 - f'(x,1) = \mc{P}(x) \mc{H}(x)$ that $\mc{H} \in p \BZ_p[x]$ as well.

Let $f''$ be as in the proof of Lemma~\ref{lem-divtranslate2}, and
consider the following $R_{f''}$-submodule of $K_{f''} = K_f$:
$$I'' \defeq R_{f''}\langle p^{-1} \mc{R}(\theta'), \mc{P}(\theta') \rangle.$$ Note that $\theta' \in R_{f''}$ because $p^{-1} f_0' \in \BZ_p^\times$, so since $\mc{P} \in \BZ_p[x]$ and $\mc{R} \in p \BZ_p[x]$, the fractional ideal $I''$ is in fact an integral ideal of $R_{f''}$. Letting $\on{N}''$ denote the fractional ideal norm with respect to the basis~\eqref{eq-rfbasis} of $R_{f''}$, given explicitly by $\BZ_p\langle 1, p^{-1} \zeta_1, \dots, p^{-1}  \zeta_{n-1} \rangle$, it follows immediately from~\cite[Proof of Prop.~8.5]{MR3156850} that $p^m\on{N}''(I'')^2$ and
$\on{N}''(\alpha')$ generate the same fractional ideal of $\BZ_p$.

We now construct an explicit list of elements that form part of a basis for $I''$ as a $\BZ_p$-module:
\begin{lemma} \label{lem-itsprim}
Let $\ell \defeq \frac{n-m-2}{2}$. Then the list $\mc{P}(\theta'), {\theta'} \mc{P}(\theta'),\dots, {\theta'}^\ell \mc{P}(\theta') \in I''$ can be extended to a basis of $I''$ as a $\BZ_p$-module.
\end{lemma}
\begin{proof}
Because $I''$ is an integral ideal of $R_{f''}$, it suffices to prove that $s(\vec{x}) \defeq \sum_{i = 0}^\ell x_i {\theta'}^i \mc{P}(\theta')$ is primitive as an element of $R_{f''}$ (i.e., that we have $s(\vec{x}) = R_{f''} \smallsetminus pR_{f''}$) for every primitive tuple $\vec{x} = (x_0, \dots, x_\ell) \in \BZ_p^{\ell+1} \smallsetminus p \BZ_p^{\ell+1}$. Given such a tuple $\vec{x}$, let $i \in \{0, \dots, \ell\}$ be the largest index such that $x_i \in \BZ_p^\times$. Then, when $s(\vec{x})$ is expressed in the power basis $\BZ_p\langle1, \theta',\dots, {\theta'}^{n-1}\rangle$ of $R_{f''}$, the coefficient of ${\theta'}^{m+i}$ is a $p$-adic unit. Thus, $p^{-1} x \not\in R_{f''}$.
\end{proof}

Consider the $R_{f''}$-submodule $p  I'' \subset K_{f''}$. Then, by Lemma~\ref{lem-itsprim}, we can find elements $\eta_1, \dots, \eta_{n-\ell-1} \in p I''$ such that the following is a basis of $p I''$ as a $\BZ_p$-module:
$$p I'' = \BZ_p\big\langle \eta_1, \dots, \eta_{n-\ell-1}, p
\mc{P}(\theta'),\dots, p  {\theta'}^{\ell} \mc{P}(\theta')\big\rangle.$$
Let $I' \subset K_{f'}$ be the free rank-$n$ $\BZ_p$-submodule with basis given by
$$I' = \BZ_p\big\langle \eta_1, \dots,
\eta_{n-\ell-1}, \mc{P}(\theta'),\dots, {\theta'}^{\ell} \mc{P}(\theta') \big\rangle.$$
\begin{lemma} \label{lem-issubmod}
 The $\BZ_p$-submodule $I' \subset K_{f'}$ is a fractional ideal of
 $R_{f'}$.
\end{lemma}
\begin{proof}
It suffices to show that $I'$ is closed under multiplication by
elements of $R_{f'}$. Since $p  I''$ is closed under
multiplication by elements of $R_{f''} \supset R_{f'}$, it suffices to
verify that $\zeta_i'  {\theta'}^j  \mc{P}(\theta') \in I'$ for each $i \in \{1, \dots, n-1\}$ and $j \in \{0, \dots, \ell\}$, or equivalently that $(p  {\theta'}^i)  \mc{P}(\theta') \in I'$ for each $i \in \{1,\dots, n+\ell-1\}$. But notice that
$(p {\theta'}^i) \mc{P}(\theta') = {\theta'}^i  (p \mc{P}(\theta')) \in p  I'' \subset I'$, because $p I''$ is an $R_{f''}$-submodule of $K_{f''}$ and is hence closed under multiplication by powers of $\theta'$.
\end{proof}

We now prove that the pair $(I',\alpha')$ is of the
form~\eqref{eq-newconds}:
\begin{lemma}
 We have ${I'}^2 \subset \alpha' I_{f'}^{n-2}$, and for some
 choice of a basis of $I'$, we have $\on{N}'(I')^2 =
 \on{N}'(\alpha'){f_0'}^{2-n}$.
\end{lemma}
\begin{proof}
To prove the containment, it suffices to prove the following: that $\eta_i \eta_j \in \alpha' I_{f'}^{n-2}$ for each $i,j \in \{1,\dots,n-\ell-1\}$, that $\eta_i {\theta'}^j \mc{P}(\theta') \in \alpha' I_{f'}^{n-2}$ for each $i \in \{1, \dots, n-\ell-1\}$ and $j \in \{0, \dots, \ell\}$, and that ${\theta'}^i \mc{P}(\theta')^2 \in \alpha'
I_{f'}^{n-2}$ for each $i \in \{0,\dots,2\ell\}$. By construction, we have that $\eta_i \in p I''$ for each $i \in \{1,\dots,n-\ell-1\}$. Thus, there exist $\beta_i,\gamma_i^{(0)},\dots,\gamma_i^{(\ell)} \in R_{f''}$ for each $i$ such that $\eta_i =
\beta_i \mc{R}(\theta') + \sum_{k = 0}^\ell \gamma_i^{(k)}  (p  {\theta'}^k  \mc{P}(\theta'))$. Using the relation $\mc{R}(\theta')^2 =
\mc{P}(\theta') \mc{H}(\theta')$, we find that
\begin{align} \label{eq-containcheck1}
{\alpha'}^{-1} (\eta_i  \eta_j) & \equiv \beta_i\beta_j
  \mc{H}(\theta') \pmod{p R_{f''}}, \quad \text{and} \\
  {\alpha'}^{-1} (\eta_i  {\theta'}^j  \mc{P}(\theta')) & =
  \eta_i  {\theta'}^j \equiv \beta_i  {\theta'}^j  \mc{R}(\theta') \pmod{p  R_{f''}}. \label{eq-containcheck2}
\end{align}
Since $\mc{R},\mc{H} \in p  \BZ_p[x]$, it follows from~\eqref{eq-containcheck1}
and~\eqref{eq-containcheck2} that every term in the expression of ${\alpha'}^{-1}  (\eta_i  \eta_j)$ and of ${\alpha'}^{-1} (\eta_i  {\theta'}^j  \mc{P}(\theta'))$ in the power basis $\BQ_p\langle1, \theta',\dots, {\theta'}^{n-1}\rangle$ has coefficient
divisible by $p$. We deduce that $\eta_i  \eta_j$ and $\eta_i  {\theta'}^j  \mc{P}(\theta')$ are contained in $\alpha'  \mc{R}_{f'} \subset \alpha' I_{f'}^{n-2}$ for each possible $i,j$. Finally, notice that ${\alpha'}^{-1}
{\theta'}^i \mc{P}(\theta')^2 = {\theta'}^i \mc{P}(\theta')$; since $\mc{P} \in \BZ_p[x]$ has degree $m$, and since $\theta',\dots, \theta'^{n-2} \in I_{f'}^{n-2}$, it follows that ${\theta'}^i \mc{P}(\theta')^2 \in \alpha' I_{f'}^{n-2}$ for each $i \in \{0, \dots, 2\ell = n-m-2\}$.

Recall that $p^m \on{N}''(I'')^2$ and $\on{N}''(\alpha')$ generate the same fractional ideal of $\BZ_p$. It follows that the same holds for $p^{m-2} \on{N}'(p  I'')^2$ and $\on{N}'(\alpha')$, and hence also for $\on{N}'(I')^2$ and $\on{N}'(\alpha')  {f_0'}^{2-n}$. Since both $\on{N}'(I')^2$ and $\on{N}'(\alpha') {f_0'}^{2-n}$ are elements of $\BQ_p^{\times 2}$, we can choose a basis of $I'$ to ensure that $\on{N}'(I')^2 = \on{N}'(\alpha') {f_0'}^{2-n}$.
\end{proof}

Next, suppose that we have $\mc{R} \not\in p  \BZ_p[x]$. Then we also have that $\mc{R}^2 \not\in p  \BZ_p[x]$. Letting $j(x) = p^{-1} (\mc{R}(x)^2 - f'(x,1)) = p^{-1}  \mc{P}(x)  \mc{H}(x)$, it follows that the coefficient of $x^i$ in $j(x)$ is a $p$-adic unit if $i = n$ and a $p$-adic integer for every $i \in \{2m-1, \dots, n-1\}$, and that the coefficient of $x^i$ in $j(x)$ has negative valuation for some $i \in \{0,\dots, 2m-2\}$. Analyzing the Newton polygon of $j$ then reveals that $j$ has at least $n - (2m-2)$ roots with negative valuation. Since $\mc{P}(x) \mid j(x)$ (as elements of $\BQ_p[x]$), we also know that $j$ has $m$ roots with positive valuation at the $\on{Gal}(\ol{\BQ}_p/\BQ_p)$-invariant list $a_1', \dots, a_m' \in \mc{O}_{\ol{\BQ}_p}$. Thus, letting $D' \in \on{Div}(C_{f'})(\ol{\BQ}_p)$ be the divisor supported on the remaining roots of $j$ that have positive valuation, we find that $D' \in \on{Div}(C_{f'})(\BQ_p)$ and $\deg D' \leq m-2$.

Just like in Step 2 of the proof of Lemma~\ref{lem-divtranslate}, consider the function $g \defeq y - \mc{R}(x)$ on $C_{f'}$. Note that $g$ is defined over $\BQ_p$, that the zeros of $g$ are precisely the zeros of $j$, and that the poles of $g$ lie over infinity. Applying the argument in step 3 of the proof of Lemma~\ref{lem-divtranslate} then implies that $D'$ satisfies the conditions of the lemma with $m$ replaced by $\deg D' \leq m-2$. The proof now concludes by downward induction on $m$; once $m = 1$, it is immediate that $\mc{R} \in p \BZ_p[x]$.

\subsubsection{Construction of integral representative for $p < n$} \label{sec-smallprimeconstruct}

Now let $p < n$. As in \S\ref{sec-idealconstruct2}, we
reduce the problem to the case where $p^d = \gcd(f_0, \dots, f_n) \in \{1, p\}$. However, we might \emph{not necessarily} be able to replace $f$ with an $\on{SL}_2(\BZ_p)$-translate $f'$ such that $\nu_p(f_0') = d$, where $\nu_p \colon \BQ_p^\times \to \BZ$ denotes the usual normalized $p$-adic valuation. Indeed, such a form $f'$ fails to exist precisely when $p < n$ and $f$ is divisible by $y  \prod_{i = 0}^{p-1}(x-i y)$ modulo $p$. Nonetheless, it follows from~\cite{MR1792411} that we can always find an $\on{SL}_2(\BZ_p)$-translate $f'$ of $f$ such that $\nu_p(f_0') \leq d_1$ for some bounded integer $d_1 > 0$ depending only on the degree $n$.

Let $\wt{f}(x,y) = f'(x,p^{d_2} y)$, where $d_2 \geq d_1$ is an
even integer, to be chosen shortly.
Then $C_{f'} \simeq C_{\wt{f}}$, and this isomorphism induces an identification $J_{f'}(\BQ_p)/2J_{f'}(\BQ_p) \simeq J_{\wt{f}}(\BQ_p)/2J_{\wt{f}}(\BQ_p)$; letting $\wt{\alpha}_0$ be the image of $\alpha_0'$ under this identification, observe that $(x - \theta')(\alpha_0') = (x - p^{d_2} \theta')(\wt{\alpha}_0)$. Let $\wt{\alpha} \in K_{\wt{f}}^\times = K_{f'}^\times$ be a representative of $(x- p^{d_2} \theta')(\wt{\alpha}_0)$. Because $\wt{f}$ is a multiple of
a form with unit leading coefficient, the constructions described in
\S\ref{sec-idealconstruct}--\ref{sec-idealconstruct2} imply the
existence of a pair $(\wt{I}, \wt{\alpha})$ of the
form~\eqref{eq-newconds} for $\wt{f}$, with one caveat: when $p = 2$,
the argument in case 1 of step 3 of the proof of
Lemma~\ref{lem-divtranslate} only goes through for $d_2 \gg 1$, where
the implied constant depends only on $n$. Choosing $d_2$ so that this argument goes through, let $(\wt{A},\wt{B}) \in W(\Z_p)$ be a representative of the $(\on{SL}_n/\mu_2)(\BZ_p)$-orbit corresponding to $(\wt{I},\wt{\alpha})$
via Theorem~\ref{thm-theconstruction}.

Since the pair $(\wt{A}, \wt{B})$ has resolvent $\wt{f}^{\operatorname{mon}}$, it follows that the pair $(\wt{A},
p^{-d_2} \wt{B}) \in W(\BQ_p)$ has resolvent $f'^{\operatorname{mon}}$. Imitating the proof of Lemma~\ref{lem-canscale}, we see that $(\wt{A}, p^{-d_2} \wt{B})$ corresponds via Theorem~\ref{thm-theconstruction} to the pair $(\wt{I}, p^{(n-1)d_2} \wt{\alpha})$, where we view $\wt{I}$ as a basis for
$K_{f'}$. By Lemma~\ref{lem-sl2transform}, the pair $(\wt{I}, p^{(n-1)d_2} \wt{\alpha})$ corresponds to a pair $(I,\alpha)$ of the
form~\eqref{eq-newconds} for $f$. We then have the following lemma:

\begin{lemma} \label{lem-cleardenoms}
 Let $(A,B) \in W(\BQ_p)$ be the pair arising from $(I,\alpha)$ via
 Theorem~\ref{thm-theconstruction}. Then there exists an element $g\in
 \on{SL}_n^{\pm}(\BQ_p)$ satisfying the following two properties: $(1)$ $g Ag^T = \mc{A}$; and $(2)$ When expressed in lowest terms, the denominators of the matrix entries of $g$ are bounded, depending only on $n$.
\end{lemma}
\begin{proof}
We first claim that the matrix entries of $A$ have bounded
denominators. Consider the map $\langle -,-\rangle^\sim \colon \wt{I}
\times \wt{I} \to I_{\wt{f}}^2$ defined by $(\beta,\gamma) \mapsto
(p^{(n-1)d_2} \wt{\alpha})^{-1} \beta\gamma$. It follows from
the proof of Theorem~\ref{thm-theconstruction} (see~\cite[Proof of
  Theorem~14]{swacl}) that, with respect to the chosen basis of
$\wt{I}$ and the standard basis~\eqref{eq-idealdef} of $I_{\wt{f}}^{n-2}$,
the matrix entries of $\langle -, -\rangle^\sim$ are polynomial
functions in the following quantities: the matrix entries of $A$ and
$B$, the coefficients of $\wt{f}$, and ${f_0'}^{-1}$. Thus, the
matrix entries of the map $\langle -, - \rangle^\sim$ have bounded denominators. From the proof of Lemma~\ref{lem-sl2transform},
we see that $I$ is a multiple of $\wt{I}$ and that $I_f^{n-2}$ is a
multiple of $I_{\wt{f}}^{n-2}$; in this way, $I$ and $I_f^{n-2}$ inherit bases
from $\wt{I}$ and $I_{\wt{f}}^{n-2}$. It is then clear that the matrix of
the map $\langle -, - \rangle^\sim$ with respect to the chosen basis
of $\wt{I}$ and the standard basis of $I_{\wt{f}}^{n-2}$ is equal to the
matrix of the map $\langle-,-\rangle \colon I \times I \to I_f^{n-2}$
defined by $(\beta,\gamma) \mapsto \alpha^{-1} \beta\gamma$ with
respect to the bases of $I$ and $I_f^{n-2}$ inherited from $\wt{I}$ and
$I_{\wt{f}}^{n-2}$. Since the basis that $I_f^{n-2}$ inherits from
$I_{\wt{f}}^{n-2}$ differs from the standard basis of $I_f^{n-2}$ by an
element of $\on{GL}_n(\BZ_p)$, the matrix entries of the map $\langle
- , - \rangle$ with respect to the basis of $I$ inherited from
$\wt{I}$ and the standard basis of $I_f^{n-2}$, and hence the matrix
entries of $A$ and $B$, have bounded denominators.

By Proposition~\ref{prop-splits}, $A$ is split over $\BQ_p$; by the $p$-adic Jordan decomposition, there exist finitely many $\on{SL}_n^{\pm}(\BZ_p)$-orbits of symmetric matrices in $\on{Mat}_n(\BQ_p)$ of determinant $1$ whose matrix
entries have bounded denominators. Combining these two observations with~\cite[Chapter~I,~Lemma~6.3]{MR0506372}, we see that there exists $g' \in \on{SL}_n^{\pm}(\BQ_p)$ such that its matrix entries have bounded denominators and such that $g'  A  {g'}^T = \left(\begin{array}{c|c} 0 & \on{id}
    \\ \hline \on{id} & 0 \end{array}\right) \in
\on{Mat}_n(\BZ_p)$. Taking $g'' \in \on{SL}_n^{\pm}(\BQ_p)$ such that $g''
\left(\begin{array}{c|c} 0 & \on{id} \\ \hline \on{id} & 0 \end{array}\right) {g''}^T = \mc{A}$, we set $g = g'' g'$.
\end{proof}

The matrix entries of $\wt{A}$ are integral, and the matrix entries of
$p^{-d_2} \wt{B}$ have denominators bounded by $p^{d_2}$; thus, the matrix entries of $-\wt{A}^{-1}(p^{-d_2} \wt{B})$ have bounded denominators. For each $i \in \{0, \dots, n-1\}$, let $m_i \colon K_f \to K_f$ denote the map of multiplication by $\zeta_i$. Then~\cite[Proof of Theorem~10]{swacl} implies that the matrix representing $m_i$ with respect to the basis of $\wt{I}$ has entries with bounded denominators for each $i$. Consequently, the matrix representing $m_i$ on $K_f$ with respect to the basis of $I$ inherited from $\wt{I}$ has entries with bounded denominators for each $i$. Letting $g \in \on{SL}_n^{\pm}(\BQ_p)$ be as in the statement of
Lemma~\ref{lem-cleardenoms}, it follows that the matrix representing
$m_i$ with respect to the basis of $I$ obtained by applying $g$ to the
basis inherited from $\wt{I}$ has entries with bounded denominators for each $i$. Since $B$ has bounded denominators, there exists a bounded even integer $\mc{D} \geq 0$ such that the pair $(g  A  g^T, g  (p^{\mc{D}}  B)  g^T)$ is an element of $W(\BZ_p)$ that arises for the form $f(x, p^{\mc{D}}  y)$.

\section{$2$-Selmer groups of elliptic curves, and orbits of small rank}

In this section, we start by recalling the parametrization of $2$-Selmer elements of elliptic curves by $\on{PGL}_2(\Q)$-orbits of integral binary quartic forms. Then, we strengthen the results of \S\ref{sec-maxranks} on ranks of orbits arising from binary $n$-ic forms in the case where $n = 4$.

\subsection{Representation of $\on{PGL}_2$ on binary quartic forms}  \label{sec-2selelliptic}

In this section, we define a certain representation of the group $\on{PGL}_2$ on the space of binary quartic forms. As we explain in \S\ref{sec-2selelliptic}, this representation plays a central role in the parametrization of $2$-Selmer elements of elliptic curves; see~\cite[\S~1.2]{MR3272925} for a thorough treatment.

Let $V$ denote the affine scheme over $\Z$ whose $R$-points are given by $\Sym_4 R^2$ for any $\Z$-algebra $R$ (i.e., binary quartic forms over $R$). Then $V$ has a natural structure of $\on{PGL}_2$-representation, given as follows: for each element $\gamma \in \on{PGL}_2(R)$, choose a lift $\wt{\gamma} \in \on{GL}_2(R)$, and let
\begin{equation} \label{eq-pgl2acts}
(\gamma \cdot f)(x,y) = (\det \wt{\gamma})^{-2} f((x,y)
\wt{\gamma}) \in V(R)
\end{equation}
for any binary quartic form $f \in V(R)$.

Let $a,b,c,d,e \colon V \to \BA_{\BZ}^1$ be the functions defined by sending a binary quartic form to its $x^4$-, $x^3y$-, $x^2y^2$-, $xy^3$-, and $y^4$-coefficients, respectively. Consider the homogeneous polynomial functions $I,J \colon V \to \BA_{\BZ}^1$ of respective degrees $2$ and $3$ defined in terms of the coefficient functions $a,b,c,d,e$ by the following formulas:
\begin{equation*}
I= 12ae-3bd+c^2 \quad \text{and} \quad J=72ace+9bcd-27ad^2-27eb^2-2c^3.
\end{equation*}
 Let $\on{Inv}$ denote the affine plane $\BA_{\BZ}^2$, and let $\iota \colon V \to \on{Inv}$ be the function defined by $\iota = (I,J)$. One readily verifies that $\iota$ is $\on{PGL}_2$-invariant; in fact, the functions $I$ and $J$ freely generate the ring of polynomial invariants for the action of $\on{PGL}_2$ on $V$. In particular, the discriminant function $\Delta \colon V \to \BA_{\BZ}^1$ is expressible as a polynomial in $I$ and $J$, and it is given explicitly by the formula
$$\Delta = (4I^3 - J^2)/27.$$

One key ingredient in the computation of the average size of the $2$-Selmer groups of elliptic curves in~\cite{MR3272925} is the following parametrization of $2$-Selmer elements of elliptic curves in terms of $\on{PGL}_2(\Q)$-orbits on $V(\Z)$:

\begin{theorem}[\protect{\cite[Theorem~3.5]{MR3272925}}] \label{thm-binquartparam}
Let $E = E^{I,J}$ be an elliptic curve over $\Q$. Then the elements of $\Sel_2({E^{I,J}})$ are in bijection with the set of $\PGL_2(\Q)$-equivalence classes of locally soluble integral binary quartic forms with invariants $2^4I$ and $2^6J$.
\end{theorem}

We note that the correspondence in Theorem~\ref{thm-binquartparam} was originally introduced by Birch and Swinnerton-Dyer (see~\cite{MR146143}) and was further developed by Cremona (see~\cite{MR1628193}, as well as~\cite{MR2509048,MR1951757,MR3472915}).

\subsection{Construction of special integral representatives} \label{sec-interpret}

Let $R = \Z$ or $\Z_p$ for a prime $p$, and fix $a \in R \smallsetminus \{0\}$. We say that $B \in U(R)$ is \emph{special at $a$} if for each $i \in \{2,3,4\}$ we have $\wedge^i B \equiv 0 \pmod{a^{i-1}}$. If $B \in U(R)$ is special at $a$, then the following properties hold:
\begin{itemize}[leftmargin=1cm,itemsep=0pt]
\item $g  B'  g^T$ is also special at $a$ for any $g \in (\on{SL}_4/\mu_2)(R)$ and $B' \in U(R)$ such that $B' \equiv B \pmod a$; and
\item The pair $(A,B)$ arises for the demonicization of $\det(x A + y B)$ at $a$ for any $A \in U_1(R)$.
\end{itemize}
The goal of this section is to show that, after scaling $B$ by a factor of $3$, the $(\on{SL}_4/\mu_2)(K)$-orbit of any pair $(A,B) \in W(R)$ that arises for a binary quartic form in $V_a(R)$ admits an integral representative pair $(A',B') \in W(R)$ such that $B'$ is special at $a$. It suffices to prove this when $R = \Z_p$, which we do as follows:

\begin{theorem} \label{thm-interpret}
Let $f(x,y) = ax^4 + bx^3y+cx^2y^2 + dxy^3 + ey^4 \in V(\Z_p)$, and suppose that a pair $(A,B) \in W(\BZ_p)$ arises for $f$. Then there exists another pair $(A',B') \in W(\BZ_p)$ belonging to the same $(\on{SL}_4/\mu_2)(\BQ_p)$-orbit
as $(A,3B)$ such that $B'$ is special at $a$.\footnote{Note that if $p \nmid b$ or if $R_f$ is the maximal order in $K_f$, then Theorem~\ref{thm-interpret} follows immediately from Theorem~\ref{thm-orbsmallrank}.}
\end{theorem}
\begin{proof}
We start by deciphering what it means for $(A,B)$ to arise for $f$. Condition (a) in Theorem~\ref{thm-theconstruction} tells us that the resolvent of $(A,B)$ is given by $f^{\operatorname{mon}}$, and so $a^3 \mid e(f^{\on{mon}}) = \det B$. Condition (b) in Theorem~\ref{thm-theconstruction} comprises the following three conditions, one for each value of the index $i \in \{1, 2,3\}$:
\begin{itemize}[leftmargin=50pt,itemsep=0pt]
\item[(1): $i = 1$.] We have $p_1(\frac{1}{a}
  \times -A^{-1}B)
  = -A^{-1}B
  \in \on{Mat}_4(\BZ_p)$; this is trivially true because \mbox{$\det A
  = 1$.}
\item[(2): $i = 2$.] We have
  $p_2(\frac{1}{a} \times -A^{-1}B) = \frac{1}{a} \big((A^{-1}B)^2 + b (-A^{-1}B)\big) \in \on{Mat}_4(\BZ_p)$. This is
  equivalent to stipulating that certain linear combinations of the $2
  \times 2$ minors of $-A^{-1}B$ are divisible by $a$, and just as in~\eqref{eq-secondi3}, it
  can be rewritten in the following convenient form:
\begin{equation} \label{eq-secondi}
    BA^{-1}B - b B \equiv 0 \pmod{a}.
\end{equation}
\item[(3): $i = 3$.] We have $p_3(\frac{1}{a}
  \times -A^{-1}B) = \frac{1}{a^2}  \big((A^{-1}B)^3 + b  (A^{-1}B)^2 + ac  (A^{-1}B) \big) \in \on{Mat}_4(\BZ_p)$. As alluded to in the proof of Theorem~\ref{thm-orbsmallrank}, this is equivalent to stipulating that each non-principal $3 \times 3$ minor of $-A^{-1}B$ is
  divisible by $a^2$ and the sum of each triple of principal $3 \times 3$ minors of $-A^{-1}B$ is \mbox{divisible by $a^2$.} Letting $\wt{B} = 3 B$, it follows that
  \begin{equation} \label{eq-needmod3}
  \wedge^3 \wt{B} \equiv 0 \pmod{a^2}.
  \end{equation}
\end{itemize}
Note that each of the above conditions is preserved under the action of $\on{SL}_4^{\pm}(\BZ_p)$, and so they continue to hold if we
replace $(A,B)$ with any $\on{SL}_4^{\pm}(\BZ_p)$-translate thereof.

We now apply the $p$-adic Jordan decomposition to put $\wt{B}$ into one of five possible shapes that are convenient to work with. By replacing the pair $(A,B)$ with a suitable $\on{SL}_4^{\pm}(\BZ_p)$-translate, we may assume that $\wt{B}$ takes the following shape if $p$ is odd:
\begin{equation} \label{eq-Bshape}
    \wt{B}  = \left(\begin{array}{cccc}u_1  p^{\mathfrak{b}_1} & 0 & 0 & 0
        \\ 0 & u_2  p^{\mathfrak{b}_2} & 0 & 0 \\ 0 & 0 & u_3
        p^{\mathfrak{b}_3} & 0 \\ 0 & 0 & 0 & u_4
        p^{\mathfrak{b}_4} \end{array}\right)
\end{equation}
where the $\mathfrak{b}_i\in \Z_{\geq 0} \cup \{\infty\}$ are such that $\mathfrak{b}_1 \geq \mathfrak{b}_2 \geq
\mathfrak{b}_3 \geq \mathfrak{b}_4 \geq 0$ and $u_i \in \BZ_p^\times$ for each $i
\in \{1, \dots, 4\}$. If $p = 2$, then $\wt{B}$ can take any one of the
following four shapes in addition to the shape~\eqref{eq-Bshape}:
\begin{align} \label{eq-Bshapeeven}
    \wt{B} & = \left(\begin{array}{c|c} 2^{\mathfrak{b}_1} \ol{B}_1 & 0
        \\ \hline 0 & \begin{array}{cc} u_3  2^{\mathfrak{b}_3} & 0 \\ 0
          & u_4  2^{\mathfrak{b}_4} \end{array} \end{array}\right),\quad
    \left(\begin{array}{c|c} \begin{array}{cc} u_1  2^{\mathfrak{b}_1} &
          0 \\ 0 & u_2  2^{\mathfrak{b}_2} \end{array} & 0 \\ \hline 0 &
        2^{\mathfrak{b}_3}  \ol{B}_3 \end{array}\right), \\ &
    \qquad\qquad\, \left(\begin{array}{c|c} 2^{\mathfrak{b}_1} \ol{B}_1
        & 0 \\ \hline 0 & 2^{\mathfrak{b}_3}
        \ol{B}_3 \end{array}\right), \quad \left(\begin{array}{c|c|c}
        u_1  2^{\mathfrak{b}_1} & 0 & 0 \\ \hline 0 & 2^{\mathfrak{b}_3}
        \ol{B}_3 & 0 \\ \hline 0 & 0 & u_4
        2^{\mathfrak{b}_4} \end{array}\right)\nonumber
\end{align}
where the $\mathfrak{b}_i \in \Z_{\geq 0} \cup \{\infty\}$ are valuations such that $\mathfrak{b}_1 \geq \mathfrak{b}_2 \geq \mathfrak{b}_3 \geq \mathfrak{b}_4 \geq 0$ and $u_i \in \BZ_p^\times$ for each $i
\in \{1, \dots, 4\}$ and where $\ol{B}_i \in \on{Mat}_2(\BZ_2)$ is
symmetric and has even diagonal entries and odd anti-diagonal entries
for each $i \in \{1,3\}$. We label the five cases
in~\eqref{eq-Bshape} and~\eqref{eq-Bshapeeven} in the above order as
1, 2, 3, 4, and 5.

We are now in position to prove the theorem. If $\wt{B}$ is special at $a$, then we are done, so suppose that $\wt{B}$ is \emph{not} special at $a$. Then $\mathfrak{a} \defeq \nu_p(a) > 0$, for the condition of being special at $a$ is trivial when $p \nmid a$. Moreover, condition (3) above implies that for some
integer $\mathfrak{c} \in \{1, \dots, \mathfrak{a}\}$, we have $\mathfrak{b}_3 + \mathfrak{b}_4 = \mathfrak{a}-\mathfrak{c}$ in
cases 1 and 2; that $2\mathfrak{b}_3 = \mathfrak{a}-\mathfrak{c}$ in cases 3 and 4; and that
$\mathfrak{b}_i \geq \mathfrak{a}+\mathfrak{c}$ for each $i \in \{1,2\}$ in cases 1, 2, 3, and
4. In case 5, it is not possible for $B$ to fail to be special at $a$, because condition (3) implies that $2\mathfrak{b}_3 + \mathfrak{b}_4 \geq
2\mathfrak{a}$, from which it follows that $\mathfrak{b}_3 + \mathfrak{b}_4 \geq \mathfrak{a}$. We may
thus disregard case 5 in the remainder of the proof.

 Write $A = [a_{ij}]$, $A^{-1} = [\breve{a}_{ij}]$ and $\wt{B} = [\wt{b}_{ij}]$, where $i,j \in \{1,2,3,4\}$. Then one readily verifies using~\eqref{eq-Bshape} that we have the following identities:
\begin{equation} \label{eq-secondishape}
    \wt{B}A^{-1}\wt{B} - b \wt{B} \equiv u_3u_4 p^{\mathfrak{b}_3 + \mathfrak{b}_4}  \left(\begin{array}{cccc} 0 & 0 & 0 & 0 \\ 0 & 0 & 0 & 0
        \\ 0 & 0 & -\breve{a}_{44} & \breve{a}_{34} \\ 0 & 0 &
        \breve{a}_{34} & -\breve{a}_{33} \end{array}\right) \pmod{a}
\end{equation}
in cases 1 and 2, and
\begin{equation} \label{eq-secondishapesecond}
    \wt{B}A^{-1}\wt{B} - b \wt{B} \equiv (\det \ol{B}_3) 2^{2\mathfrak{b}_3}
     \left(\begin{array}{cccc} 0 & 0 & 0 & 0 \\ 0 & 0 & 0 & 0
        \\ 0 & 0 & -\breve{a}_{44} & \breve{a}_{34} \\ 0 & 0 &
        \breve{a}_{34} & -\breve{a}_{33} \end{array}\right) \pmod{a}
\end{equation}
in cases 3 and 4. Combining~\eqref{eq-secondishape}
and~\eqref{eq-secondishapesecond} with condition (2) above,
reexpressed as in~\eqref{eq-secondi}, yields that $p^{\mathfrak{c}} \mid
\breve{a}_{ij}$ for each $i,j \in \{3,4\}$. Using the fact that $\det A
= 1$, a calculation reveals that
\begin{equation} \label{eq-flipident}
\breve{a}_{33}\breve{a}_{44} - \breve{a}_{34}^2 = a_{11}a_{22} -
  a_{12}^2.\footnote{In fact, this identity holds for any symmetric $4
    \times 4$ matrix of determinant $1$ over any ring.}
\end{equation}
Combining~\eqref{eq-flipident} with the fact that $p^{\mathfrak{c}} \mid
\breve{a}_{ij}$ for each $i,j \in \{3,4\}$ yields that $p^{2\mathfrak{c}} \mid
(a_{11}a_{22} - a_{12}^2)$.

We now claim that $p^{\mathfrak{c}} \mid a_{ij}$ for each $i,j \in \{1,2\}$. First suppose that the symmetric bilinear form $\left[\begin{smallmatrix} a_{11} & a_{12} \\ a_{12} & a_{22} \end{smallmatrix}\right]$ is
\mbox{$p$-adically} diagonalizable (i.e., its $p$-adic Jordan
decomposition consists of two $1 \times 1$ blocks); note that this
holds whenever $p$ is odd, but only sometimes when $p = 2$. Then, by
replacing the pair $(A,\wt{B})$ with a suitable translate under the action
of the subgroup $\on{SL}_2(\BZ_p) \subset \on{SL}_4^{\pm}(\BZ_p)$ that
acts by change-of-basis on the first two rows and columns of $A$ and
$B$, we may assume that $a_{11} = \upsilon_1 p^{\mathfrak{a}_1}$,
$a_{12} = 0$, and $a_{22} = \upsilon_2  p^{\mathfrak{a}_2}$, where
the $\mathfrak{a}_i$ are integers such that $\mathfrak{a}_1 \geq \mathfrak{a}_2 \geq 0$ and
$\mathfrak{a}_1 + \mathfrak{a}_2 \geq 2\mathfrak{c}$, and where $\upsilon_i \in \BZ_p^\times$ for
each $i \in \{1, 2\}$. Note that this transformation alters the first
two rows and columns of $\wt{B}$, but it preserves the property that
$p^{\mathfrak{a}+\mathfrak{c}} \mid \wt{B}_{ij}$ for each $i \in \{1, 2\}$. A calculation then
reveals that we have
\begin{align}
    \breve{a}_{33} & = -(\upsilon_2
  p^{\mathfrak{a}_2})a_{14}^2-(\upsilon_1  p^{\mathfrak{a}_1})a_{24}^2+(\upsilon_1\upsilon_2 p^{\mathfrak{a}_1+\mathfrak{a}_2})a_{44} \equiv -(\upsilon_2
  p^{\mathfrak{a}_2})a_{14}^2 \pmod{p^{\mathfrak{c}}}, \label{eq-a33form}\\ \breve{a}_{44}
  & = -(\upsilon_2 p^{\mathfrak{a}_2})a_{13}^2-(\upsilon_1  p^{\mathfrak{a}_1})a_{23}^2+(\upsilon_1\upsilon_2 p^{\mathfrak{a}_1+\mathfrak{a}_2})a_{33} \equiv -(\upsilon_2
  p^{\mathfrak{a}_2})a_{13}^2 \pmod{p^{\mathfrak{c}}}, \label{eq-a44form}
\end{align}
where we have used the fact that $\mathfrak{a}_1 \geq \mathfrak{c}$ because $\mathfrak{a}_1
\geq \mathfrak{a}_2$ and $\mathfrak{a}_1 + \mathfrak{a}_2 \geq 2\mathfrak{c}$. Since $p^{\mathfrak{c}} \mid
\breve{a}_{ii}$ for each $i \in \{3,4\}$, it follows
from~\eqref{eq-a33form} and~\eqref{eq-a44form} that either $\mathfrak{a}_2 <
\mathfrak{c}$ and $p \mid a_{1i}$ for each $i \in \{3,4\}$ or $\mathfrak{a}_2 \geq
\mathfrak{a}$. But the first case is impossible, for it would imply that
\mbox{$\det A \equiv 0 \pmod p$,} contradicting the fact that $\det A
= 1$. Thus, $\mathfrak{a}_2 \geq \mathfrak{c}$, and the claim holds in this case. If $p = 2$ and the symmetric bilinear form $\left[\begin{smallmatrix} a_{11} & a_{12} \\ a_{12} & a_{22} \end{smallmatrix}\right]$ is not \mbox{$2$-adically}
diagonalizable, then it is a multiple of a symmetric matrix over
$\BZ_2$ having odd determinant, and the claim then follows from the
fact that $2^{2\mathfrak{c}} \mid (a_{11}a_{22} - a_{12}^2)$.

We have now shown that, by replacing the pair $(A,\wt{B})$ with a
suitable translate under the action of $\on{SL}_4^{\pm}(\BZ_p)$, we
may assume that $A$ and $\wt{B}$ take the following shape:
 \begin{equation} \label{eq-ABshape}
    A = \left(\begin{array}{c|c} p^{\mathfrak{c}} \ol{A} & * \\ \hline * &
        *\end{array}\right) \text{ and } \wt{B} = \left(\begin{array}{c|c}
        p^{\mathfrak{a}+\mathfrak{c}}  \ol{B} & 0 \\ \hline 0 & \begin{array}{cc} u_3
           p^{\mathfrak{b}_3} & 0 \\ 0 & u_4
          p^{\mathfrak{b}_4} \end{array}\end{array}\right) \text{ or }
    \left(\begin{array}{c|c} 2^{\mathfrak{a}+\mathfrak{c}} \ol{B} & 0 \\ \hline 0 &
        2^{\mathfrak{b}_3} \ol{B}_3 \end{array}\right)
 \end{equation}
 where $\ol{A}, \ol{B} \in \on{Mat}_2(\BZ_p)$ and where ``$*$'' denotes an element of $\on{Mat}_2(\BZ_p)$ whose
 value is irrelevant.

 Consider the diagonal matrix $\varepsilon_{p^{-\mathfrak{c}}} \in \on{SL}_4(\ol{\BQ}_p)$ with diagonal entries given by $\sqrt{p^{-\mathfrak{c}}}$, $\sqrt{p^{-\mathfrak{c}}}$, $1/\sqrt{p^{-\mathfrak{c}}}$, and $1/\sqrt{p^{-\mathfrak{c}}}$. By Lemma~\ref{lem-newdiag}, the class of $\varepsilon_{p^{-\mathfrak{c}}}$ is an element of $(\on{SL}_4/\mu_2)(\BQ_p)$. We then define $(A',B') \defeq \varepsilon_{p^{-\mathfrak{c}}}\cdot (A,\wt{B}) \in W(\BQ_p)$. Because $A$ and $\wt{B}$ take the shapes given in~\eqref{eq-ABshape}, it follows that the pair
 $(A',B')$ is in fact an element of $W(\BZ_p)$ and moreover that $B'$ is given as follows:
\begin{equation} \label{eq-B'shape}
    B' = \left(\begin{array}{c|c} p^{\mathfrak{a}} \ol{B} & 0 \\ \hline 0
        & \begin{array}{cc} u_3 p^{\mathfrak{b}_3 + \mathfrak{c}} & 0 \\ 0 & u_4 p^{\mathfrak{b}_4 + \mathfrak{c}} \end{array}\end{array}\right) \quad
    \text{or} \quad \left(\begin{array}{c|c} 2^{\mathfrak{a}} \ol{B} & 0
        \\ \hline 0 & 2^{\mathfrak{b}_3+\mathfrak{c}}  \ol{B}_3 \end{array}\right)
\end{equation}
Since $\mathfrak{b}_3 + \mathfrak{b}_4 = \mathfrak{a}-\mathfrak{c}$ in cases 1 and 2 and that $2\mathfrak{b}_3 =
\mathfrak{a}-\mathfrak{c}$ in cases 3 and 4, it follows from~\eqref{eq-B'shape} that the
pair $(A',B')$ arises for $f$, and that $B'$ is special at $a$.
\end{proof}

Let $p$ be a sufficiently large prime. Theorem~\ref{thm-interpret} shows that even if the
integral representative constructed in
\S\ref{sec-idealconstruct}--\ref{sec-idealconstruct2} for the $(\on{SL}_4/\mu_2)(\Q_p)$-orbit arising from an element of $J_f(\Q_p)/2J_f(\Q_p)$ fails to be special at $a$, we can nonetheless find a \emph{different} integral representative that
is special at $a$, after potentially scaling $\upkappa$ by a factor of $6$. But it is often the case that every integral pair $(A,B)$ arising for $f$ is such that $B$ is special at $a$: indeed, as we demonstrated in Theorem~\ref{thm-orbsmallrank}, every orbit that arises for a binary $n$-ic form $f$ over $\BZ_p$ defining a maximal order $R_f$ is of rank $\leq 1$ modulo the leading coefficient $f_0$ of $f$.

\subsection{Ranks of special elements}

Let $m \geq 1$ be an integer. We denote the set of quaternary quadratic forms $B$ of rank $\leq 1$ modulo $m$ by
$$\mc{L}_m \defeq \{B \in U(R) : \text{$B$ is of rank $\leq 1$ modulo $m$}\}.$$
Next, we say that $B \in U(R)$ is of \emph{rank $\leq 2$ modulo $m$} if for every prime $p \mid m$, when the $p$-adic Jordan decomposition of $B$ is expressed as in~\eqref{eq-Bshape} and~\eqref{eq-Bshapeeven}, we have $\mathfrak{b}_2 \geq \nu_p(m)$ in cases $1$ and $3$, $\mathfrak{b}_1 \geq \nu_p(m)$ in cases $2$ and $4$, and $\mathfrak{b}_3 \geq \nu_p(m)$ in case $5$. Let $m' \geq 1$ be an integer. Then we write
$$\mc{L}_{m,m'} \defeq \{B \in \mc{L}_m : \text{$B$ is of rank $\leq 2$ modulo $mm'$}\}.$$
We have the following useful characterization of those quaternary quadratic forms $B \in U(R)$ that are special at $a$ in terms of the sets $\mc{L}_m$ and $\mc{L}_{m,m'}$:
\begin{proposition} \label{prop-whatspecialmeans}
Let $B \in U(R)$, and suppose that $B$ is special at $a$. Then there exist integers $a_1, a_2, a_3 \geq 1$ such that the following properties hold:
\begin{itemize}[leftmargin=1cm,itemsep=-3pt]
    \item For each prime $p \mid a$, there exist integers $\mathfrak{e}_1 = \mathfrak{e}_1(p), \mathfrak{e}_2 = \mathfrak{e}_2(p) \geq 0$ such that
    \begin{equation} \label{eq-exppropswant}
    \nu_p(a_1) = \mathfrak{e}_1,\quad  \nu_p(a_2) = \nu_p(a) - 2 \mathfrak{e}_1 + \mathfrak{e}_2, \quad \text{and} \quad \nu_p(a_3) = \mathfrak{e}_1 - 2\mathfrak{e}_2; \quad \text{and}
    \end{equation}
    \item We have that $B \in a_1 \mc{L}_{a_2,a_3}$.
\end{itemize}
\end{proposition}
\begin{proof}
Fix a prime $p$. By the Chinese remainder theorem, it suffices to prove the proposition for $R = \Z_p$. We retain notation from the proof of Theorem~\ref{thm-interpret} with $B$ replacing $\wt{B}$, and we assume without loss of generality that $B$ is in one of the five cases listed in~\eqref{eq-Bshape} and~\eqref{eq-Bshapeeven}. In case 1, the exponents $\mathfrak{a}$, $\mathfrak{b}_2$, $\mathfrak{b}_3$, and $\mathfrak{b}_4$ satisfy the following properties:
\begin{equation} \label{eq-expprop}
\mathfrak{b}_2 \geq \mathfrak{b}_3 \geq \mathfrak{b}_4 \geq 0, \quad \mathfrak{b}_3 + \mathfrak{b}_4 \geq \mathfrak{a} \geq 0,\quad \text{and} \quad \mathfrak{b}_2 + \mathfrak{b}_3 + \mathfrak{b}_4 \geq 2 \mathfrak{a}.
\end{equation}
We first claim that there exists an integer $\mathfrak{e}_1$ such that
\begin{equation} \label{eq-expprop2.75}
\max\{0,\mathfrak{a}-\mathfrak{b}_3,2\mathfrak{a}-\mathfrak{b}_2-\mathfrak{b}_3\} \leq \mathfrak{e}_1 \leq \min\{\mathfrak{b}_4,\tfrac{2\mathfrak{a}}{3}\}.
\end{equation}
It suffices to check that each of the lower bounds on $\mathfrak{e}_1$ in~\eqref{eq-expprop2.75} is less than or equal to each of the upper bounds. The three inequalities $0 \leq \mathfrak{b}_4$, $\mathfrak{a} - \mathfrak{b}_3 \leq \mathfrak{b}_4$, and $2\mathfrak{a}-\mathfrak{b}_2-\mathfrak{b}_3 \leq \mathfrak{b}_4$ respectively follow from the three properties in~\eqref{eq-expprop}. The inequality $0 \leq \frac{2\mathfrak{a}}{3}$ follows from the second property in~\eqref{eq-expprop}. The inequality $\mathfrak{a} - \mathfrak{b}_3 \leq \frac{2\mathfrak{a}}{3}$ is equivalent to $\mathfrak{b}_3 \geq \frac{\mathfrak{a}}{3}$, but this holds since the first two properties in~\eqref{eq-expprop} imply that $\mathfrak{b}_3 \geq \frac{\mathfrak{a}}{2}$. The inequality $2\mathfrak{a} - \mathfrak{b}_2 - \mathfrak{b}_3 \leq \frac{2\mathfrak{a}}{3}$ is only relevant when $\mathfrak{b}_4 \geq \frac{2\mathfrak{a}}{3}$, and it is equivalent to $\mathfrak{b}_2 + \mathfrak{b}_3 \geq \frac{4\mathfrak{a}}{3}$, but this follows from the first property in~\eqref{eq-expprop}.

We next claim that there exists an integer $\mathfrak{e}_2$ such that
\begin{equation} \label{eq-expprop2} \max\{0, \mathfrak{a}-\mathfrak{b}_2,2\mathfrak{e}_1 - \mathfrak{a}\} \leq \mathfrak{e}_2 \leq \min\{\mathfrak{e}_1 - \mathfrak{a} + \mathfrak{b}_3, \tfrac{\mathfrak{e}_1}{2}\}.
\end{equation}
Again, we check that each of the lower bounds on $\mathfrak{e}_2$ in~\eqref{eq-expprop2} is less than or equal to each of the upper bounds. The inequalities $0 \leq \mathfrak{e}_1 - \mathfrak{a} + \mathfrak{b}_3$, $0 \leq \frac{\mathfrak{e}_1}{2}$, and $\mathfrak{a} - \mathfrak{b}_2 \leq \mathfrak{e}_1 - \mathfrak{a} + \mathfrak{b}_3$ follow from~\eqref{eq-expprop2.75}. The inequality $\mathfrak{a} - \mathfrak{b}_2 \leq \frac{\mathfrak{e}_1}{2}$ follows from the first property in~\eqref{eq-expprop} combined with~\eqref{eq-expprop2.75}. The inequality $2\mathfrak{e}_1 - \mathfrak{a} \leq \mathfrak{e}_1 - \mathfrak{a} + \mathfrak{b}_3$ is equivalent to $\mathfrak{e}_1 \leq \mathfrak{b}_3$, which follows from the first property in~\eqref{eq-expprop} combined with~\eqref{eq-expprop2.75}. The inequality $2\mathfrak{e}_1 - \mathfrak{a} \leq \frac{\mathfrak{e_1}}{2}$ is equivalent to $\mathfrak{e}_1 \leq \frac{2\mathfrak{a}}{3}$, which follows from~\eqref{eq-expprop2.75}.

Having chosen $\mathfrak{e}_1$ and $\mathfrak{e}_2$ to satisfy the conditions~\eqref{eq-expprop2.75} and~\eqref{eq-expprop2}, respectively, we then take $a_1 = p^{\mathfrak{e}_1}$, $a_2 = p^{\mathfrak{a} - 2\mathfrak{e}_1 + \mathfrak{e}_2}$, and $a_3 = p^{\mathfrak{e}_1 - 2\mathfrak{e}_2}$. Since $\mathfrak{e}_1 \leq \mathfrak{b}_4$, the matrix $a_1^{-1}B$ lies in $U(\Z_p)$. Since $\mathfrak{a} - \mathfrak{e}_1 + \mathfrak{e}_2 \leq \mathfrak{b}_3$, the matrix $a_1^{-1}B$ is of rank $\leq 1$ modulo $a_2$. Finally, since $\mathfrak{a} - \mathfrak{e}_2 \leq \mathfrak{b}_2$, the matrix $a_1^{-1}B$ is of rank $\leq 2$ modulo $a_2a_3$.

A similar proof works for the other four cases. The argument in case 2 is the same as in case 1, with $\mathfrak{b}_2$ replaced by $\mathfrak{b}_1$. The argument in case 3 is the same as in case 1, with $\mathfrak{b}_3 = \mathfrak{b}_4$. The argument in case 4 is the same as in case 2, with $\mathfrak{b}_3 = \mathfrak{b}_4$. The argument in case 5 is the same as in case 1, with $\mathfrak{b}_2 = \mathfrak{b}_3$.
\end{proof}

\section{Preliminary results about functions on coregular spaces} \label{sec-coregular}

Let $U$ be a vector space with an action by a semisimple group $G$.
We make the following assumptions:
\begin{itemize}[leftmargin=1cm,itemsep=0pt]
\item The ring of polynomial invariants for the action of $G(\C)$ on $U(\C)$ is freely generated by $n$ integer polynomials $I_1,\ldots,I_n$.
\item There exists a $G$-invariant {\it discriminant} polynomial function $\Delta = \Delta_U$ on $U$ such that every orbit on which $\Delta$ does not vanish is semistable in the sense of GIT.
\item Every semistable orbit is stable, i.e., has finite stabilizer.
\end{itemize}
Let $\I$ denote affine $n$-space. We then have the {\it invariant map}
$\iota$ given by
\begin{equation*}
\begin{array}{rcl}
\iota\colon U&\to&\I\\[.1in]
u&\mapsto&\bigl(I_1(u),\ldots,I_n(u)\bigr),
\end{array}
\end{equation*}
which satisfies $\iota(u)=\iota(g\cdot u)$ for all $u\in U$ and $g\in G$.
Let $\mathcal{I}\subset \I(\R)$ be an open set, not containing any
elements where the discriminant vanishes.  Let $\kappa\colon \mathcal{I}\to
U(\R)$ be a smooth {\it section}, i.e., we have $\iota(\kappa(I))=I$
for every $I\in\mathcal{I}$. Assume that the size of the stabilizer in
$G(\R)$ of every element in the image of $\kappa$ is the same, say
$\sigma$.  Let $\phi\colon\I(\R)\to\R$ and $\Theta\colon G(\R)\to\R$ be smooth
compactly supported functions. We define the function
$\psi=\cS(\phi,\Theta)=\cS_\kappa(\phi,\Theta)\colon U(\R)\to \R$ as follows:
\begin{equation}\label{eqsmoothpushforward}
  \cS(\phi,\Theta)(u) \defeq
  \sum_{\substack{(\gamma,I)\in G(\R)\times\mathcal{I}\\\gamma\cdot\kappa(I)=u}}
  \Theta(\gamma)\phi(I).
\end{equation}
Note that the above sum has length $\sigma$ for every $u$ in the image
of $G(\R)\times\mathcal{I}$, and that $\cS(\phi,\Theta)$ is
also smooth and compactly supported.

The group $G(\R)$ acts on the set of smooth functions on $U(\R)$ via
$(\gamma\cdot\psi)(u) \defeq \psi(\gamma^{-1}\cdot u)$.  We then have the
following result, which is a smoothed version of the averaging method
(see \cite{MR2113024,MR3272925}).
\begin{theorem}\label{thsmoothavggen}
Let notation be as above. Let $L\subset G(\R)\cdot \kappa(\mathcal{I})$
be any discrete $G(\Z)$-invariant set, and let $T\colon L\to \R$ be a
$G(\Z)$-invariant function. Let $\FF_G$ be a fundamental domain for
the action of $G(\Z)$ on $G(\R)$, and let $d\gamma$ be a Haar-measure
on $G(\R)$. Let $\phi$ and $\Theta$ be
smooth, compactly supported functions, and let $\psi=\cS(\phi,\Theta)$. Then we have that
\begin{equation*}
  \sum_{u\in\frac{L}{G(\Z)}}\frac{T(u)}{|\Stab_{G(\Z)}(u)|}\phi(\iota(u))
  =\frac{1}{\sigma\Vol(\Theta)}
\int_{\gamma\in\FF_G}\Bigl(\sum_{u\in L}(\gamma\cdot\psi)(u)T(u)\Bigr)d\gamma,
\end{equation*}
where $\Vol(\Theta)$ denotes
$\int_{\gamma \in G(\R)}\Theta(\gamma)d\gamma$.
\end{theorem}
\begin{proof}
As usual, we regard $\FF_G\cdot \kappa(\cI)$ as a multiset, in which
an element $x\in U(\R)$ is represented
$\#\{(g,u):g\in\FF_G,u\in\kappa(\cI),g\cdot u=x\}$ times. As in~\cite[(8)]{MR3272925}, the $G(\Z)$-orbit of an element $x\in
G(\R)\cdot\kappa(\cI)$ is represented
\begin{equation*}
m(x)=\frac{\#\Stab_{G(\R)}(x)}{\#\Stab_{G(\Z)}(x)}
\end{equation*}
times in the multiset $\FF_G\cdot\kappa(\cI)$. Therefore, we have
\begin{equation}\label{eqsmavgpr0}
  \sum_{x\in\frac{L}{G(\Z)}}\frac{T(u)}{|\Stab_{G(\Z)}(u)|}\phi(\iota(u))
  =\frac{1}{\sigma}\sum_{\substack{x\in\FF_G\cdot\kappa(\cI)\\x\in L}}T(x)\phi(\iota(x)).
\end{equation}
Since for all $h\in G(\R)$, the set $\FF_G\cdot h$ is also a
fundamental domain for the left action of $G(\Z)$ on $G(\R)$, we have
\begin{equation}\label{eqsmavgpr1}
\sum_{\substack{x\in\FF_G\cdot\kappa(\cI)\\x\in L}}T(x)\phi(\iota(x))
=\frac{1}{\Vol(\Theta)}\int_{g\in G(\R)}\Bigl(\sum_{\substack{x\in
    g\FF_G\cdot\kappa(\cI)\\x\in
    L}}T(x)\phi(\iota(x))\Bigr)\Theta(g)dg.
\end{equation}
For $x\in g\FF_G\cdot\kappa(\cI)$, there exists a unique $x_I\in
\kappa(\cI)$ such that $x$ and $x_I$ are in the same
$G(\R)$-orbit. Moreover, there exist a finite number of elements
$g_1,\ldots,g_\sigma \in G(\R)$ such that $g_j\cdot x_I=x$. For $x\in L$ with
$\iota(x)=I$, we write the contribution from $x$ to the right hand side
of \eqref{eqsmavgpr1} as
\begin{equation*}
\begin{array}{rcl}
\displaystyle\int_{g\in G(\R)}\Bigl(\sum_{\substack{h\in\FF_G\\ghx_I=x}}
T(x)\phi(I)\Bigr)\Theta(g)dg&=&
\displaystyle T(x)\phi(I)
\sum_{j=1}^\sigma
\int_{g\in G(\R)}\#\{h\in\FF_G:hg=g_j\}\Theta(g)dg
\\[.2in]&=&\displaystyle
T(x)\phi(I)\sum_{j=1}^\sigma\int_{g\in(\FF_G)^{-1}g_j}\Theta(g)dg
\\[.2in]&=&\displaystyle
T(x)\int_{h\in \FF_G}\Bigl(\sum_{j=1}^\sigma\Theta(h^{-1}g_j)\phi(I)\Bigr)dh
\\[.2in]&=&\displaystyle
\int_{h\in \FF_G}(h\cdot\psi)(x)T(x)dh.
\end{array}
\end{equation*}
Adding up over $x\in L$ in the right hand side of \eqref{eqsmavgpr1},
we obtain
\begin{equation*}
\sum_{\substack{x\in\FF_G\cdot\kappa(\cI)\\x\in L}}T(x)\phi(\iota(x))
=\frac{1}{\Vol(\Theta)}\int_{h\in \FF_G}
\Bigl(\sum_{x\in L}(h\cdot\psi)(x)T(x)\Bigr)dh.
\end{equation*}
Together with \eqref{eqsmavgpr0}, we obtain the result.
\end{proof}

We now make the following further assumption on the pair $(G,U)$,
namely, that we have
\begin{equation*}
\sum_{j=1}^n\deg(I_n)=\dim(U).
\end{equation*}
In other words, the sum of the degrees of the invariants of $U$ equals
its dimension.  Let $dI$ and $du$ be top degree forms on $\I$ and $U$,
respectively. We will assume that $dI$ and $du$ induce Haar-measures
on $\I(\R)$ and $U(\R)$, respectively, normalized so that the
covolumes of $\I(\Z)$ in $\I(\R)$ and $U(\Z)$ in $U(\R)$ are $1$.  Let
$\omega$ denote a top degree form on $G$ defined over $\Z$.  Then,
from \cite[Remark 3.14]{MR3272925}, we have the following result.
\begin{theorem}\label{thJacgeneral}
Let $R$ be $\R$ or $\Z_p$, let $\smash{|\cdot|}$ denote the usual
normalized absolute value on $R$, and let $s\colon \I(R)\to U(R)$ be a
continuous section of $\iota$.
Then
there exists a constant $\J\in\Q^\times$, independent of $R$ and
$s$, such that for any measurable function $\phi$ on $U(R)$, we
have:
\begin{equation*}
  \begin{array}{rcl}
 \!\! \! \displaystyle\int_{G(R)\cdot s(\I(R))}\!\!\!\phi(u)du&\!\!\!=\!\!\!&
|\J|\!\displaystyle\int_{I\in \I(R)}\displaystyle\int_{g\in G(R)}
    \!\!\phi(g\cdot s(I))\omega(g) dI,\\[0.25in]
\displaystyle\int_{U(K)}\!\!\phi(u)du&\!\!\!=\!\!\!&
|\J|\!\displaystyle\int_{\substack{I\in U(K)\\ \Delta(I)\neq 0}}\!
\displaystyle\sum_{u\in\!\!\textstyle{\frac{\iota^{-1}(I)}{G(K)}}}
\!\!\!\frac{1}{\#\Stab_{G(K)}(u)}\int_{g\in G(K)}\!\!\!\phi(g\cdot u)
\omega(g)df.\\[-.05in]
  \end{array}
\end{equation*}
\end{theorem}

We have the following immediate consequence to Theorem
\ref{thJacgeneral}.
\begin{corollary}\label{corjacinfty}
Let $\phi\colon \I(\R)\to\R$ and $\Theta\colon G(\R)\to \R$ be smooth, compactly
supported functions, and let $\psi = \cS(\phi,\Theta)$. Then we have
\begin{equation*}
\Vol(\psi)=|\J|\Vol(\Theta)\Vol(\phi).
\end{equation*}
\end{corollary}

\subsection{Application to the representation of $\on{PGL}_2$ on $V$}

For $i\in\{0,1,2\}$, let $V(\R)^{(i)}$ denote
the set of elements $f\in V(\R)$ that have $4-2i$ distinct real
roots and $i$ distinct pairs of complex conjugate roots in
$\P^1_\C$. We write $V(\R)^{(2)}$ as $V(\R)^{(2+)}\cup V(\R)^{(2-)}$,
the union of positive definite and negative definite binary quartic
forms. For any $L\subset V(\R)$, we define $L^{(i)}\defeq L\cap
V(\R)^{(i)}$ for each $i\in\{0,1,2,2+,2-\}$. Let $\Inv(\R)^\pm$ denote the
subsets of $\Inv(\R)$ consisting of elements $(I,J)$ with
$\pm\Delta(I,J)>0$. The map $\iota$ sends $V(\R)^{(1)}$ to
$\Inv(\R)^-$ and $V(\R)^{(i)}$ to $\Inv(\R)^+$ for $i \in \{0,2\}$. We write $V(\R)^{\pm} \defeq \iota^{-1}(\on{Inv}(\R)^{\pm})$ and $V(\Z)^{\pm} \defeq V(\R)^{\pm} \cap V(\Z)$.

Let $\mathcal{H}\colon \Inv(\R)\to\R_{\geq 0}$ be a smooth, compactly
supported function, which we will later take to be an approximation of the characteristic function of the set of elements in $\Inv(\R)^\pm$
with height less than $1$. For a real number $X>0$, we set
$\H_X(I,J)\defeq \H(I/X^{1/3},J/X^{1/2})$. We say that an element $f\in
V(\Z)$ is {\it generic} if $f$ is irreducible over $\Q$ and the
stabilizer of $f$ in $\PGL_2(\Q)$ is trivial. For any set
$L\subset V(\Z)$, we denote by
$L^\gen$ the set of generic elements in $L$.  For a $\PGL_2(\Z)$-invariant function $T\colon V(\Z)\to \R$, an
$i\in\{0,1,2+,2-\}$, and a real number $X\geq 1$, we define
\begin{equation*}
\begin{array}{rcl}
N_\H(V^{(i)},T;X)&\defeq &\displaystyle\sum_{f\in\frac{V(\Z)^{(i)}}{\PGL_2(\Z)}}
\frac{T(f)}{|\Stab_{\PGL_2(\Z)}(f)|} \H_X\bigl(\iota(f)\bigr);
\\[.2in]
N^{\gen}_\H(V^{(i)},T;X)&\defeq &\displaystyle\sum_{f\in\frac{V(\Z)^{(i),\gen}}{\PGL_2(\Z)}}
T(f) \H_X\bigl(\iota(f)\bigr).
\end{array}
\end{equation*}

\subsubsection{Construction of fundamental sets}

The isomorphism class of the stabilizer in $\PGL_2(\R)$ of an element
$f\in V(\R)$ with nonzero discriminant depends only on the sign of the
discriminant of $f$. We denote the size of this stabilizer by
$\sigma^\pm$, and have $\sigma_0=\sigma_{2\pm}=\sigma^+=4$ and
$\sigma_1=\sigma^-=2$. We construct sections
\begin{equation*}
\begin{array}{rcl}
\kappa^{(i)}\colon  \Inv(\R)^+&\to& V(\R)^{(i)},\quad i\in\{0,2+,2-\}, \\[.1in]
\kappa^{(1)}\colon  \Inv(\R)^-&\to& V(\R)^{(1)},
\end{array}
\end{equation*}
satisfying the following properties:
\begin{enumerate}[itemsep=0pt]
\item We have $\iota(\kappa^{(i)}(I,J))=(I,J)$ for each $i$ and every
  pair $(I,J)\in\Inv(\R)^\pm$.
\item The coefficients of $f(x,y)=\kappa^{(i)}(I,J)$ are bounded by
  $O(H(I,J)^{1/6})$.
\end{enumerate}
The construction of these sections is standard. We denote the image of
$\kappa^{(i)}$ in $V(\R)^{(i)}$ by $\RR_V^{(i)}$.

\subsubsection{Fundamental domain for the action of $\PGL_2(\Z)$ on $\PGL_2(\R)$}

Let $\FF_2$ denote Gauss's usual fundamental domain for
$\PGL_2(\Z)\backslash \PGL_2(\R)$ in $\PGL_2(\R)$. It follows that
$\FF_2$ is contained in the Siegel set $\{n t
k:n\in N_2,t\in T_2,k\in K_2\},$
where $K_2$ is the real orthogonal group $\SO_2(\R)$, $N_2$ consists of lower triangular unipotent elements with absolutely bounded coefficients, and
\begin{equation}\label{nak}
T_2 = \left\{\left(\begin{array}{cc} t^{-1} & {} \\ {} & t \end{array}\right):
       t\geq \sqrt[4]3/\sqrt2 \right\}.
\end{equation}

\subsubsection{Smoothed averaging over $\PGL_2(\R)$}
For a function $\psi\colon V(\R)\to\C$, and a positive real number $X>0$,
define $\psi_X\colon  V(\R)\to\C$ to be $\psi_X(f)\defeq \psi(f/X)$. Fix
$i\in\{0,1,2+,2-\}$. Given smooth, compactly supported functions
$\mathcal{H}\colon \Inv(\R)\to\R_{\geq 0}$ and $\Theta\colon \PGL_2(\R)\to
\R_{\geq 0}$, we obtain the smooth, compactly supported function
$\cS_{\kappa^{(i)}}(\H,\Theta)=:\psi^{(i)}\colon V(\R)^{(i)}\to\R_{\geq 0}$.
Let $X>0$ be a real number. We clearly have
\begin{equation*}
\cS_A(\H_X,\Theta)=\cS_A(\H,\Theta)_{X^{1/6}}.
\end{equation*}
For a $\PGL_2(\Z)$-invariant function $T\colon V(\Z)^{(i)}\to \R$, a
function $\psi=\psi^{(i)}$ as above, and a real number $X>0$, we
define the integral $\cO^{(i)}(T,\psi;X)$ to be
\begin{equation}\label{eqorbintdefV}
  \cO^{(i)}(T,\psi;X)\defeq \int_{\gamma\in \FF_2}\Bigl(
  \sum_{f\in V(\Z)^\gen}(\gamma\cdot\psi_X)(f) T(f)
  \Bigr)d\gamma.
\end{equation}
Then the following result is an immediate consequence of Theorem
\ref{thsmoothavggen}.

\begin{proposition}\label{propbqfavg}
Let notation be as above. Fix $i\in\{0,1,2+,2-\}$. Let $T\colon V(\Z)\to\R$
be a $\PGL_2(\Z)$-invariant function, and let  $\psi=\cS_V(\H,\Theta)$. Then we have that
\begin{equation*}
N_\H^\gen(V^{(i)},T;X)=
\frac{1}{\sigma_i\Vol(\Theta)}\O^{(i)}(T,\psi;X^{1/6}).
\end{equation*}
\end{proposition}

\subsection{Application to the representations of $\on{SL}_4$ on $W$ and $\on{PSO}_{\cA}$ on $W_{\cA}$}

For $i\in\{0,1,2+,2-\}$, the number of $\SL_4(\R)$-orbits on $W(\R)$  with resolvent $f$ for some fixed $f\in V(\R)^{(i)}$ depends only on $i$. We denote this number by $\uptau(i)$, and we write $W(\R)^{(i)} = \iota^{-1}(V(\R)^{(i)}$. Similarly, the number of $\PSO_{\DA}(\R)$-orbits (resp., soluble $\PSO_{\DA}(\R)$-orbits) on $W_{\DA}(\R)$ arising from some fixed $f\in V(\R)^{(i)}$ depends only on $i$. We denote this number by $\uptau_{\DA}(i)$ (resp., $\uptau_{\DA}^{\on{sol}}(i)$), and we write $W_{\DA}(\R)^{(i)} = \iota^{-1}(V_1(\R)^{(i)})$. This yields a natural partition of the set of
elements in $W(\R)^{(i)}$ (resp., $W_{\DA}(\R)^{(i)}$) with nonzero discriminant into
subsets $W(\R)^{(i,j)}$ (resp., $W_{\DA}(\R)^{(i,j)}$) with $1\leq j\leq
\uptau(i)$ (resp., $1\leq j\leq \uptau_{\DA}(i)$).

Furthermore, the size of the stabilizer in $\SL_4(\R)$ (resp., $\PSO_{\DA}(\R)$) of an
element of $W(\R)^{(i)}$ (resp., $W_{\DA}(\R)^{(i)}$) depends only on $i$. We denote this size
by $\sigma(i)$ (resp., $\sigma_{\DA}(i)$).

\subsubsection{Construction of fundamental sets}

For each $i \in \{0,1,2+\}$, and for each $j$ with $1\leq j\leq \uptau(i)$
and $1\leq j\leq \uptau_{\DA}(i)$, there exist smooth sections
\begin{equation*}
\begin{array}{rcl}
\kappa^{(i,j)}\colon  V(\R)^{(i)}&\to& W(\R)^{(i,j)}, \\[.1in]
\kappa_{\DA}^{(i,j)}\colon  V_1(\R)^{(i)}&\to& W_{\DA}(\R)^{(i,j)},
\end{array}
\end{equation*}
respectively, satisfying $\iota(\kappa^{(i,j)}(f))=f$ for every $f\in
V(\R)^{(i)}$ (resp., $\iota(\kappa_{\DA}^{(i,j)}(f))=f$ for every $f \in V_1(\R)^{(i)}$) for each $i,j$. We denote the images of $\kappa^{(i,j)}$ and
$\kappa_{\DA}^{(i,j)}$ by $\RR^{(i,j)}$ and $\RR_{\DA}^{(i,j)}$, respectively.

We can choose the section $\kappa^{(i,j)}$ so that the
following is satisfied. For real numbers $Y,T>1$, and for a binary
quartic form $f(x,y)\in V(\R)^{(i)}$ with coefficients $(a,b,c,d,e)$,
such that $a=O(Y)$, $b=O(TY)$, $c=O(T^2Y)$, $d=O(T^3Y)$, and
$e=O(T^4Y)$, the coefficients $a_{ij}$ and $b_{ij}$ of
$(A,B)=\kappa^{(i,j)}(f)$ are bounded by $O(Y^{1/4})$ and
$O(TY^{1/4})$, respectively. Similarly, we can choose $\kappa_{\DA}^{(i,j)}$ such that the
following is satisfied. For a real number $Y>1$, and for a binary
quartic form $f(x,y)\in V(\R)^{(i)}$ with coefficients $a=\det(\DA)=1$, $b=O(Y)$, $c=O(Y^2)$, $d=O(Y^3)$, and
$e=O(Y^4)$, the coefficients $b_{ij}$ of $(\DA,B)=\kappa_{\DA}^{(i,j)}(f)$
are bounded by $O(Y)$.

\subsubsection{Fundamental domains for the action of $G(\Z)$ on $G(\R)$}

We now describe Siegel sets containing fundamental domains
$\FF_4$ and $\FF_{\DA}$ for the actions of $\SL_4(\Z)$ on $\SL_4(\R)$ and
$\PSO_{\DA}(\Z)$ on $\PSO_{\DA}(\R)$, respectively.

We may choose a fundamental domain
$\FF_4$ for the action of $\SL_4(\Z)$ on $\SL_4(\R)$ contained
 in a Siegel domain $N_4T_4K_4$ in
$\SL_4(\R)$. Here, $K_4=\SO_4(\R)$ is a maximal compact subgroup of
$\SL_4(\R)$, $N_4$ consists of lower triangular unipotent elements
with absolutely bounded coefficients, and $T_4$ is given by
\begin{equation*}
T_4=\left\{\left(\begin{array}{cccc}
    {s_1^{-3}s_2^{-1}s_3^{-1}} & {}&{}&{}\\ {} & \!\!\!\!\!\!\!\!\!\!\!\!\!\!\!\!\!{s_1s_2^{-1}s_3^{-1}}&{}&{}\\{} & {}&\!\!\!\!\!\!\!\!\!\!\!\!\!{s_1s_2s_3^{-1}}&{}\\{} & {}&{}&\!\!\!\!\!\!\!\!\!\!{s_1s_2s_3^{3}}
\end{array}\right):s_1,s_2,s_3>c\right\},
\end{equation*}
for some $c>0$. The {\it weights} $w_{ij}$ of $a_{ij}$ and $b_{ij}$
are the eigenvalues for the torus action on the space dual to $W(\R)$,
and are given below: here $a_{ij}$ and $b_{ij}$ are the coefficients
of $A$ and $B$, respectively, for $(A,B)\in W(\R)$.

\begin{equation*}
\begin{array}{lclclcl}
w_{11}= s_1^{-6}s_2^{-2}s_3^{-2} & \quad& w_{12}= s_1^{-2}s_2^{-2}s_3^{-2} &\quad&
w_{13}= s_1^{-2}s_2^{0}s_3^{-2} &\quad& w_{14}= s_1^{-2}s_2^{0}s_3^{2}\\[.1in]
&&w_{22}= s_1^{2}s_2^{-2}s_3^{-2}&\quad&
w_{23}= s_1^{2}s_2^{0}s_3^{-2}&\quad& w_{24}= s_1^{2}s_2^{0}s_3^{2}\\[.1in]
&&&&
w_{33}= s_1^{2}s_2^{2}s_3^{-2}&\quad& w_{34}= s_1^{2}s_2^{2}s_3^{2}\\[.1in]
&&&&&& w_{44}=s_1^{2}s_2^{2}s_3^{6}
\end{array}
\end{equation*}
A Haar-measure on $\SL_4(\R)$ in Iwasawa coordinates is
$s_1^{-12}s_2^{-8}s_3^{-12} dn d^\times s dk$, where $dn$ is Euclidean
measure on $N_4$, $d^\times s$ denotes $\prod_i d^\times s_i$, and
$dk$ is probability Haar-measure on $\SO_4(\R)$.

\medskip

It follows from Lemma~\ref{lem-newdiag} that $\PSO_{\DA}(\Z) \backslash \PSO_{\DA}(\R) \simeq \SO_{\DA}(\Z) \backslash \SO_{\DA}(\R)$, and so it suffices to construct a fundamental domain for the action of $\SO_{\DA}(\Z)$ on $\SO_{\DA}(\R)$.
We may choose a
fundamental domain $\FF_\DA$ for the action of $\SO_{\DA}(\Z)$ on
$\SO_{\DA}(\R)$ contained in a Siegel domain $N_\DA T_\DA K_\DA$ in
$\SO_{\DA}(\R)$. Here, $K_{\DA}$ is a maximal compact subgroup of
$\SO_{\DA}(\R)$, $N_\DA$ consists of lower triangular unipotent elements
with absolutely bounded coefficients, and $T_\DA$ is given by
\begin{equation*}
  T_\DA=\left\{\left(\begin{array}{cccc}
    {s_1^{-1}s_2^{-1}} & {} & {} & {}
    \\ {} & {s_1^{-1}s_2} & {} & {}
    \\{} & {} & {s_1s_2^{-1}} & {}
    \\{} & {} & {} & {s_1s_2}
  \end{array}\right):s_1,s_2>c\right\},
\end{equation*}
for some $c>0$. The group $\SO_\DA(\R)$ acts on $W_\DA(\R)$, and the weights
$w(b_{ij})$ of the coefficients of $B$ for $(\DA,B)\in W_\DA(\R)$ are as follows:
\begin{equation*}
\begin{array}{lclclcl}
w(b_{11})= s_1^{-2}s_2^{-2} & \quad& w(b_{12})= s_1^{-2}&\quad&
w(b_{13})= s_2^{-2}&\quad& w(b_{14})= 1\\[.1in]
&&w(b_{22})= s_1^{-2}s_2^{2}&\quad&
w(b_{23})= 1&\quad& w(b_{24})= s_2^2\\[.1in]
&&&&
w(b_{33})= s_1^{2}s_2^{-2}&\quad& w(b_{34})= s_1^2\\[.1in]
&&&&&& w(b_{44})=s_1^2 s_2^2
\end{array}
\end{equation*}

\subsubsection{Smoothed averaging over $\SL_4(\R)$ and $\PSO_{\DA}(\R)$}

Recall that for a function $\psi\colon V(\R)\to\C$ and a real number $X>0$,
we define $\psi_X\colon V(\R)\to\C$ via $\psi_X(f)=\psi(f/X)$. For a real
number $Y>0$, we define $\psi[Y]\colon V(\R)\to\C$ via
\begin{equation*}
\psi[Y](f)\defeq
\Bigl(\Bigl(\begin{array}{cc}Y^{-1/2}&\\&Y^{1/2}
\end{array}
\Bigr)\cdot\psi\Bigr)(f)
=\psi\bigl(
aY^2x^4+bYx^3y+cx^2y^2+dY^{-1}xy^3+eY^{-2}y^4\bigr),
\end{equation*}
for $f(x,y)=ax^4+bx^3y+cx^2y^2+dxy^3+dy^4$.  We say that an element
$w$ in $W(\Z)$ (resp., $W_A(\Z)$) is {\it generic} if: (a) the element
$w$ is not distinguished; (b) the stabilizer in $\SL_4(\Q)$
(resp., $\PSO_A(\Q)$) of $w$ is trivial; and (c) the resolvent of $w$
is generic. For any subset $L$ of $W(\Z)$ or $W_A(\Z)$, we denote the
set of generic elements in $L$ by $L^\gen$. We now define the
following counting functions on $W(\Z)$:
\begin{equation}\label{eqNW}
\begin{array}{rcl}
\displaystyle N_\psi(W^{(i,j)},T;X,Y)&\defeq &
\displaystyle \sum_{w\in\frac{W(\Z)^{(i,j)}}{\SL_4(\Z)}}
\frac{T(w)}{|\Stab_{\SL_4(\Z)}(w)|}\cdot (\psi_X[Y])(\iota(w));
\\[.2in]
\displaystyle N_\psi^\gen(W^{(i,j)},T;X,Y)&\defeq &
\displaystyle \sum_{w\in\frac{W(\Z)^{(i,j),\gen}}{\SL_4(\Z)}}
T(w)\cdot (\psi_X[Y])(\iota(w)).
\end{array}
\end{equation}

Let $\Phi\colon W(\R)\to \C$ be a function. For real numbers $X>0$ and
$Y>0$, we define $\Phi[X,Y]\colon W(\R)\to\C$ to be
\begin{equation}\label{eqPsiWXY}
(\Phi[X,Y])(A,B)\defeq \Phi\Bigl(\frac{A}{X},\frac{B}{Y}\Bigr).
\end{equation}
We now have the following result.

\begin{proposition}\label{propSL4avg}
Let $\psi\colon V(\R)\to\R$ be a smooth, compactly supported function. Let
$T\colon W(\Z)\to\R$ be a $\SL_4(\Z)$-invariant function. Let
$\Theta\colon \SL_4(\R)\to\R$ be a smooth and compactly supported
function. Let $\Phi$ denote $\cS(\psi,\Theta)$. For positive real
numbers $X$ and $Y$, we have
\begin{equation*}
N^\gen_\psi(W^{(i,j)},T;X,Y)=
\frac{1}{\sigma(i)\Vol(\Theta)}\int_{\gamma\in\FF_4}\Bigl(
\sum_{w\in W(\Z)^{(i,j),\gen}} T(w)\cdot(\gamma\cdot\Phi[Y^{-1/2}X^{1/4},Y^{1/2}X^{1/4}])(w)
\Bigr)d\gamma.
\end{equation*}
\end{proposition}
\begin{proof}
The only thing necessary to verify is that for positive real numbers
$X$ and $Y$, we have
\begin{equation*}
\cS(\psi_X[Y],\Theta)=\cS(\psi,\Theta)[Y^{-1/2}X^{1/4},Y^{1/2}X^{1/4}].
\end{equation*}
This equality is apparent, and the result follows from Theorem
\ref{thsmoothavggen}.
\end{proof}

\medskip

Let $A$ be a fixed quaternary quadratic form with $\det A = n$. Let
$V_n$ denote the subspace of $V$ consisting of binary quartic forms
whose $x^4$-coefficient is $n$. Given a function $\psi\colon V_n(\R)\to \C$,
define $\psi\{Y\}\colon V_n(\R)\to\C$ to be
\begin{equation*}
\psi\{Y\}(f(x,y))\defeq
\psi(f(x,y/Y)).
\end{equation*}
For a function $\Phi\colon W_A(\R)\to\C$, define $\Phi_Y\colon W_A(\R)\to\C$ via
\begin{equation*}
\Phi_Y(B)\defeq \Phi\Bigl(\frac{B}{Y}\Bigr).
\end{equation*}
Finally, for a $\PSO_A(\Z)$-invariant function $T\colon W_A(\Z)\to \R$, we define
\begin{equation*}
\begin{array}{rcl}
\displaystyle N_\psi(W^{(i,j)}_A,T;X)&\defeq &
\displaystyle\sum_{B\in\frac{W_A(\Z)^{(i,j)}}{\PSO_A(\Z)}}
\frac{T(B)}{|\Stab_{\PSO_A(\Z)}(B)|}\cdot\psi\{X\}(\iota(B)).
\\[.2in]
\displaystyle N^\gen_\psi(W^{(i,j)}_A,T;X)&\defeq &
\displaystyle\sum_{B\in\frac{W_A(\Z)^{(i,j),\gen}}{\PSO_A(\Z)}}
T(B)\cdot\psi\{X\}(\iota(B)).
\end{array}
\end{equation*}
Then we have the following consequence of Theorem \ref{thsmoothavggen}.
\begin{proposition}\label{propsmavgPSO}
Let $\psi\colon V_n(\R)\to\R$ be a smooth, compactly supported function. Let
$T\colon W_A(\Z)\to\R$ be an $\PSO_A(\Z)$-invariant function. Let
$\Theta\colon \PSO_A(\R)\to\R$ be a smooth, compactly supported
function. Let $\Phi$ denote $\cS(\psi,\Theta)$. For a positive real
number $X$, we have
\begin{equation*}
N^\gen_\psi(W^{(i,j)}_A,T;X)=
\frac{1}{\sigma_A(i)\Vol(\Theta)}\int_{\gamma\in\FF_A}\Bigl(
\sum_{B\in W_A(\Z)^{(i,j),\gen}} T(B)(\gamma\cdot\Phi_X)(B)
\Bigr)d\gamma.
\end{equation*}
\end{proposition}

We will have occasion to use Proposition \ref{propsmavgPSO} in the
situation detailed below. As before, let $\DA$ denote the antidiagonal
quaternary quadratic form with $1$'s in the antidiagonal entries.  Let
$\psi\colon V(\R)\to\R$ and $\Theta\colon \PSO_\DA(\R)\to\R$ be fixed smooth
compactly supported functions. Let $X$ and $Y$ be positive real
numbers with $Y\ll X^{1/2}$. Consider the function
$\psi_X[Y]\colon V(\R)\to\R$. We restrict $\psi_X[Y]$ to $V_1(\R)$,
obtaining a smooth, compactly supported function
$\psi_{X,Y} \defeq \psi_X[Y]|_1 \colon V_1(\R)\to\R$. We then denote $\cS(\psi_{X,Y},\Theta)$ by
$\Phi_{X,Y}$. The twisted functions $\psi_{X,Y}\{Y/X\}$ and
$\Phi_{X,Y}[Y/X]$ are smooth and compactly supported on elements whose
coefficients are $\ll 1$. We denote them by $\psi^{(X,Y)}$ and
$\Phi^{(X,Y)}$, respectively. Finally, parallel to \eqref{eqNW},
define
\begin{align*}
\displaystyle N_\psi^\gen(W_\DA^{(i,j)},T;X,Y)&\defeq
\displaystyle\sum_{B\in\frac{W_\DA(\Z)^{(i,j),\gen}}{\PSO_\DA(\Z)}}
T(B)\cdot(\psi_X[Y])(\iota(B))
\\&=\displaystyle
\sum_{B\in\frac{W_\DA(\Z)^{(i,j),\gen}}{\PSO_\DA(\Z)}}
T(B)\cdot(\psi^{(X,Y)}\{X/Y\})(\iota(B)).
\end{align*}

We then have the following immediate corollary to Proposition
\ref{propsmavgPSO}.
\begin{corollary} \label{cor-needproof53}
With notation as above, we have
\begin{equation*}
N^\gen_\psi(W_\DA^{(i,j)},T;X,Y)=
\frac{1}{\sigma_\DA(i)\Vol(\Theta)}\int_{\gamma\in\FF_\DA}\Bigl(
\sum_{B\in W_\DA(\Z)^{(i,j),\gen}} T(B)\cdot(\gamma\cdot\Phi^{(X,Y)}[X/Y])(B)
\Bigr)d\gamma.
\end{equation*}
\end{corollary}

\section{Counting special integral orbits via equidistribution}

For a function $T\colon \Z^n\to\R$ defined by congruence conditions,
let $T_p\colon \Z_p^n\to\R$ be the closure of $T$ to $\Z_p^n$, and
let
\begin{equation*}
\nu_p(T)\defeq \int_{\Z_p^n}T_p(v)dv,\quad
\nu(T)\defeq \int_{\widehat{\Z}^n}T(v)dv=\prod_p\nu_p(T),
\end{equation*}
where $dv$ is Euclidean measure on affine $n$-space, normalized so that
$\Z^n$ has covolume $1$ in $\R^n$.  Let $0<X$ and $0<Y\ll\sqrt{X}$ be
real numbers. The purpose of this section is to bound and estimate
$N^\gen_\psi(W_\DA^{(i,j)},T;X,Y)$ for various functions
$T\colon W_\DA(\Z)\to\R$. We will prove two such results. The first result
is an upper bound, to state which we need the following notation.
For a prime $p$, define the quantities
\begin{equation*}
    \nu(p)\defeq p^{-6},\quad \nu(p^2)\defeq p^{-12},\quad \nu(p^3)\defeq p^{-17},\quad \nu(p^k)\defeq p^{-11k/2},
\end{equation*}
for integers $k\geq 4$. For integers $n > 0$, we define $\nu(n)$ multiplicatively: $\nu(n):=\prod_{p^k\|n}\nu(p^k)$.
Then we have:

\begin{theorem}\label{thSLMT}
Let $X$ and $Y$ be as above, and let $a$ be an integer such that $a\ll
X^{1/4}Y^{1/2}$. Let $T_a\colon W_\DA\to\R$ be the characteristic function
of the set of non-distinguished elements $(\DA,B)\in W_\DA(\Z)$ that
are special at $a$. Let $\psi\colon V(\R)\to\R$ be a smooth and compactly
supported function. Then we have that
\begin{equation*}
N_\psi(W_\DA,T_a;X,Y)\ll_\epsilon \nu(a)X^{5-1/4+\epsilon}Y^{1/2}.
\end{equation*}
\end{theorem}

The above result has the following immediate consequence when $Y\ll
X^{1/2-\delta}$.

\begin{corollary}\label{coravgSO}
Let notation be as in Theorem~\ref{thSLMT} with the added assumption
that $Y\ll X^{1/2-\delta}$ for some
$\delta>0$. Then we have that
\begin{equation*}
N_\psi(W_\DA,T_a;X,Y)\ll_\epsilon \nu(a)X^{5-\delta/2+\epsilon}.
\end{equation*}
\end{corollary}

A function $T\colon \Z^n\to\R$ is said to be {\it defined by finitely many
  congruence conditions} if there exists an integer $q>0$ and a
function $T_q\colon (\Z/q\Z)^n\to\R$ such that $T$ is the lift of $T_q$ to
$\Z^n$. Let $M_1$ and $M_2$ be positive real numbers. A function
$T\colon W_\DA(\Z)\to\R$ is said to be {\it $(M_1,M_2)$-acceptable} if $T$ is
bounded and defined by congruence conditions modulo two relatively prime positive
integers $b$ and $q$, where $b\ll M_1$, $q\ll M_2$, $q$ is squarefree, and the congruence conditions modulo every prime $p$ dividing $q$ is
exactly that $T_p\colon W_\DA(\F_p)\to\R$ is the characteristic function of
the set of quaternary quadratic forms that are of rank $\leq 1$ modulo $p$.
We next use equidistribution methods to prove an estimate on $N^\gen_\psi(W_\DA^{(i,j)},T;X,Y)$ when $Y$ is close to $X^{1/2}$.
\begin{theorem}\label{thmainWA}
Let $\theta$ and $\delta$ be positive real numbers, which will
eventually be picked to be sufficiently small. Let $X$ and $Y\gg
X^{1/2-\delta}$ be positive real numbers. Let $T\colon W_\DA(\Z)\to\R$ be an
$(X^\theta,X^{1/2})$-acceptable function, with conditions defined modulo $b \leq X^\theta$ and $q \leq X^{1/2}$. Then we have that
\begin{equation*}
N^\gen_\psi(W_\DA^{(i,j)},T;X,Y)=
\frac{|\J_\DA|}{\sigma_\DA(i)}\Vol(\FF_A)\nu(T)
\Vol\bigl(\psi|_{\frac{Y^2}{X}}\bigr)X^4Y^2
+O\Bigl(\frac{X^{5+10\theta+60\delta-\lambda+\epsilon}}{q^6}\Bigr)
\end{equation*}
where $\lambda > 0$ is an absolute constant independent of
$\theta$ and $\delta$. Here, $\J_\DA$ is the rational constant arising
from Corollary~\ref{corjacinfty} applied to the action of $\PSO_\DA$ on $W_\DA$.
\end{theorem}

\subsection{Bounding the number of special $\SL_4(\Z)$-orbits on $W(\Z)$ having bounded height}

Let $\cD$ be a bounded ball inside $U(\R)$, the space of real
quaternary quadratic forms. Let $\gamma=(n,s,k)\in \SL_4(\R)$ be an
element where $s=(s_1,s_2,s_3)$, and let $Z>1$ be a real
number. For any $L\subset U(\Z)$, consider the set
\begin{equation*}
\cD(L;\gamma,Z)=\bigl\{\gamma\cdot Z\cD\cap L\bigr\}.
\end{equation*}
We will bound the number of elements in $\cD(L;\gamma,Z)$ for various
sets $L$. We begin with the following lemma:

\begin{lemma}\label{lemb1234}
Let $q=p^k$ be a prime power, and let $L_q^{*} \subset U(\Z/q\Z)$ denote the set of quaternary quadratic forms that are squares of linear forms modulo $q$. Let $(a_{11},a_{12},a_{13},a_{14})\in(\Z/q\Z)^4$ be a fixed tuple.
\begin{itemize}[leftmargin=1cm,itemsep=0pt]
\item[{\rm (a)}] Suppose that $a_{11}\not\equiv 0 \pmod{q}$ and that $p^\beta\parallel a_{11}$.
  Then there are no elements $B\in L_q^{*}$ with $b_{1i}=a_{1i}$ unless $\beta=2\ell$ is even and $p^\ell\mid a_{1i}$
  for $i \in\{2,3,4\}$, in which case
  $a_{11}$, $a_{12}$, $a_{13}$, and $a_{14}$ determine $B$ up to
  $O(p^{5\ell})$ choices.
\item[{\rm (b)}] Suppose that $a_{11}\equiv0\pmod{q}$, and set $p^\ell={\rm
  gcd}(a_{12},a_{13},a_{14})$. Then there are no elements $B\in L_q^{*}$
  with $b_{1i}=a_{1i}$ unless $2\ell\geq k$, in which case $a_{11}$, $a_{12}$, $a_{13}$, and $a_{14}$
  determine $B$ up to $O(p^{k+2\ell})$ choices.
\end{itemize}
\end{lemma}
\begin{proof}
Any element $B \in L_q^{*}$ can be expressed as $B = (\alpha_1 x_1+\alpha_2 x_2+\alpha_3 x_3+\alpha_4 x_4)^2$ for $\alpha_i\in\Z/p^k\Z$. The coefficient $b_{11}$ of $B$ then satisfies
$b_{11}\equiv\alpha_1^2\pmod{p^k}$. Assume that $a_{11}\not\equiv
0 \pmod{p^k}$. Clearly, the power of $p$ dividing $a_{11}$ must be even, say $2\ell$, for $a_{11}$ to occur as the leading coefficient of $B\in
L_q$. Then the congruence $\alpha_1^2\equiv a_{11}\pmod {p^k}$ determines $\alpha_1\in \Z/p^k\Z$ up to $O(p^{2\ell})$ choices, and moreover $p^\ell\parallel \alpha_1$. Next, for $j\neq 1$, we consider the congruences $2\alpha_1\alpha_j\equiv a_{1j}\pmod{p^k}$. This
implies that $p^\ell\mid a_{1j}$ for $j\in\{2,3,4\}$. Moreover, since
$p^\ell\parallel\alpha_1$, the value of $a_{1j}$ determines $\alpha_j$ up to $O(p^\ell)$ possibilities. Thus, the values of $a_{11}$,
$a_{12}$, $a_{13}$, and $a_{14}$, determine $\alpha_1$, $\alpha_2$,
$\alpha_3$, and $\alpha_4$, and hence determine $B$, up to
$O(p^{5\ell})$ possibilities.

We move on to part (b) of the lemma. Since $a_{11}\equiv0\pmod{p^k}$, we must have
$p^\ell\parallel \alpha_1$, with $2\ell\geq k$, and since
$2\alpha_1\alpha_j\equiv a_{1j} \pmod{p^k}$, it follows that $p^\ell\mid a_{1j}$ for
$j\in\{2,3,4\}$. Therefore, $\alpha_1$ is determined up to
$O(p^{k-\ell})$ choices, and as before, $\alpha_2$, $\alpha_3$, and
$\alpha_4$ are determined up to $O(p^\ell)$ choices each. This gives a
total of $O(p^{k+2\ell})$ choices for the $\alpha_i$, yielding the
result.
\end{proof}

The lemma has the following application to the important case when $q$ is a squarefree integer:
\begin{corollary}\label{corb1234}
Let $q$ be a squarefree positive integer, and let $L_q\subset U(\Z/q\Z)$ denote the set of quaternary quadratic forms that are of rank $\leq 1$ modulo
$q$. Let $q' \mid q$, and let $(a_{11},a_{12},a_{13},a_{14})\in(\Z/q\Z)^4$ be a fixed
tuple such that $a_{11} \equiv 0 \pmod{q'}$. Then there are no elements $B\in
  L_q$ with coefficients $b_{1i}=a_{1i}$ unless $q'\mid a_{1j}$ for
  $j\in\{2,3,4\}$, in which case $a_{11}$, $a_{12}$, $a_{13}$, and $a_{14}$ determine $B$ up to
  $O_\epsilon({q'}^3q^{\epsilon})$ choices.
\end{corollary}
\begin{proof}
The number of $B \in L_q$ with $b_{1j} = a_{1j}$ for $j \in \{1,2,3,4\}$ is at most the number of $B \in L_q^*$ such that for some $\gamma \in (\Z/q\Z)^{\times}/(\Z/q\Z)^{\times 2}$ we have $b_{1j} = \gamma a_{1j}$ for $j \in \{1,2,3,4\}$. Since $|(\Z/q\Z)^{\times}/(\Z/q\Z)^{\times 2}| \ll_\epsilon q^{\epsilon}$, it suffices to prove the lemma with $L_q$ replaced by $L_q^*$. The result then follows by applying Lemma~\ref{lemb1234} along with the Chinese remainder theorem.
\end{proof}

Next we have the following lemma describing conditions on the
coefficients of $B$ imposed by $B$ being rank $\leq 2$ modulo a prime
power.
\begin{lemma}\label{lemrank2}
Let $q=p^k$ be a prime power, and let
$(a_{11},a_{12},a_{13},a_{14})\in(\Z/q\Z)^4$ be
fixed and nonzero. Then the condition that $B\in U(\Z/q\Z)$ satisfies
$b_{ij}=a_{1j}$ for $j\in\{1,2,3,4\}$ and to be rank $\leq 2$ modulo $q$ forces $B$ to satisfy two independent conditions, one of density $O(p^{-k})$, and another of density $O(p^{-1})$.
\end{lemma}
\begin{proof}
The first condition is simply that the $3\times 3$ matrix
$(b_{ij})_{2\leq i,j\leq 4}$ has determinant $0$. When the
tuple $(b_{11},b_{12},b_{13},b_{14})$ is nonzero, the first condition is insufficient to ensure that $B$ is of rank $\leq 2$. Hence, there must be a second, necessarily independent, condition. The largest possible density for this condition is $O(1/p)$.
\end{proof}

Let $a,b > 0$ be integers. Let $\L_a$ be the set of elements in $U(\Z)$ of rank $\leq 1$ modulo $a$, and let
$\L_{a,b}\subset \L_a$ be the subset of element of rank $\leq 2$ modulo $ab$.
Define the quantity $\mu(a,b)$ tby
\begin{equation}
\mu(a,b):=a^{-6}b^{-1}b_1^{-1},
\end{equation}
where $b_1$ is the product of primes of $b$ not dividing $a$.
For any subset
$S\subset U(\Z)$ and two subsets $\var_1$, $\var_2$ of coefficients
of $U$, let $S(\var_1;\var_2) \subset S$ be the subset of elements
in which all coefficients in $\var_1$ vanish and no coefficients in $\var_2$ vanish. Let $U_1 \subset U$ be the determinant-$1$ subvariety. Then we have the
following result:

\begin{lemma}\label{lemSL4cases}
Let $\gamma\in\FF_4$ be such that $\gamma=n(s_1,s_2,s_3)k$. We define the quantity
\begin{equation*}
M(s)\defeq M(s_1,s_2,s_3)\defeq \max(1,s_1^2s_2^2s_3^{-2},s_1^4s_2^2s_3^{-4},s_1^6s_3^{-6}).
\end{equation*}
Let $a$ and $b$ be positive integers, and let $Z$ be a real
number satisfying $ab\ll Z$.  Then estimates on $\#\cD(L;\gamma,Z)$ for
the following choices of $L$ are as in the second column of the
following table. Moreover, in order for $\cD(L;\gamma,Z)$ to be
nonempty, the condition listed in the third column must be true.
\begin{table}[ht]
\centering
\begin{tabular}{|c | c| c|c|}
\hline
$\#$& $L$ & $\# \cD(L;\gamma,Z)\ll $ &
$\cD(L;\gamma,Z)\neq\emptyset$ implies\\
\hline\hline
\bf{1} & $U_1(\Z)(\emptyset;\{a_{11}\})$
& $s_1^{-2}s_3^{-2} Z^9$
& $s_1^6s_2^2s_3^2\ll Z$ \\[.1in]
\bf{2} & $U_1(\Z)(\{a_{11}\};\{a_{12}\})$
&  $s_1^{4}s_2^{2} Z^8$
& $s_1^2s_2^2s_3^2\ll Z$
\\[.1in]
\bf{3} & $U_1(\Z)(\{a_{11},a_{12}\};\{a_{13},a_{22}\})$
& $s_1^{6}s_2^{4}s_3^2 Z^7$
& $s_1^2s_3^2\ll Z,\;\;\;s_2^2s_3^2\ll s_1^2 Z$
\\[.1in]
\bf{4} & $U_1(\Z)(\{a_{11},a_{12},a_{22}\};\{a_{13}\})$
& $s_1^4s_2^6s_3^4Z^6$
& $s_1^2s_3^2\ll Z$
\\[.1in]
\bf{5} & $U_1(\Z)(\{a_{11},a_{12},a_{13}\};\{a_{22}\})$
& $s_1^{10}s_2^2s_3^6 Z^5$
& $s_1^2\ll s_3^2 Z,\;\; s_2^2s_3^2\ll s_1^2 Z$
\\[.1in]
\bf{6} & $U_1(\Z)(\{a_{11},a_{12},a_{13},a_{22}\};\emptyset)$
& $s_1^8s_2^6s_3^8 Z^4$
& $s_1^2\ll s_3^2 Z,\;\; s_3^2\ll s_1^2 Z$
\\[.1in]
\hline
\hline
\bf{a} & $\L_{a,b}(\emptyset;\{b_{11}\})$
& $\mu(a,b)s_1^{-6}s_3^6M(s)Z^{10+\epsilon}$
& $s_1^6s_2^2s_3^2\ll Z$ \\[.1in]
\bf{b} & $\L_{a,b}(\{b_{11}\};\{b_{12}\})$
& $\sqrt{a}\mu(a,b)s_2^{2}s_3^8 M(s)Z^{9+\epsilon}$
& $s_1^2s_2^2s_3^2\ll Z$ \\[.1in]
\bf{c} & $\L_{a,b}(\{b_{11},b_{12}\};\emptyset)$
& $a\mu(a,b)s_1^4s_2^4s_3^{12}M(s)Z^{8+\epsilon}$
&  \\[.1in]
\hline
\end{tabular}
\end{table}\label{tabcases}
\end{lemma}
\begin{proof}
First, note that the claims in the fourth column of the table
follow immediately from the nonvanishing conditions stipulated in the
second column. Hence it remains to prove the claims \mbox{in the third
column.}

We begin by proving the claims in the first six rows of the table. It will be convenient to divide into two cases: first
assume that not all of $a_{11}$, $a_{12}$, and $a_{13}$ are $0$. In
this case, either $a_{24}$, $a_{34}$, or $a_{44}$ must arise in the
determinant polynomial expansion of $\det(A)$ with nonzero
coefficient. We then fiber over all but this one coefficient, and note
that the condition $\det(A)=1$ determines this coefficient uniquely
(if it is integral at all). With this observation, Cases {\bf 1}
through {\bf 4} follow immediately. To illustrate how, we briefly
discuss Case {\bf 3}. Here $a_{11}$ and $a_{12}$ are set to $0$. Since
$a_{13}$ and $a_{22}$ are assumed to be nonzero, for there to exist
integral points at all, it must be the case that the ranges
$Zw(a_{13})$ and $Zw(a_{22})$ of $a_{13}$ and $a_{22}$, respectively,
are $\gg 1$. This implies that the range $Zw(a_{ij})\gg 1$ for all
$(i,j)$ not equal to $(1,1)$ or $(1,2)$. Hence, the number of
possibilities for all the coefficients $a_{ij}$, other than $a_{24}$
is $\ll$
\begin{equation*}
\prod_{(i,j)\not\in\{(1,1),(1,2),(2,4)\}} Zw(a_{ij})=s_1^6s_2^4s_3^2Z^7,
\end{equation*}
as necessary. (When $a_{34}$ or $a_{44}$ is the excluded coefficient
instead of $a_{24}$, the computation is similar and produces a better
bound since we have $s_i\gg 1$.) This yields Case {\bf 3}, and Cases
{\bf 1}, {\bf 2}, and {\bf 4} are similar.

Next, assume that $a_{11}=a_{12}=a_{13}=0$. Then for $\det A$ to be nonzero, we must have $a_{14}=\pm1$, and so
\begin{equation}
1=\det A =\det\left(\begin{array}{cc} a_{22} & a_{23}\\a_{23} & a_{33}
\end{array}\right).
\end{equation}
The coefficients $a_{24}$, $a_{34}$, and $a_{44}$ run freely and do
not affect $\det A$. When $a_{22}$ is nonzero, as in Case {\bf 5}, we
fiber over $a_{22}$ and $a_{23}$, and note that $a_{33}$ is then
determined by the determinant condition to yield the result. When
$a_{22}=0$, as in Case {\bf 6}, we must have $a_{23}=\pm 1$. Then we
fiber over $a_{23}$ to obtain the result.

\medskip

We now turn out attention to the bottom part of the table. In each
case, we will fiber over the coefficients $b_{11}$, $b_{12}$,
$b_{13}$, and $b_{14}$, and use Lemma~\ref{lemb1234} to determine the
number of possibilities of $B$ modulo $a$ given these
coefficients. Then, for each fixed $b_{11}$, $b_{12}$, $b_{13}$,
$b_{14}$, and $\bar{B}\pmod{a}$, we count the number of possibilities
for $B$ as follows: Let $d = \gcd(b_{11},b_{12},b_{13},b_{14},b_1)$, where $b_1$ is the product of primes dividing $b$ which do not divide $a$. From Lemma \ref{lemrank2} it follows that the density of
possible
$b_{22},b_{23},b_{24},b_{33},b_{34},b_{44}$ is $\ll b^{-1}b_1^{-1}d$. It follows that the
number of choices for the possible values $B$ is $\ll$
\begin{equation*}
    \prod_{(i,j):i\geq 2}\max
\Bigl(\frac{Zw(b_{ij})}{f_{ij}},1\Bigr),
\end{equation*}
where $f_{ij}$ are factors one of which is $ab$, one of which is $ab_1/d$, and the rest are $a$. Since $ab\ll Z$, it follows that
\begin{equation}\label{eqSLboundtemp1}
\begin{array}{rcl}
\displaystyle \prod_{(i,j):i\geq 2}\max
\Bigl(\frac{Zw(b_{ij})}{f_{ij}},1\Bigr)&\ll_\epsilon&
\displaystyle Z^{6+\epsilon}a^{-6}b^{-1}b_1^{-1}d s_1^6s_2^4s_3^{10}M(s)
\\[.1in]&=&\displaystyle
d\mu(a,b)Z^{6+\epsilon}s_1^6s_2^4s_3^{10}M(s).
\end{array}
\end{equation}
Note that the $d/b_1$-density condition must necessarily be independent of the $\operatorname{mod}{a}$ densities arising from the fact that $B\pmod{a}$ is fixed since $\gcd(a,b_1)=1$.

Let $\Sigma_a$ be the set of primes dividing $a$. For a subset $S \subset \Sigma_a$, let $\ol{S} \defeq \Sigma_a \smallsetminus S$, let $\Pi(S)$ be the set of positive integers divisible only by primes in $S$, and let $a_S$ be the part of $a$ coprime to every element of $S$. Turning now to Case {\bf a}, Lemma~\ref{lemb1234} implies that the number of choices for $b_{11}$, $b_{12}$, $b_{13}$, and $b_{14}$, weighted by
the number of arising possibilities for $B\pmod{a}$, such that $\gcd(b_{11},b_{12},b_{13},b_{14},b_1)=d$ is $\ll_\epsilon$
\begin{equation*}
\sum_{S \subset \Sigma_a} \sum_{\substack{q \in \Pi(\ol{S}) \\ q^2\mid a}}\sum_{\substack{q' \in \Pi(S) \\ q' \mid a_{\ol{S}} \mid {q'}^2}}\sum_{\substack{b_{1j}\ll Zw(b_{1j})\\dq^2a_{\ol{S}}\mid b_{11},dqq'\mid b_{1j}}}q^5a_{\ol{S}}{q'}^2\ll_\epsilon
\frac{Z^{4+\epsilon}}{d^4}w(b_{11})w(b_{12})w(b_{13})w(b_{14})=\frac{Z^{4+\epsilon}}{d^4} s_1^{-12}s_2^{-4}s_3^{-4}.
\end{equation*}
Multiplying by the right-hand side of~\eqref{eqSLboundtemp1} and summing over $d\mid b_1$ yields the result for Case {\bf a}.

Similarly, the numbers of choices for $b_{11}$, $b_{12}$,
$b_{13}$, and $b_{14}$ with $\gcd(b_{11},b_{12},b_{13},b_{14},b_1)=d$, weighted by the number of arising
possibilities for $A\pmod{a}$, in Cases {\bf b}, {\bf c} are
respectively $\ll_\epsilon$
\begin{equation*}
\begin{array}{rcccl}
\displaystyle \sum_{S \subset \Sigma_a} \sum_{\substack{q' \in \Pi(S) \\ q'\mid a_S\mid {q'}^2}}\sum_{\substack{db_{12},db_{13},db_{14}\\q'\mid b_{1j}}}a_S{q'}^2
&\ll_\epsilon &\displaystyle \frac{Z^{3+\epsilon}}{d^3}
\frac{w(b_{12})w(b_{13})w(b_{14})}{{q'}^3}a{q'}^2
&\ll_\epsilon &\displaystyle \frac{Z^{3+\epsilon}}{d^3}a^{1/2}s_1^{-6}s_2^{-2}s_3^{-2},
\\[.2in]
\displaystyle \sum_{S \subset \Sigma_a} \sum_{\substack{q' \in \Pi(S) \\ q'\mid a_S\mid {q'}^2}}\sum_{\substack{db_{13},db_{14}\\{q'}\mid b_{1j}}}a_S{q'}^2
&\ll_\epsilon &\displaystyle \frac{Z^{1+\epsilon}}{d}\max\left\{\frac{Zw(b_{13})}{q'},1\right\}
\frac{w(b_{14})}{q'}a{q'}^2
&\ll_\epsilon &\displaystyle \frac{Z^{2+\epsilon}}{d}a s_1^{-2}s_3^2.
\end{array}
\end{equation*}
Once again, multiplying by the right-hand side of~\eqref{eqSLboundtemp1} and summing over $d\mid b_1$ yields the result for Cases {\bf b} and {\bf c}. This concludes the proof of the lemma.
\end{proof}

For positive real numbers $X$ and $Y$, and $\gamma\in\SL_4(\R)$, we
define the quantity
\begin{equation*}
N(\gamma;X,Y)\defeq \#\bigr\{
(A,B)\in\cD(U_1(\Z);\gamma,X)\times\cD(\L_{a,b};\gamma,Y):(A,B)
\mbox{ is non-distinguished}\bigl\}.
\end{equation*}
We now have the following result.
\begin{proposition}\label{propSL4keybound}
Let $Y\gg X\geq 1$ be real numbers, and fix
$\gamma=n(s_1,s_2,s_3)k\in\FF_4$. Let $a$ and $b$ be positive integers such that $ab\ll Y$.
Then we have that
\begin{equation*}
  N(\gamma;X,Y)\ll_\epsilon \mu(a,b)X^9Y^{10+\epsilon}s_1^{12}s_2^8s_3^{12}.
\end{equation*}
\end{proposition}
\begin{proof}
We partition $U_1(\Z)$ into the first six sets listed in the second
column of Table \ref{tabcases}, and partition $\L_{a,b}$ into the last
three sets listed in the second column of Table \ref{tabcases}. This
gives us 18 cases to consider, and we label these cases {\bf 1a}
through {\bf 6c}. A pair $(A,B)\in W(\Z)$ with $\det A = 1$ and $a_{11}=a_{12}=a_{22}=b_{11}=b_{12}=0$ is automatically distinguished, and so
Cases {\bf 4c} and {\bf 6c} can be discarded.

Continuing with the {\bf c}'s, note that the bound for Case {\bf 1c} follows directly from the relevant third column bounds of Lemma~\ref{lemSL4cases} and the inequality $a \ll Y$. The bound for Case {\bf 2c} follows directly from the relevant third column bounds, along with the relevant fourth column inequality and the inequality $a \ll Y$. The bound for Case {\bf 3c} follows directly from the relevant third column bounds, fourth column inequalities, and the inequalities $X \ll Y$ and $a \ll Y$. For Case {\bf 5c}, we observe that any pair $(A,B) \in W(\Z)$ with $\det A = 1$ and $a_{11} = a_{12} = a_{13} = b_{11} = b_{12} = b_{13} = 0$ is automatically distinguished. This observation has two consequences: first, we may assume that $s_1^2s_3^2 \ll Y$, and second, the bound for the number of eligible forms $B$ given in Lemma~\ref{lemSL4cases} can be improved to $ a\mu(a,b)s_1^2s_2^4s_3^{10}M(s)Y^{8+\epsilon}$. Using these consequences along with the relevant fourth column inequalities and the inequalities $a \ll Y$ and $s_1^2s_3^2 \ll Y$ yields Case {\bf 5c}.

We now move along to the {\bf a}'s and {\bf b}'s. It is easy to check
that Cases {\bf 1a}, {\bf 2a}, {\bf 3a}, {\bf 4a}, {\bf 5a}, {\bf 1b},
{\bf 2b}, and {\bf 3b} follow directly from the bounds in the second
column of Table \ref{tabcases}. Case {\bf 6a} follows from the bound
in the second column together with the inequality $s_3^2\ll s_1^2
X$. Similarly, Case {\bf 5b} follows from the bound in the second
column together with the inequalities $s_1^2\ll s_3^2 X$ and $s_2^2s_3^2\ll
s_1^2 X$, which should be used to make the powers of $s_1$ and $s_3$
equal in the product $\sqrt{a}\mu(a,b)X^5Y^9s_1^{10}s_2^{4}s_3^{14}M(s)$.

This only leaves Case {\bf 6b}, which is the most intricate case
stemming from the difficulty in controlling the exponent of $s_2$. We
proceed as follows: Since $a_{11}=a_{12}=a_{22}=0$, it follows that
the quadratic form $Q(x,y)=b_{11}x^2+b_{12}xy+b_{22}y^2$ must have
nonzero discriminant for the pair $(A,B)$ to be irreducible. However,
since $B\in \L_a$, it follows that $a\mid \Delta(Q)$, and in particular
$|\Delta(Q)|\geq a$. Therefore, it follows that
$Y^2w(b_{12})^2=Y^2s_1^{-4}s_2^{-4}s_3^{-4}\gg a$, or
$s_1^2s_2^2s_3^2\ll Ya^{-1/2}$. Denote the set
\begin{equation*}
\cD(U_1(\Z)(\{a_{11},a_{12},a_{13},a_{22}\};\emptyset);\gamma,X)
\times\cD(\L_{a,b}(\{a_{11}\},\{a_{12}\});\gamma,Y)
\end{equation*}
by $T$.  Now, from the bounds in the third column of Table
\ref{tabcases}, it follows that we have
\begin{equation*}
\begin{array}{rcl}
\displaystyle\#\bigl\{(A,B)\in  T: (A,B)\mbox{ is non-dist.}\bigr\}
&\ll& \sqrt{a}\mu(a,b)X^4Y^9s_1^8s_2^8s_3^{16}M(s)\\[.1in]
&\ll& \mu(a,b)X^4Y^{10}s_1^6s_2^6s_3^{14}M(s),
\end{array}
\end{equation*}
where the final inequality follows by multiplying the right-hand side
of the first line by
$Ya^{-1/2}s_1^{-2}s_2^{-2}s_3^{-2}\gg 1$. Using the inequalities
$s_1^2\ll Xs_3^2$ and $s_3^2\ll Xs_1^2$ to ensure that the exponents
of $s_1$ and $s_3$ in $s_1^6s_3^{14}M(s)$ are equal, we obtain the
required bound in Case {\bf 6b} as well. This concludes the proof of
the proposition.
\end{proof}

We are now ready to prove Theorem \ref{thSLMT}:

\medskip

\noindent {\bf Proof of Theorem \ref{thSLMT}:}
Let $T$ denote the characteristic function of the set of non-distinguished elements $(A,B)\in W(\Z)$ that are special at $a$ and satisfy $\det A =1$. Then we clearly have
\begin{equation*}
N_\psi(W^{(i,j)}_\DA,T_a;X,Y)\leq N_\psi(W^{(i,j)},T;X,Y).
\end{equation*}
From Proposition \ref{propSL4avg}, it follows that
we have
\begin{align*}
& N_\psi(W^{(i,j)},T;X,Y) \ll\displaystyle
\int_{\gamma\in\FF_4}\Bigl(\sum_{w\in W(\Z)}T(w)\bigl(
\gamma\cdot\Phi[Y^{-1/2}X^{1/4},Y^{1/2}X^{1/4}]\bigr)(w)\Bigr)d\gamma \ll
\\ \displaystyle
& \qquad \int_{\gamma\in\FF_4}\#\bigl\{(A,B)\in\cD\big(U_1(\Z);\gamma,\tfrac{X^{1/4}}{\sqrt{Y}}\big)
\times\cD(\cS_a;\gamma,\sqrt{Y}X^{1/4})\cap W(\Z) : (A,B) \text{ is non-dist.}\bigr\}d\gamma,
\end{align*}
as long as $\cD\subset U(\R)$ is large enough that $\cD\times\cD$
contains the support of the compactly supported function $\Phi$, and where
$\cS_a$ denotes the set of elements $B\in U(\Z)$ that are special at $a$.

By Proposition~\ref{prop-whatspecialmeans}, we may write
\begin{equation*}
    \cS_a=\bigcup_{b,c,d} b\L_{c,d},
\end{equation*}
where $b$, $c$, and $d$, are integers dividing $a$ that satisfy the following property: for each prime $p$ dividing $a$ with $p^k\|a$, there exist nonnegative integers $\alpha$ and $\beta$ such that $p^\alpha\| b$, $p^{k-2\alpha+\beta}\| c$, $p^{\alpha-2\beta}\|d$, $k-2\alpha+\beta\geq 0$, and $\alpha\geq 2\beta$.
Applying Proposition \ref{propSL4keybound}, on each triple $(b,c,d)$ (to
apply the proposition for the set $b\L_{c,d}$ rather than $\L_{c,d}$, we simply
divide each $B$ by $b$), we therefore obtain
\begin{equation}\label{eq51final}
\begin{array}{rcl}
\displaystyle N_\psi(W^{(i,j)}_\DA,T_a;X,Y)&\ll_\epsilon&
\displaystyle\sum_{b,c,d}\int_{\substack{s_1,s_2,s_3\gg 1\\s_i\ll X}}
b^{-10}\mu(c,d)(Y^{-1/2}X^{1/4})^9(Y^{1/2}X^{1/4})^{10+\epsilon}d^\times s
\\[.2in]&\ll_\epsilon &\displaystyle
\sum_{b,c,d} b^{-10}\mu(c,d)X^{5-1/4+\epsilon}Y^{1/2},
\end{array}
\end{equation}
since the term $s_1^{12}s_2^8s_3^{12}$ in Proposition
\ref{propSL4keybound} exactly cancels out with the Haar-measure
factor. If $p^k\|a$, then the contribution from $p$ to $b^{-10}\mu(c,d)$ is easily checked to be
\begin{equation*}
\left\{
\begin{array}{rcl}
\displaystyle p^{-6k+\alpha-4\beta+1}, &\mbox{if} \text{  }k-2\alpha+\beta=0,\;\alpha-2\beta>0;\\[.1in]
\displaystyle p^{-6k+\alpha-4\beta}, &\mbox{otherwise}.
\end{array}
\right.
\end{equation*}
We claim that the contribution from $p$ to $b^{10}\mu(c,d)$ is $\leq$ the contribution of $p$ to $\nu(a)$ for each $k\geq 1$.
Since $\alpha\leq k/2+\beta/2$, this follows  when $k\geq 4$. When $k=1$, we must have $\alpha=\beta=0$, and the claim follows. When $k=2$ or $k=3$, we either have $\alpha=\beta=0$, or $\alpha=1$ and $\beta=0$.
The claim is easily checked in all these cases.
Therefore, we have $b^{10}\mu(c,d)\leq\nu(a)$.
Finally, since the number of triples $(b,c,d)$ being summed over is $O_\epsilon(a^\epsilon)=O_\epsilon(X^\epsilon)$, the result follows from \eqref{eq51final}. $\Box$

\subsection{Fourier analysis on the space of $n$-ary quadratic forms}

In this section, we collect some Fourier analytic results on the space
of quadratic forms. Since our results apply generally, we will work
with the space of $n$-ary quadratic forms, for integers $n\geq
2$, specializing only occasionally to the space of quaternary
quadratic forms.

\subsubsection{Preliminary definitions and results}

Let $n\geq 2$ be an integer, and as in \S\ref{sec-reps}, let $U\subset \on{Mat}_n$ denote the subscheme over $\Z$ whose $R$-points are given by $\Sym_2 R^n$ for any $\Z$-algebra $R$ (i.e., $n \times n$ symmetric matrices with entries in $R$, or equivalently, classically integral $n$-ary quadratic forms over $R$). For an integer $N$, let $\widehat{U(\Z/N\Z)}$
be the Pontryagin dual of $U(\Z/N\Z)$. Then
$\widehat{U(\Z/N\Z)}$ can be identified with $U(\Z/N\Z)=(\Z/N\Z)^{n}$,
where $D\defeq n(n+1)/2$. We write elements $\chi \in\widehat{U(\Z/N\Z)}$ as
$D$-tuples $\chi=(\check{a}_{ij})_{1\leq i\leq j\leq n}$ via the
identification
\begin{equation*}
  \chi(B)\defeq e\Bigl(\frac{\sum_{i,j} \check{a}_{ij}b_{ij}}{N}\Bigr)=
  \prod_{i,j}e\Bigl(\frac{\check{a}_{ij}b_{ij}}{N}\Bigr),
\end{equation*}
where $B=(b_{ij})_{i,j}\in U(\Z/N\Z)$ and $e(x)\defeq e^{2\pi i x}$. We
also consider elements $\check{A}=(\check{a}_{ij})\in
\widehat{U(\Z/N\Z)}$ as $n$-ary quadratic forms, by regarding
$\check{a}_{ij}$ as the coefficients of a $n\times n$ symmetric
matrix.

For a function $\phi\colon U(\Z/N\Z)\to\C$, we recall the
definition of $\widehat{\phi}\colon \widehat{U(\Z/N\Z)}\to\C$, the Fourier
dual of $\phi$:
\begin{equation*}
\widehat{\phi}(\chi)= \displaystyle\sum_{B\in U(\Z/N\Z)}\phi(B)\chi(B).
\end{equation*}
Fourier duality then implies the equality
\begin{equation*}
\frac{1}{N^{D}}\sum_{\chi}\widehat{\phi}(\chi)\overline{\chi(f)}=\phi(f).
\end{equation*}

Let $\widehat{U(\Z)}$ denote the lattice in $U(\R)$ dual to
$U(\Z)$. For every positive integer $N$, we have a natural
identification of $\widehat{U(\Z)}/N\widehat{U(\Z)}$ with
$\widehat{U(\Z/N\Z)}$. The following result is a consequence of
Poisson summation.
\begin{proposition} \label{prop-twistedpoisson}
Let $\psi\colon U(\R)\to\R$ denote a smooth function with super polynomial
decay. Let $N$ be a positive integer, and let $\phi\colon U(\Z/N\Z)\to\R$ be
any function. Then, for positive real numbers $Y_{ij}$, we have
\begin{equation*}
\sum_{B=(b_{ij})\in U(\Z/N\Z)}\psi\Bigl(\frac{b_{ij}}{Y_{ij}}\Bigr)\phi(B)
=\frac{\prod_{i,j}Y_{ij}}{N^{D}}\sum_{\chi=(\check{a}_{ij})\in \widehat{U(\Z)}}
\widehat{\psi}\Bigl(\frac{Y_{ij}\check{a}_{ij}}{N} \Bigr)\widehat{\phi}(\chi).
\end{equation*}
\end{proposition}

\subsubsection{The Fourier transform of the set of rank-$1$
  quadratic forms modulo a prime $p$}

Let $p$ be an odd prime number. In the following result, we estimate the Fourier transform of
the characteristic function $S\colon U(\Z/p\Z)\to\{0,1\}$ of the set of elements in $U(\Z/p\Z)$ that are of rank $\leq 1$ modulo $p$.
\begin{proposition} \label{prop-fourierrank1bound}
With notation as above, we have the following bounds on $\widehat{S}$:
\begin{equation*}
|\widehat{S}(\chi)|\ll p^{n-\frac{{\rm rank}(\chi)}{2}}.
\end{equation*}
\end{proposition}
\begin{proof}
Let $\chi=(\check{a})_{ij}$ correspond to the quaternary quadratic
form $\check{A}$. Since $S$ is $\GL_n(\F_p)$-invariant, it follows
that for any $\gamma\in\GL_n(\F_p)$, we have
$\widehat{S}(\chi)=\widehat{S}(\gamma\chi)$. Thus, we may assume that
$\check{A}$ is diagonal, with coefficients
$\alpha_1,\ldots,\alpha_n$. The support of $S$ is precisely
the set
$$
\{B=\gamma(\beta_1x_1+\cdots+\beta_n x_n)^2:\gamma \in \F_p^\times/\F_p^{\times 2},\beta_1,\dots,\beta_n\in\F_p\}.
$$
With $\chi$ and $B$ as above, we have $\chi(B)=e\Bigl(\frac{\sum
  \alpha_i\gamma\beta_i^2}{p}\Bigr)$. Therefore, we have that
  \begin{equation*}
\displaystyle\widehat{S}(\chi) \ll\displaystyle\sum_{\beta_i\in\F_p}e\Bigl(\frac{\sum
  \alpha_i\beta_i^2}{p}\Bigr)=\displaystyle
\prod_{i=1}^n\sum_{\beta_i\in\F_p}e\Bigl(\frac{\alpha_i\beta_i^2}{p}\Bigr)
=\displaystyle\prod_{i=1}^n g(\alpha_i,p),
\end{equation*}
where $g(\alpha_i,p)$ denotes the quadratic Gauss sum modulo
$p$. Since we know that $|g(\alpha_i,p)|$ is $\sqrt{p}$ when
$\alpha_i\neq 0$ and $p$ when $\alpha=0$, the result follows.
\end{proof}

\subsubsection*{Bounding the Fourier transform of certain smooth functions}

Let $\psi\colon V(\R)\to\R$ and $\Phi\colon \PSO_A(\R)\to\R$ be fixed smooth,
compactly supported functions. Let $X,Y$ be real
numbers with $0 < Y\ll X^{1/2}$. Recall that we restrict $\psi_X[Y]$ to
$V_1(\R)$, obtaining a smooth, compactly supported function $\psi_{X,Y}\colon V_1(\R)\to\R$. We then wrote $\cS(\psi_{X,Y},\Phi)$ for
$\Phi_{X,Y}$. Finally, we wrote $\psi_{X,Y}[Y/X]$ and
$\Phi_{X,Y}[Y/X]$, respectively, for $\psi^{(X,Y)}$ and
$\Phi^{(X,Y)}$. Next, we estimate the Fourier transform of
$\Phi^{(X,Y)}$:

\begin{lemma}\label{lempdbound}
The $m$th partial derivatives of $\Phi^{(X,Y)}$ are bounded by
$O(X^{3m}/Y^{6m})$.
\end{lemma}
\begin{proof}
It is immediately seen that
$\Phi^{(X,Y)}=\cS(\psi_{X,Y}\{Y/X\},\Phi)$.
Therefore, the lemma follows
from bounds on the derivatives of $\psi_{X,Y}\{Y/X\}$, which are readily computed using the chain rule.
\end{proof}

As a standard consequence of the above lemma, we obtain the
following result.

\begin{proposition} \label{prop-fourierpartialbound}
Let $Z$ denote $X^3/Y^6$. The Fourier transform of $\Phi^{(X,Y)}$
satisfies
\begin{equation*}
|\widehat{\Phi^{(X,Y)}}(\check{A})|\ll_M \frac{Z^{M}}{|\check{A}|^{M}},
\end{equation*}
where $|\check{A}|$ denotes the $L^1$-norm of $A$.
\end{proposition}

\subsection{Counting special $\PSO_\cA(\Z)$-orbits on $W_\cA(\Z)$ having bounded height}

Let $\Phi\colon W_{\DA}(\R)\to\R$ be a smooth function. Given a
$\PSO_{\DA}(\Z)$-invariant function $T\colon  W_\DA(\Z)\to\R$, we define
\begin{equation}\label{eqavgquat1}
  N_\Phi(T;X)\defeq
  \int_{\gamma\in\FF_A}\Bigl(\sum_{B\in W_\DA(\Z)^\gen}T(B)\cdot
  (\gamma\cdot\Phi)(B/X)\Bigr)
  d\gamma.
\end{equation}
Above, we take $d\gamma=s_1^{-2}s_2^{-2}dnd^\times s_1d^\times s_2 dk$
for $\gamma=n(s_1,s_2)k$ in Iwasawa coordinates.  Write $n
(s_1,s_2) k=(s_1,s_2) n(s) k$ for $k\in K_{\DA}$, $n\in N'$
and $(s_1,s_2)\in T'$. It is easily checked that $n(s)$ has absolutely
bounded coefficients. Then for $\gamma=n (s_1,s_2) k$, we
have $\gamma\cdot\Phi=(s_1,s_2)\cdot\Phi_{n(s)k}$, where
$\Phi_{n(s)k}$ is smooth and has compact support independent of
$(n,s)\in\FF_{\DA}$.

We assume that $T$ is bounded and $(X^\theta,X)$-acceptable, i.e., $T$
is defined via congruence conditions modulo two relatively prime
positive integers $b\ll X^\theta$ and $q\ll X$, where $q$ is
squarefree and the congruence condition modulo $q$ is simply the one
which cuts out the elements having rank $\leq 1$ modulo $q$. For such
functions $T$, we estimate $N_\Phi(T;X)$ in the following three steps:
first, we bound the number of non-distinguished points in the cuspidal
region of the integral. Second, we bound the number of non generic
elements in the main ball of the integral. Third, we estimate the
number of integer points in the main ball of the integral.

\subsubsection{Bounding the number of integer points in the cusp}

Let $0 < q\ll X$ be a squarefree integer, and let $\L_q \defeq \{B\in W_\DA(\Z) : \text{$B$ is of rank $\leq 1$ modulo $q$}\}$, and let $\L_q^{\rm nd}$ be the set of non-distinguished elements in $\L_q$. We prove the following result:
\begin{proposition}\label{propcuspirred}
Fix $0<\delta_1<1/4$. Then we have
\begin{equation*}
\int_{\substack{nsk\in\gamma\in\FF_A\\\max(s_1,s_2)\gg X^{\delta_1}}}
\Bigl(\sum_{B\in \L_q^{\rm nd}}(\gamma\cdot\Phi)(B/X)\Bigr)d\gamma
\ll_\epsilon \frac{X^{10-2{\delta_1}+\epsilon}}{q^6}.
\end{equation*}
\end{proposition}

We start by proving the following lemmas:

\begin{lemma}\label{lemsmalldist}
Every $B\in\L_q$ with $|b_{12}^2-b_{11}b_{22}|<q$ is distinguished.
\end{lemma}
\begin{proof}
The binary quadratic form $b_{11}x^2+b_{12}xy+b_{22}y^2$ is a square
modulo $q$ since $B$ belongs to $\L_q$.  Hence, the discriminant
$\Delta = b_{12}^2-b_{11}b_{22}$ is divisible by $q$. The condition of the
lemma now implies that $\Delta = 0$, and hence that $b_{11}x^2+b_{12}xy+b_{22}y^2$ is a multiple of the square of a linear form. Thus, $B$ is
distinguished.
\end{proof}

\begin{lemma}\label{lemquad}
Let $1\ll Y_1\ll Y_2\ll Y_3$ be positive real numbers such that
$Y_2^2\asymp Y_1Y_3$. Then the number of integer triples
$(a_1,a_2,a_3)\in\Z^3$ such that $q\mid a_2^2-a_1a_3$ and $|a_i|\ll
Y_i$ is bounded by $O_\epsilon(Y_2(Y_1Y_3)^\epsilon (1+Y_1Y_3/q))$.
\end{lemma}
\begin{proof}
There are $O(Y_2)$ choices for $a_2$. Once we fix $a_2$, we have $a_1a_3=a_2^2+kq$ for some integer $k$. There are
$O(1+Y_1Y_3/q)$ choices for $k$. For each choice of $k$, the pair
$(a_1,a_3)$ is fixed up to $(Y_1Y_3)^{\epsilon}$ choices.
\end{proof}

Let $\L^{\rm nd}_{q,b_{11}=0}$ and $\L^{\rm nd}_{q,b_{11}\neq 0}$
denote the subsets of $\L_q^{\rm nd}$ consisting of elements $(\DA,B)$
with $b_{11}=0$ and $b_{11}\neq 0$, respectively. For each divisor $q' \mid q$, let $\L^{\rm nd}_{q,b_{11}\neq0}(q') = \{B \in  \L^{\rm nd}_{q,b_{11}\neq0} : q' \mid b_{11}(B)\}$. We have the following consequence of Lemma~\ref{lemquad}:
\begin{corollary}\label{corcount}
Write $\phi=\Phi_{n(s)k}$ for $n\in N',\,s=(s_1,s_2)\in T'$. \mbox{Then for
$s_1>s_2$ with $w(b_{11})X\gg q'$, we have}
\begin{equation*}
\sum_{B\in \L_{q,b_{11}\neq0}^{\rm nd}(q')}((s_1,s_2)\cdot\phi_X)(B)\ll_\epsilon
\frac{X^{10+\epsilon}}{q^6}+\frac{X^{8+\epsilon}s_1^4}{q^5}.
\end{equation*}
\end{corollary}
\begin{proof}
We fiber over $b_{11}$, $b_{12}$, $b_{22}$, $b_{13}$, and $b_{14}$. Since the binary quadratic form $b_{11}x^2+b_{12}xy+b_{22}y^2$ is
a square modulo $q$, Lemma~\ref{lemquad} implies that the number of options for the triple $(b_{11},b_{12},b_{22})$
is bounded by $O_\epsilon(X^{1+\epsilon}(1/(q's_1^2)+X^{2}/(q'qs_1^6)))$. The number of options for the pair $(b_{13},b_{14})$ is bounded by
$O(X^2/({q'}^2s_2^2))$.

Once these coefficients have been fixed, Corollary~\ref{corb1234} implies that $B$ is determined modulo $q$ up to $O_\epsilon({q'}^3q^\epsilon)$
choices. Since $Xw(b_{ij})\gg X\gg q$ for the remaining coefficients, it is
clear that we have
\begin{equation*}
  \sum_{B\in \L_{q,b_{11}\neq0}^{\on{nd}}(q')}((s_1,s_2)\cdot\phi_X)(B)\ll_\epsilon
  X^{1+\epsilon}\Bigl(\frac{1}{q's_1^2}+\frac{X^2}{q'qs_1^6}\Bigr)\cdot\frac{X^2}{{q'}^2s_2^2} \cdot
\frac{X^5{q'}^3q^{\epsilon}s_1^6s_2^2}{q^5},
\end{equation*}
as necessary.
\end{proof}

We are now in position to prove Proposition \ref{propcuspirred}.
\medskip

\noindent {\bf Proof of Proposition \ref{propcuspirred}:}
  By symmetry, we may
assume that $s_1>s_2$ for $\gamma=n(s_1,s_2)k\in\FF_\DA$. Moreover, by Lemma~\ref{lemsmalldist}, we may assume that $s_1\ll
X^{1/2}/q^{1/4}$. Thus, stratifying the set $\L_{q,b_{11}\neq 0}^{\rm nd}$ with the subsets $\L_{q,b_{11}\neq 0}^{\rm nd}(q')$ for $q' \mid q$ and applying Corollary~\ref{corcount} to each stratum, we obtain
\begin{equation*}
\int_{\substack{nsk\in\gamma\in\FF_A\\\max(s_1,s_2)\gg X^{\delta_1}}}
\Bigl(\sum_{B\in \L_{q,b_{11}\neq 0}^{\rm nd}}(\gamma\cdot\Phi)(B/X)\Bigr)d\gamma
\ll_\epsilon q^\epsilon
\int_{s_1\gg X^{\delta_1}}^{X^{1/2}/q^{1/4}}\!\!\!\int_{s_2=1}^{s_1}
  \Bigl(\frac{X^{10+\epsilon}}{q^6}+\frac{X^{8+\epsilon}s_1^4}{q^5}
  \Bigr)
  s_1^{-2}s_2^{-2}d^\times s_1d^\times s_2,
\end{equation*}
which is sufficient.

Next suppose that $B\in \L_q^{\rm nd}$ with $b_{11} = 0$. Then $b_{1j}\neq 0$ for $j \in \{2,3\}$ and $q$ divides $b_{1j}$ for $j\in\{2,3,4\}$. It follows that $Xw(b_{1j})\gg q$ for $j\in\{2,3,4\}$, which in turn implies that $Xw(b_{ij})\gg q$ for all $(i,j)\neq (1,1)$. Under these conditions, there are $\ll X^3w(b_{12})w(b_{13})w(b_{14})/q^3$ options for the
tuple $(0=b_{11},b_{12},b_{13},b_{14})$. Once such a tuple has been fixed, it determines $B$ modulo $q$ up to $O_\epsilon(q^{3+\epsilon})$ choices by Corollary~\ref{corb1234}. Thus, we have that
\begin{equation*}
\begin{array}{rcl}
\displaystyle\int_{\substack{nsk\in\gamma\in\FF_A\\\max(s_1,s_2)\gg X^{\delta_1}}}
\Bigl(\sum_{B\in \L_{q,b_{11}= 0}^{\rm nd}}(\gamma\cdot\Phi)(B/X)\Bigr)d\gamma
&\ll_\epsilon&\displaystyle
\int_{s_1\gg X^{\delta_1}}^{X^{1/2}/q^{1/2}}\!\!\!\int_{s_2=1}^{s_1}
\frac{X^{9+\epsilon}}{q^6}w(b_{11})^{-1}
s_1^{-2}s_2^{-2}d^\times s_1d^\times s_2
\\[.2in]&\ll_\epsilon& \displaystyle
\frac{X^{9+\epsilon}}{q^6},
\end{array}
\end{equation*}
which is also sufficient. This concludes the proof of the proposition. $\Box$

\subsubsection{Bounding the number of non-distinguished elements in the main ball}
Let $0 < q \ll X$ be an integer (not necessarily squarefree), let $\mathcal{S}_q \defeq \{B\in W_\DA(\Z) : \text{$B$ is special at $q$}\}$, and for a set $S\subset W_\DA(\Z)$, let $S^\ngen$ denote the set of {\it non-generic} elements in $S$ (i.e., the set of elements in $S$ that are distinguished or have a non-trivial stabilizer in
$\PSO_\DA(\Q)$). The following result bounds the number of non-generic elements in the main ball.
\begin{proposition}\label{propnon-generic}
Let $\delta_1>0$ be sufficiently small. Then, for a constant $\kappa > 0$ independent of $\delta_1$, we have that
\begin{equation*}
\int_{\substack{nsk\in\gamma\in\FF_A\\\max(s_1,s_2) \ll X^{\delta_1}}}
\Bigl(\sum_{B\in \mathcal{S}_q^\ngen}(\gamma\cdot\Phi_X)(B)\Bigr)d\gamma
=O\Bigl(\frac{X^{10-\kappa}}{q^6}\Bigr).
\end{equation*}
\end{proposition}
\begin{proof}
It is a consequence of~\cite[Corollary~1.2]{density} that the density of $\mathcal{S}_q$ in $W_{\DA}(\Z)$ is $O(q^{-6})$. The result then follows from using the Selberg sieve applied as in \cite{MR3264252}. The only additional ingredient needed to apply the sieve is to show that for every prime $p > 2$, there exists a set $T_p$ (resp., $T_p'$) contained in $W_\DA(\F_p)$ satisfying the following two properties: (1) every $B\in W_{\DA}(\Z)$ whose reduction modulo $p$ lies in $T_p$ (resp., $T_p'$) is non-distinguished (resp., has trivial stabilizer in $\PSO_\DA(\Q)$); and (2) the size $\#T_p$ (resp., $\#T_p'$) is $\gg p^{10}$.

We now construct the sets $T_p$ and $T_p'$. For any separable monic binary quartic form $f \in V_1(\F_p)$, it follows from~\cite[Proposition~14]{MR3719247} that the set of $\PSO_\DA(\F_p)$-orbits on $W_\DA(\F_p)$ with resolvent $f$ is in bijection with $J_f(\F_p)/2J_f(\F_p)$. It further follows from~\cite[Equation (8)]{MR3719247} that the stabilizer of any element of such an orbit is isomorphic to $J_f(\F_p)[2]$.
Thus, we can take $T_p$ to be the set of all quaternary quadratic forms $B\in W_\DA(\F_p)$ with separable resolvent $f$ corresponding to a non-trivial element of $J_f(\F_p)/2J_f(\F_p)$. Similarly, we can take $T_p'$ to be the set of all forms $B\in W_\DA(\F_p)$ with separable resolvent $f$ such that $J_f(\F_p)[2]$ is trivial.

It remains to show that the sizes of $T_p$ and $T_p'$ satisfy the required lower bounds. Recall that for any separable form $f \in V_1(\F_p)$, the $2$-torsion subgroup $J_f[2]$ depends only on the splitting type of $f$ over $\F_p$. Since each unramified splitting type occurs with positive proportion (where by ``positive,'' we mean ``positive in the large-$p$ limit''), it follows from the orbit-stabilizer theorem that the set $T_p$ (resp., $T_p'$) comprises a positive proportion of $W_{\DA}(\F_p)$ --- i.e., we have $\#T_p\gg p^{10}$ (resp., $\#T_p'\gg p^{10}$), as necessary.
\end{proof}

\subsubsection{Estimating the number of integer points in the main ball}

Assume that the Fourier transform of a smooth, compactly supported function
$\Phi\colon W_\DA(\R)^{(i,j)}\to\R$ satisfies the bound
$\widehat{\Phi}(\check{A})\ll_M Z^M/|\check{A}|^M$, for some real number $Z > 0$. Then, for $\gamma=n (s_1,s_2)k\in\FF_\DA$,
Proposition~\ref{prop-twistedpoisson} gives
\begin{equation}\label{eqmt1}
\sum_{B\in W_{\DA}(\Z)^{(i,j)}}T(B)\cdot ((s_1,s_2)\Phi_{n(s)k})(B/X)
=X^{10}\widehat{\Phi}(0)\nu(T)+E.
\end{equation}
The error term $E$ above is bounded by
\begin{equation}\label{eqemt2}
\begin{array}{rcl}
  \displaystyle  E &\ll_\epsilon &
  \displaystyle\frac{X^{10}}{(bq)^{10}}
  \sum_{\substack{0\neq\check{A}\in\widehat{W(\Z)}
      \\|\check{a}_{ij}|\ll (Zbq)^{1+\epsilon}/(Xw(b_{ij}))}}
  |\widehat{T}(\check{A})|
\\[.25in]
&\ll_\epsilon&\displaystyle
\frac{X^{10}}{(bq)^{10}}\cdot\frac{(Zbq)^{10+\epsilon}}{X^{10}}
\max(s_1,s_2)^6\min(s_1,s_2)^2\cdot q^{7/2+\epsilon}b^{10}
\\[.2in]&\ll_\epsilon&\displaystyle
\max(s_1,s_2)^6\min(s_1,s_2)^2\cdot (Zb)^{10+\epsilon}q^{7/2+\epsilon},
\end{array}
\end{equation}
where we use Proposition~\ref{prop-fourierrank1bound} to bound the value of $|\widehat{T}(\check{A})|$ by $q^{7/2+\epsilon}$ when $\check{A}$ is not $0$ modulo $p$ for any $p$ dividing $q$. If the gcd of $q$ and all of the coefficients of $\check{A}$ is $q_1$, then we have a bound
of $q^{7/2+\epsilon}q_1^{1/2}$ on $|\widehat{T}(\check{A})|$. However, we recover the required bound from the fact that every coefficient of such $\check{A}$ is a multiple of $q_1$, thereby shortening the sum over $\check{A}$.

We are ready to estimate $N_\Phi(T;X)$:
\begin{proposition}\label{propWAestimate}
Let notation be as above. We have for some $\kappa > 0$ and any $0<\delta_1<1/4$ that
\begin{equation*}
  N_\Phi(T;X)=\Vol(\FF_\DA)X^{10}\widehat{\Phi}(0)\nu(T)
  +O\Bigl(\frac{X^{10-\kappa}}{q^6}\Bigr)+O_\epsilon\Bigl(\frac{X^{10-2\delta_1+\epsilon}}{q^6}\Bigr)+
O_\epsilon\Bigl(
(Zb)^{10+\epsilon}q^{7/2+\epsilon}X^{4{\delta_1}}\Bigr).
\end{equation*}
\end{proposition}
\begin{proof}
From Propositions \ref{propcuspirred} and \ref{propnon-generic}, we have
\begin{equation*}
\begin{array}{rcl}
\displaystyle N_\Phi(T;X)&=&\displaystyle
\int_{\gamma\in\FF_\DA}\Bigl(\sum_{B\in W_\DA(\Z)^{(i,j),\irr}}(\gamma\cdot\Phi_X)(B)T(B)
\Bigr)d\gamma
\\[.2in]&=&\displaystyle
\int_{\substack{nsk=\gamma\in\FF_\DA\\\max(s_1,s_2) \ll X^{\delta_1}}}
\Bigl(\sum_{B\in W_\DA(\Z)^{(i,j)}}(\gamma\cdot\Phi_X)(B)T(B)\Bigr)d\gamma
+O\Bigl(\frac{X^{10-\kappa}}{q^6}\Bigr)+O_\epsilon\Bigl(\frac{X^{10-2\delta_1+\epsilon}}{q^6}
\Bigr).
\end{array}
\end{equation*}
The main term in the second line of the above equation can be
estimated using~\eqref{eqmt1} and~\eqref{eqemt2}:
\begin{equation*}
\int_{\substack{nsk\in\gamma\in\FF_\DA\\\max(s_1,s_2) \ll X^{\delta_1}}}
\Bigl(\sum_{B\in W_\DA(\Z)^{(i,j)}}(\gamma\cdot\Phi_X)(B)T(B)\Bigr)d\gamma
=\Vol(\FF_\DA)X^{10}\widehat{\Phi}(0)\nu(T)+O_\epsilon\Bigl(
(Zb)^{10+\epsilon}q^{7/2+\epsilon}X^{4{\delta_1+\epsilon}}
\Bigr),
\end{equation*}
as necessary.
\end{proof}

We are now ready to prove Theorem \ref{thmainWA}.

\medskip

\noindent {\bf Proof of Theorem~\ref{thmainWA}:} Applying Corollary \ref{cor-needproof53}
yields
\begin{equation*}
\begin{array}{rcl}
N_\psi^\gen(W_\DA^{(i,j)},T;X,Y)&=&\displaystyle
\frac{1}{\sigma_\DA(i)\Vol(\Theta)}\int_{\gamma\in\FF_\DA}\Bigl(
\sum_{B\in W_\DA(\Z)^{(i,j)}}
T(B)(\gamma\cdot\Phi^{(X,Y)}[X/Y])(B)
\Bigr)d\gamma
\\[.2in]&=&\displaystyle
\frac{1}{\sigma_\DA(i)\Vol(\Theta)} N_{\Phi^{(X,Y)}}(T;X/Y).
\end{array}
\end{equation*}

We apply Proposition \ref{propWAestimate} to estimate
$N_{\Phi^{(X,Y)}}(T;X/Y)$. First note that since $T$ is
$(X^\theta,X^{1/2})$-acceptable, we have $b\ll
X^\theta$ and $q\ll X^{1/2}$. Thus, the error term from the proposition is $\ll_\epsilon$
\begin{equation*}
\frac{(X/Y)^{10-\kappa}}{q^6}+\frac{(X/Y)^{10-2\delta_1+\epsilon}}{q^6}+
(Z)^{10+\epsilon}(X/Y)^{7/2+10\theta+4\delta_1+\epsilon},
\end{equation*}
for some $0<\kappa$ and any $0<\delta_1<1/4$. By assumption, we have
$Y\gg X^{1/2-\delta}$, so by Proposition~\ref{prop-fourierpartialbound}, we
may take $Z=X^{6\delta}$. Therefore, the error term arising from
Proposition \ref{propWAestimate} is $\ll_\epsilon$
\begin{equation*}
\frac{(X^{1/2+\delta})^{10-\kappa}}{q^6}+\frac{(X^{1/2+\delta})^{10-2\delta_1+\epsilon}}{q^6}+
(X^{1/2+\delta})^{7/2+10\theta+60\delta+4\delta_1+\epsilon}\ll_\epsilon
\frac{X^{5-\lambda+10\theta+60\delta+\epsilon}}{q^6}
\end{equation*}
for some sufficiently small $\lambda>0$. We now compute the main term
to be
\begin{equation*}
\begin{array}{rcl}
\displaystyle\frac{1}{\sigma_\DA(i)}
\Vol(\FF_\DA)\frac{X^{10}}{Y^{10}}\frac{\widehat{\Phi^{(X,Y)}}(0)}{\Vol(\Theta)}\nu(T)
&=&\displaystyle
\frac{|\J_\DA|}{\sigma_\DA(i)}
\Vol(\FF_\DA)\frac{X^{10}}{Y^{10}}\Vol(\psi^{(X,Y)})\nu(T)
\\[.2in]&=&\displaystyle
\frac{|\J_\DA|}{\sigma_\DA(i)}\Vol(\FF_\DA)\Vol(\psi_{X,Y})\nu(T)
\\[.2in]&=&\displaystyle
\frac{|\J_\DA|}{\sigma_\DA(i)}\Vol(\FF_\DA)\Vol\bigl(\psi|_{\frac{Y^2}{X}}\bigr)
X^4Y^2\nu(T),
\end{array}
\end{equation*}
where the first equality follows from Corollary \ref{corjacinfty}, and
the rest follow directly from the definitions of the various smooth
functions involved. This concludes the proof of Theorem \ref{thmainWA}. $\Box$

\section{The second moment of $|\Sel_2(E)|$} \label{sec-thesecondmoment}

For $(I,J)\in\Z^2$ with $\Delta(I,J)\neq 0$, let ${E^{I,J}}$ denote the
elliptic curve with invariants $I$ and $J$. We know from~\cite[Theorem~3.5]{MR3272925} that $\Sel_2({E^{I,J}})$ is in bijection with the set of $\PGL_2(\Q)$-equivalence classes of locally soluble integral binary quartic forms with invariants $2^4I$ and $2^6J$. Let $\FF$ be an acceptable collection of elliptic curves over $\Q$, and let $\FF_\inv$ be the set
\begin{equation*}
\FF_\inv\defeq \{(2^4I,2^6J)\in\Inv(\Z):{E^{I,J}}\in\FF\}.
\end{equation*}
We denote the natural bijection $\FF\to\FF_\inv$ by $\iota$, and we write $\FF^{\pm} \defeq \iota^{-1}(\on{Inv}(\R)^{\pm})$. Let
$\H\colon \Inv(\R)^\pm\to\R$ be a smooth bounded function. Then, as in
\cite[\S3.2]{MR3272925}, we have the equality
\begin{equation*}
\begin{array}{rcl}
\displaystyle \sum_{E\in \FF^\pm}(|\Sel_2(E)|-1)\H_X(\iota(E))&=&
\displaystyle \sum_{f\in \frac{V(\Z)^\pm}{\PGL_2(\Z)}}\cN(\FF,f)\H_X(\iota(f)),
\end{array}
\end{equation*}
where $\cN(\FF,f)$ is defined below:
\begin{equation*}
\cN(\FF,f)\defeq \left\{
\begin{array}{rl}
1/n(f) &\mbox{ if $f$ is locally soluble, has no rational linear
  factor, and }\iota(f)\in\FF_\inv;\\[.1in] 0&\mbox{ otherwise.}
\end{array}
\right.
\end{equation*}
Above, $n(f)$ denotes the number of
$\PGL_2(\Z)$-orbits in the $\PGL_2(\Q)$-equivalence class of $f$ in
$V(\Z)$.

For an integral binary quartic form $f$, write $S(f)\defeq \Sel_2(\Jac(C_f))$ and $|S|(f) \defeq |\Sel_2(\Jac(C_f))|$. Then we have the following tautological identity:
\begin{equation}\label{eqsecmom1}
\begin{array}{rcl}
\displaystyle \sum_{E\in \FF^\pm}|\Sel_2(E)|(|\Sel_2(E)|-1)\H_X(\iota(E))&=&
\displaystyle \sum_{f\in \frac{V(\Z)^\pm}{\PGL_2(\Z)}}\cN(\FF,f)|S|(f)\H_X(\iota(f)).
\end{array}
\end{equation}
We will replace $\cN(\FF,f)$ with $\cM(\FF,f)$, where $\cM(\FF,f)$
defined below. The function $\cN(\FF,f)$ vanishes on $V(\Z)^{(2-)}$,
and we will define $\cM(\FF,f)$ to be $0$ there too. For
$i\in\{0,1,2+\}$, and $f\in V(\Z)^{(i)}$ with nonzero discriminant, we
define $\cM(\FF,f)=\prod_p\cM_p(\FF,f)$, where
$\cM_p(\FF,f)\colon V(\Z_p)\to\R$ is as follows: letting $\mc{F}_{\on{inv},p}$ denote the $p$-adic closure of $\mc{F}_{\on{inv}}$ in $\on{Inv}(\Z_p)$, we put
\begin{equation*}
\cM_p(\FF,f)\defeq \left\{
\begin{array}{rl}
1/m_p(f) &\mbox{ if $f$ is soluble over $\Q_p$ and }\iota(f)\in\FF_{\inv,p};\\[.1in]
0&\mbox{ otherwise.}
\end{array}
\right.
\end{equation*}
Again, as in \cite[\S3.2]{MR3272925}, $m_p(f)$ denotes the number of
$\PGL_2(\Z_p)$-orbits in the $\PGL_2(\Q_p)$-equivalence class of $f$
in $V(\Z_p)$. Over the next few subsections, we estimate the
term
\begin{equation}\label{eq2sel2ndstage}
N^\gen_\H(V^{(i)},\mc{M}(\FF,\cdot)|S|;X)=
\sum_{f\in \frac{V(\Z)^{(i)}}{\PGL_2(\Z)}}\cM(\FF,f)|S|(f)\H_X(\iota(f)),
\end{equation}
for smooth, compactly supported functions $\H\colon \Inv(\R)^\pm\to\R$.

\subsection{Slicing over the leading coefficient}

Let $T\colon V(\Z)\to\R$ be a $\PGL_2(\Z)$-invariant function, and fix
$i\in\{0,1,2+\}$. By Proposition \ref{propbqfavg}, we have that
\begin{equation*}
N_\H^\gen(V^{(i)},T;X)=\frac{1}{\sigma_i\Vol(\Theta)}\O(T,\psi;X^{1/6}),
\end{equation*}
where $\psi=\psi^{(i)}=\cS_V(\H,\Theta)$.  We may assume without loss
of generality that $\psi$ is $\SO_2(\R)$-invariant (this is guaranteed
as long as $\Theta$ is $\SO_2(\R)$-invariant). To evaluate the
integral $\cO(T,\psi;X)$, we fiber the sum over $f\in V(\Z)^\gen$ by
the value of $a(f)$, obtaining
\begin{equation*}
\begin{array}{rcl}
\displaystyle\cO(T,\psi;X)&=&
\displaystyle
\int_{(t,n)\in\FF_2}\Bigl(\sum_{|a|\neq 0}
\sum_{f\in V_a(\Z)^\gen}
T(f)\cdot((t,t^{-2}n)\cdot\psi_X)(f)
\Bigr)dn\frac{d^\times t}{t^2},
\end{array}
\end{equation*}
where $(t,t^{-2}n)$ denotes the matrix $\Bigl(\begin{array}{cc}t^{-1}&\\&t
\end{array}
\Bigr)\cdot\Bigl(\begin{array}{cc}1&\\t^{-2}n&1
\end{array}
\Bigr) $. We will transform the innermost sum in the above equation
into one over monic binary quartic forms as follows.

For $f \in V_a(\Z)$, let
$f^{\mon*} \defeq (\upkappa f)^{\on{mon}}$, and for $g \in \on{mon}(\upkappa V_a(\Z))$, let $g^{\on{dem}*}_a \defeq \upkappa^{-1}g^{\on{dem}}_a$. For a function $\psi\colon V(\R)\to\R$ and a real number $n$, let
$\psi^{(n)}$ denote $\Bigl(\begin{array}{cc}1&\\n&1
\end{array}
\Bigr)\cdot \psi $. Then we have that
\begin{equation*}
\sum_{f\in V_a(\Z)^\gen}
T(f)\cdot((t,t^{-2}n)\cdot\psi_X)(f)=
\sum_{f\in V_a(\Z)^\gen}T_a(f^{\on{mon}*})\cdot\bigl(\psi_{\upkappa^2 aX}^{(t^{-2}n)}[\upkappa at^2]\bigr)(f^{\on{mon}*}),
\end{equation*}
where for $g \in \on{mon}(\upkappa V_a(\Z))$, we write $T_a(g) \defeq T(g^{\on{dem}*}_a)$. Then we immediately obtain the following result:
\begin{theorem}\label{thotxeval}
We have that
\begin{equation*}
\displaystyle N_\H^\gen(V^{(i)},T;X)=
\frac{1}{\sigma_i\Vol(\Theta)}\int_{(t,n)\in\FF_2}\Bigl(\sum_{|a|\neq 0}
\sum_{f\in V_a(\Z)^\gen}T_a(f^{\on{mon}*})\cdot\bigl(\psi_{\upkappa^2 aX^{1/6}}^{(t^{-2}n)}[\upkappa at^2]\bigr)(f^{\on{mon}*})
\Bigr)dn\frac{d^\times t}{t^2}.
\end{equation*}
\end{theorem}

\subsection{Introducing local masses on $W_\DA(\Z)$}

Let $f \in V(\Z)$ with $a(f) \neq 0$. From \S\ref{sec-2selhypjac}, we know that $S(f)$ is in bijection with the set of $\PSO_\DA(\Q)$-equivalence classes on certain integral quaternary quadratic forms $B$ with resolvent $f^{\mon*}$ defined by local conditions. Moreover, every such $\PSO_\DA(\Q)$-equivalence class has an element satisfying the following additional congruence condition modulo $a(f)$: if $p$ is a prime such that $p\parallel a(f)$, then $B$ is of rank $\leq 1$ modulo $p$.

This motivates the definition of the following mass functions on
$W_\DA(\Z)$. For each prime $p$, let $\cS_p\subset W_\DA(\Z_p)$ be a
$\PSO_\DA(\Z_p)$-invariant set. Denote the collection $(\cS_p)_p$ by
$\cS$. Let $W_\DA(\cS)$ be the set
$W_\DA(\Z)\cap\big(\cap_p\cS_p\big)$, which is immediately seen to
be $\PSO_\DA(\Z)$-invariant. For an element $B\in\cS_p$ (resp.\ $B\in
W_\DA(\cS)$), let $\cB_p(B)$ (resp.,~$\cB(B)$) be a set of
representatives for the action of $\PSO_\DA(\Z_p)$
(resp.,~$\PSO_\DA(\Z)$) on the $\PSO_\DA(\Q_p)$-equivalence class
(resp.\ $\PSO_\DA(\Q)$-equivalence class) of $B$ in $\cS_p$
(resp.\ $W_\DA(\cS)$). Then, for $B\in W_\DA(\cS)$, we define:
\begin{equation*}
n(\cS,B)\defeq \sum_{B'\in \cB(B)}1;\quad\quad\quad
m(\cS,B)\defeq \sum_{B'\in \cB(B)}\frac{|\Stab_{\PSO_\DA(\Q)}(B)|}{|\Stab_{\PSO_\DA(\Z)}(B')|}.
\end{equation*}
Analogously, for $B\in\cS_p$, we define
\begin{equation*}
  m_p(\cS_p,B)\defeq
  \sum_{B'\in \cB_p(B)}\frac{|\Stab_{\PSO_\DA(\Q_p)}(B)|}{|\Stab_{\PSO_\DA(\Z_p)}(B')|}.
\end{equation*}
Then we have the following result whose proof is identical to that of
\cite[Proposition 3.6]{MR3272925}.

\begin{proposition}
Let notation be as above. Then for $B\in W_\DA(\cS)$, we have
$m(\cS,B)=\prod_pm_p(\cS_p,B)$.
\end{proposition}

Now let $f$ be a binary quartic form and denote $a(f)$ by $a$. Let
$\cS(a)=(\cS(a)_p)_p$ be defined as follows: for $p\parallel a$, we
set $\cS(a)_p$ to be the set of all elements in $W_\DA(\Z_p)$ with rank $\leq 1$ modulo $p$; for all other primes $p$ we simply set
$\cS(a)_p=W_\DA(\Z_p)$. We know that we have
\begin{equation}\label{eqSfglobal}
|S|(f)=\sum_{\substack{B\in \frac{W_\DA(\cS(a))\cap\iota^{-1}(f^{\mon*})}{\PSO_\DA(\Z)} \\ \text{$B$ loc.~sol.}}}
\frac{1}{n(\cS(a),B)}.
\end{equation}
Once again, since $m(\cS(a),\cdot)$ has an Euler product while
$n(\cS(a),\cdot)$ does not, we replace $n(\cS(a),\cdot)$ with
$m(\cS(a),\cdot)$ and define $\cM_p(\FF;a)\colon W_\DA(\cS(a))\to\R$ as follows:
\begin{equation*}
\cM_p(\FF;a;B)\defeq \left\{
\begin{array}{rl}
\displaystyle\frac{\cM_p(\FF, \iota(B)^{\on{dem}*}_a)}{m_p(\cS(a)_p,B)} &
\mbox{ if $B$ is soluble at $p$};
\\[.15in]
0&\mbox{ otherwise,}
\end{array}
\right.
\end{equation*}
We define the global mass $\cM(\FF;a;\cdot)$ as follows:
\begin{equation*}
\begin{array}{rcl}
\cM(\FF;a;\cdot)\colon W_\DA(\Z)&\to&\R\\[.15in]
B&\mapsto&\prod_p\cM_p(\FF;a;B)=
\displaystyle\frac{\cM(\FF, \iota(B)^{\on{dem}*}_a)}{m(\cS(a),B)}.
\end{array}
\end{equation*}
The masses $\cN(\FF;a;B)$ and $\cM(\FF;a;B)$ only differ on non-generic
elements $B$. This motivates the following bound on the number of
non-generic and non-distinguished orbits $B$:
\begin{lemma}\label{lemnongendist}
Let $S$ denote the characteristic function of the set of elements in $W_\DA(\Z)$ that are non-generic and non-distinguished. Then we have for some $\lambda > 0$ that
\begin{equation*}
\int_{t\gg 1}^{\ll X^{1/24}}\Bigl(\sum_{0\neq |a|\ll X^{1/6}/t^4}
N_{\psi}(W_\DA^{(i)},S;aX^{1/6},at^2)\Bigr)\frac{d^\times t}{t^2}
\ll X^{5/6-\lambda}.
\end{equation*}
\end{lemma}
\begin{proof}
We split the sum over $a$ into the big range with $|a|>
(X^{1/6}/t^4)^{1-\delta}$ and the small range with $|a|\leq
(X^{1/6}/t^4)^{1-\delta}$, for some $\delta>0$ to be chosen later. Using Proposition \ref{propnon-generic} for $a$ in the big range, we obtain
\begin{equation*}
N_{\psi}(W_\DA^{(i)},S;aX^{1/6},at^2)\ll \frac{X^{10/6-\kappa}}{t^{20}a^6}
\ll t^4 X^{4/6+\delta-\kappa},
\end{equation*}
for some $\kappa>0$ independent of $\delta$. For $a$ in the small
range, we use Corollary \ref{coravgSO} to obtain the bound
\begin{equation} \label{eq-nongensbound}
N_{\psi}(W_\DA^{(i)},S;aX^{1/6},at^2)\ll \frac{X^{5/6-\delta/2}}{a}\prod_{\substack{p^k \parallel a \\ k \geq 3}} p^{k/2}.
\end{equation}
Summing the bound~\eqref{eq-nongensbound} over $a$ and integrating over $t$ gives a contribution that is $\ll_\epsilon$
\begin{equation*}
X^{5/6-\delta/2+\epsilon} \sum_{\substack{1 \leq |c| \leq (X^{1/6}/t^4)^{1-\delta} \\ \text{$c$ cubeful}}} \frac{1}{\sqrt{c}} \ll_\epsilon  X^{5/6 - \delta/2 + \epsilon},
\end{equation*}
as desired, where we say that an integer $m$ is cubeful if, for each prime $p \mid m$, we have $p^3 \mid m$.
\end{proof}

Finally, for $i\in\{0,1,2+,2-\}$, we define
\begin{equation*}
\cI^{(i)}=
\frac{1}{\sigma_i\Vol(\Theta)}\int_{(t,n)\in\FF_2}\Bigl(\sum_{|a|\neq 0}
N^\gen_{\psi^{(t^{-2}n)}}\bigl(
W_\DA^{(i)},\cM(\FF;a,\cdot);\six^2 aX^{1/6},\six at^2
\bigr)
\Bigr)dn\frac{d^\times t}{t^2},
 \end{equation*}
and set $\cI^+=\cI^{(0)}+\cI^{(2+)}$ and $\cI^-=\cI^{(1)}$.
We now have the following result.
\begin{theorem}
We have for some $\lambda > 0$ that
\begin{equation} \label{eq-abovedisplayed}
\displaystyle\sum_{E\in \FF^\pm} \sum_{\substack{\sigma\in\Sel_2(E)\\\sigma\neq 1}}\bigl(
|S_2(\sigma)|-2\bigr)
\H_X(\iota(E))
=\cI^\pm+O(X^{5/6-\lambda}).
\end{equation}
\end{theorem}
\begin{proof}
We prove the result for negative-discriminant elliptic curves; the
positive-discriminant case follows identically. As long as $f$ is generic, there are precisely
two distinguished orbits arising from $f$. Thus,~\eqref{eqsecmom1} and Theorem \ref{thotxeval} imply that the left-hand
side of~\eqref{eq-abovedisplayed} is equal to
\begin{equation*}
\frac{1}{\sigma_1\Vol(\Theta)}\int_{(t,n)\in\FF_2}\Bigl(\sum_{|a|\neq 0}
\sum_{f\in V_a(\Z)}\bigl(\psi_{\six^2 aX^{1/6}}^{(t^{-2}n)}[\six at^2]\bigr)(f^{\on{mon}*})\cN(\FF,f)(|S|_a(f^{\on{mon}*})-2)
\Bigr)dn\frac{d^\times t}{t^2},
\end{equation*}
where as before, $|S|(g)$ denotes $|\Sel_2(J_g)|$, and
$|S|_a(f)\defeq |S|(f^{\on{dem}*}_a)$. The value of $|S|_a(f^{\on{mon}*})$ is given in \eqref{eqSfglobal}
in terms of a weighted sum of $\PSO_\DA(\Z)$-orbits on
$W_\DA(\Z)$. We define
$\cN\colon W_\DA(\Z)\to\R$ via
\begin{equation*}
\cN(\FF;a;B)\defeq \left\{
\begin{array}{rl}
\displaystyle\frac{\cN(\FF, \iota(B)^{\on{dem}*}_a)}{n(\cS(a),B)}, &
\mbox{ if $B$ is locally soluble},
\\[.15in]
0, &\mbox{ otherwise,}
\end{array}
\right.
\end{equation*}
Therefore, we have that
\begin{equation*}
\sum_{|a|\neq0}\sum_{f\in
  V_a(\Z)}\bigl(\psi_{\six^2 aX^{1/6}}^{(t^{-2}n)}[\six at^2]\bigr)(f^{\on{mon}*})\cN(\FF,f)(|S|_a(f^{\on{mon}*})-2)
\Bigr)=\sum_{|a|\neq0} N^\gen_{\psi^{(t^{-2}n)}}\bigl(W_\DA^{(1)},
\cN(\FF;a;B);\six^2 aX^{1/6},\six at^2\bigr),
\end{equation*}
up to an error of $O(X^{5/6-\lambda})$ for some $\lambda>0$, where the
error term follows from bounding non-distinguished and non-generic
elements using Lemma \ref{lemnongendist}.

Since $\cN(\FF;a;B)$ differs from $\cM(\FF;a;B)$ only on
elements $B$ that are not generic, it follows from Lemma
\ref{lemnongendist} that we have
\begin{equation*}
\sum_{|a|\neq0} \sum_{f\in V_a(\Z)}\bigl(\psi_{\six^2 aX^{1/6}}^{(t^{-2}n)}[\six at^2]
\bigr)(f^{\on{mon}*})\cN(\FF,f)(|S|_a(f^{\on{mon}*})-2)
\Bigr)=\sum_{|a|\neq0} N_{\psi^{(t^{-2}n)}}^{\gen}\bigl(W_\DA^{(1)},\cM(\FF;a;B);
\six^2 aX^{1/6},\six at^2\bigr),
\end{equation*}
up to an error of $O(X^{5/6-\lambda})$. The result now follows from
integrating over $(t,n)\in \FF_2$.
\end{proof}

\subsection{Truncating the local masses and applying the counting results}

Recall that we choose (unique up to sign) top degree forms $\omega$ on
$\PSO_\DA$ and $\PGL_2$ for the application of Theorem
\ref{thJacgeneral}. For primes $p$, we
define
\begin{equation*}
\tau_\DA(p)\defeq \int_{g\in \PSO_\DA(\Z_p)}\omega(g);\quad\quad
\tau(p)\defeq \int_{g\in \PGL_2(\Z_p)}\omega(g).
\end{equation*}
Since the Tamagawa numbers of $\PSO_\DA(\Q)$ and $\PGL_2(\Q)$ are $4$
and $2$, respectively, we have for $Z > 0$ that
\begin{equation}
\begin{array}{rcccl}
\displaystyle 4\tau^{(Z)}_\DA&\defeq &4\displaystyle
\prod_{p>Z}\tau_\DA(p)^{-1}&=&\displaystyle\Vol(\FF_\DA)
\prod_{p\leq Z}\tau_\DA(p);\\[.2in]
\displaystyle 2\tau^{(Z)}&\defeq &\displaystyle
2\prod_{p>Z}\tau(p)^{-1}&=&\displaystyle\Vol(\FF_2)
\prod_{p\leq Z}\tau(p).
\end{array}
\end{equation}

For an integer $\ell>0$, let
$\cM_{p,\ell}(\FF;a;\cdot)\colon W_\DA(\Z_p)\to\R$ be the best upper approximation of
$\cM_p(\FF;a;\cdot)$ defined modulo $p^\ell$.  For positive real numbers
$r<s$, we define the truncated mass
\begin{equation} \label{eq-truncatedefine}
  \cM^{(r,s)}(\FF;a;B)\defeq \prod_{\substack{p<r, \, p^2 \nmid a\\p^\ell<s\leq p^{\ell+1}}}
  \cM_{p,\ell}(\FF;a;B) \times
  \prod_{\substack{p < r, \, p^2 \mid a \\ p^\ell < s \leq p^{\ell + 1}}} \cM_{p,360\nu_p(a)+\ell}(\FF; a; B) \times \prod_{r\leq p,\, p | a}\cS_p(a;B)
\end{equation}
where, if $p^k\parallel a$, then $\cS_p(a,\cdot)$ is the characteristic function of the set of elements that are special at $p^k$.

Let $\theta>0$ be a real number that will eventually be picked to be
sufficiently small. For a positive real number $X$, we define
$X_\theta$ to be the largest integer satisfying $\prod_{i\leq
  X_\theta}i<X^\theta$. We will consider the truncated function
$\cM^{(\log X_\theta,X_\theta)}(\FF;a,\cdot)$ which satisfies $\cM^{(\log X_\theta,X_\theta)}(\FF;a;B)\geq \cM(\FF;a;B)$ for every $B\in W_\DA(\Z)$. For each place $v$ of $\Q$, define the ratio
\begin{equation*}
\alpha_v\defeq \frac{|E(\Q_v)/2E(\Q_v)|}{|E[2](\Q_v)|},
\end{equation*}
where $E$ is any elliptic curve over $\Q_v$. The ratio is independent
of $E$, and depends only on $v$. The following lemma follows from
\cite[Lemma 3.20]{MR3272925}.
\begin{lemma}\label{lemprodalv}
For large enough primes $p$, we have $\alpha_p=1$. Moreover,
$\prod_v\alpha_v=1$.
\end{lemma}
\begin{proof}
Indeed, we have $\alpha_\infty=\frac12$, $\alpha_2=2$, and
$\alpha_p=1$ for odd primes $p$.
\end{proof}

Let $\cN(M)$ denote the set of natural numbers whose squareful part is
greater than $M$, and let $\kappa\ll 1$ be a sufficiently small positive real number. We say that a nonzero integer $a$ is \emph{good} if $a \not\in \cN(X^\theta)$ and $a \geq (X^{1/6}/t^4)^{1-\kappa}$, and \emph{bad} otherwise.
For ease of notation, define
\begin{equation}\label{eqTa}
T_a\defeq \begin{cases} \cM^{(\log X_\theta,X_\theta)}
(\FF;a,\cdot), & \text{ if $a$ is good,} \\ \tau_{\DA}^{(\log X_\theta)} \cM(\FF;a,\cdot), & \text{ if $a$ is bad.} \end{cases}
\end{equation}
We then have the following result for good integers $a$:
\begin{proposition}\label{propgooda}
Let $a$ be a good integer, and let $\psi\colon V(\R)\to\R$ be a smooth,
compactly supported function. Then for some $\lambda > 0$ independent of the choice of the sufficiently small constants $\theta$ and $\kappa$, we have that
\begin{equation*}
\begin{array}{rcl}
\displaystyle N_{\psi}\bigl(W_\DA^{(i)},T_a;\six^2 aX^{1/6},\six at^2\bigr)
&=&\displaystyle
\alpha_\infty|\J_\DA|\Vol(\FF_\DA)\nu(T_a)
\Vol\bigl(\psi|_{\frac{at^4}{X^{1/6}}}\bigr)\six^{10}a^6X^{4/6}t^4
+O\Bigl(\frac{X^{5/6-\lambda}}{a}
\Bigr).
\end{array}
\end{equation*}
\end{proposition}
\begin{proof}
Recall that $\uptau_\DA^{\on{sol}}(i)$ denotes the number of soluble
$\PSO_\DA(\R)$-orbits on $\iota^{-1}(f)$ for any form $f\in
V(\R)^{(i)}$. We therefore have that
\begin{equation} \label{eq-tauident}
\sum_{j=1}^{\uptau_{\DA}^{\on{sol}}(i)}\frac{1}{\sigma(i)}=\frac{\uptau_{\DA}^{\on{sol}}(i)}{\sigma(i)}=\alpha_\infty.
\end{equation}
We claim that the function $T_a$ is $((aX^{1/6})^{2166\theta},(aX^{1/6})^{1/2})$-acceptable. To see why, observe that the $p$-adic factor on the right-hand side of~\eqref{eq-truncatedefine} is defined by congruence conditions modulo:
\begin{itemize}[itemsep=0pt]
\item $p^{\ell_p}$, where $\ell_p$ is the unique integer satisfying $p^{\ell_p} < X_\theta \leq p^{\ell_p+1}$, if $p < \log X_\theta$ and $p^2 \nmid a$;
\item $p^{360\nu_p(a)+\ell_p}$, if $p < \log X_\theta$ and $p^2 \mid a$; and
\item $p^{\nu_p(a)}$, if $p \geq \log X_\theta$, since $\cS_p(a;B)$ is defined modulo $p^{\nu_p(a)}$ (see~\S\ref{sec-interpret}).
\end{itemize}
Thus, $T_a$ is defined by congruence conditions modulo $m_1m_2$, where
$$m_1 = \prod_{\substack{p < \log X_\theta \\ p^2 \nmid a}} p^{\ell_p} \times \prod_{\substack{p < \log X_\theta \\ p^2 \mid a}} p^{360\nu_p(a)+\ell_p} \times \prod_{\substack{p \geq \log X_\theta \\ p^2 \mid a}} p^{\nu_p(a)} \quad \text{and} \quad m_2 = \prod_{\substack{p \geq \log X_\theta \\ p \parallel a}} p$$
Since $a \not\in \cN(X^\theta)$, we have that $m_1 \leq \prod_{i \leq X_\theta} i \times \prod_{p^2 \mid a} p^{360\nu_p(a)}  < X^{361\theta} \leq (aX^{1/6})^{2166\theta}$. Moreover, $m_2$ is squarefree with $m_2 \leq a \ll (aX^{1/6})^{1/2}$, $\gcd(m_1,m_2) = 1$, and the $p$-adic factor in the definition of $T_a$ at every prime $p \mid m_2$ is precisely the characteristic function of the set of quaternary quadratic forms in $U(\Z_p)$ that are of rank $\leq 1$ modulo $p$. Thus, we have the claim. The proposition now follows by applying Theorem~\ref{thmainWA},using the relation~\eqref{eq-tauident}, and observing that $a \leq m_1m_2$.
\end{proof}

Finally, we have the following consequence of Proposition~\ref{propgooda}:
\begin{theorem}\label{thcIboundlbo}
Let $i\in\{0,1,2+\}$. For some $\lambda>0$, we have that
\begin{equation} \label{eq-sumoveralla}
\cI^{(i)}\leq
\frac{\alpha_\infty|\J_\DA|\Vol(\FF_\DA)}{\sigma_i\Vol(\Theta)}
X^{4/6}\int_{\substack{(t,n)\in\FF_2\\t\ll X^{1/24}}}\Bigl(\sum_{0 < |a| \ll \frac{X^{1/6}}{t^4}}
\six^{10} a^6\nu(T_a)\Vol\bigl(\psi|_{\frac{at^4}{X^{1/6}}}\bigr)\Bigr)t^4
dn\frac{d^\times t}{t^2}+O(X^{5/6-\lambda}).
\end{equation}
\end{theorem}
\begin{proof}
It suffices to prove that the contribution to the right-hand side of~\eqref{eq-sumoveralla} from bad integers $a$ is $O(X^{5/6-\lambda})$. For such $a$, we have $T_a \ll \cM(\FF;a,\cdot)$, which is a bounded function supported on the set
of elements $B$ that are special at $a$. By the proof of Theorem~\ref{thSLMT}, we have that
\begin{equation}\label{eqth51bound}
\nu(T_a) \ll \nu(|a|) = \prod_{p \parallel a} p^{-6} \times \prod_{p^2 \parallel a} p^{-12} \times \prod_{p^3 \parallel a} p^{-17} \times \prod_{\substack{p^k \parallel a \\ k \geq 4}} p^{-11k/2}.
\end{equation}
Moreover, since $\psi$ is smooth and compactly supported, the factor of $\on{Vol}\big(\psi|_{\frac{at^4}{X^{1/6}}}\big)$ in the right-hand side of~\eqref{eq-sumoveralla} is absolutely bounded. This, along with~\eqref{eqth51bound}, yields that the contribution from bad integers $a$ is
\begin{equation} \label{eq-badcontrgen}
\ll X^{4/6} \int_{t \ll X^{1/24}}\Bigl(\sum_{\substack{0 < |a| \ll \frac{X^{1/6}}{t^4} \\ \text{$a$ bad}}}
\prod_{\substack{p^k \parallel a \\ k \geq 3}} p^{k/2}\Bigr)tdt.
\end{equation}
Using the bound $X^{1/6}/(at^4) \gg 1$, we find that the contribution to~\eqref{eq-badcontrgen} from bad integers $a \in \cN(X^\theta)$ is
\begin{equation} \label{eq-bad1bound}
 X^{5/6} \Bigl(\sum_{\substack{0 < |a| \ll \frac{X^{1/6}}{t^4} \\ a \in \mathcal{N}(X^\theta)}}
\frac{1}{a}\prod_{\substack{p^k \parallel a \\ k \geq 3}} p^{k/2}\Bigr) \,\ll_\epsilon\, X^{5/6+\epsilon} \sum_{\substack{|b| \geq 1 \\ \text{$b$ squareful}}}\sum_{\substack{|c| \geq X^{\theta}/b  \\ \text{$c$ cubeful}}} \frac{1}{b\sqrt{c}} \ll_\epsilon  X^{5/6 - \theta/6 + \epsilon}.
\end{equation}
Similarly, the contribution to~\eqref{eq-badcontrgen} from bad integers $a \leq (X^{1/6}/t^4)^{1-\kappa}$ is
\begin{equation} \label{eq-bad2bound}
  \ll_\epsilon X^{5/6-\kappa/6+\epsilon} \sum_{\substack{1 \leq |c| \leq (X^{1/6}/t^4)^{1-\kappa} \\ \text{$c$ cubeful}}} \frac{1}{\sqrt{c}} \ll_\epsilon  X^{5/6 - \kappa/6 + \epsilon}.
\end{equation}
Combining the bounds~\eqref{eq-bad1bound} and~\eqref{eq-bad2bound} yields the theorem.
\end{proof}

\subsection{Computing the local volumes and summing over leading coefficients}

Recall from \eqref{eqTa} that we denoted
$\cM^{(\log X_\theta,X_\theta)}(\FF;a,\cdot)$ by $T_a$.
We begin with the following lemma.

\begin{lemma}\label{lemapproxbound}
Let $p$ be a prime, and let $\ell$ be a positive integer. Then we have that
\begin{equation*}
\nu_p(\cM_{p,\ell}(\FF; a;\cdot))= \int_{W_\DA(\Z_p)}\cM_p(\FF;a;B)dB+O(p^{-\ell/60}).
\end{equation*}
\end{lemma}
\begin{proof}
Recall that $\cM_{p,\ell}(\FF;a;\cdot)$ is the best upper approximation to $\cM_p(\FF;a;\cdot)$ defined modulo
$p^\ell$. If two elements in $W_\DA(\Z_p)$ that are congruent modulo $p^\ell$, say $B$ and $B+p^\ell B'$, have different values under $\cM_p(\FF;a;\cdot)$, then we may assume that there exists an element $\gamma\in \PSO_\DA(\Q_p) \smallsetminus \PSO_\DA(\Z_p)$ such that $\gamma \cdot B\in W_\DA(\Z_p)$ but $\gamma \cdot (B+p^\ell B')\not\in W_\DA(\Z_p)$. We write $\gamma=\sigma_1\gamma'\sigma_2$, where $\sigma_1$ and $\sigma_2$ belong to $\PSO_\DA(\Z_p)$, and $\gamma'$ is a diagonal matrix. It therefore follows that $\gamma'\sigma_2 \cdot p^\ell B'\not\in W_\DA(\Z_p)$. In particular, it follows that $\gamma'$ has a nonzero coefficient whose valuation at $p$ is at most $-\ell/2$. However, since $\gamma'\sigma_2 \cdot B$ is integral, it is easy to deduce that the discriminant of $\sigma_2 \cdot B$, and hence the discriminant of $B$, is divisible by $p^{\ell/5}$. Therefore, we have
\begin{equation*}
    \nu_p(T_a)-
\int_{W_\DA(\Z_p)}\cM_p(\FF;a;B)dB\leq
\Vol\big(\{B\in W_\DA(\Z_p):p^{\ell/5}\mid\Delta(B)\}\big).
\end{equation*}
A result of Serre (see~\cite[\S3.2, Equation (57)]{MR644559}) implies the bound
\begin{equation*}
\#\{B\in W_\DA(\Z/p^n\Z):p^{n}\mid\Delta(B)\}=O(p^{119n/12}),
\end{equation*}
which in turn implies that the volume above is bounded by $O(p^{-\ell/60})$, as necessary.
\end{proof}

We have the following immediate consequence of~\cite[Corollary~1.2]{density}:

\begin{lemma}\label{lemspecialden}
Let $b\geq a\geq 1$ be positive integers, and let $p$ be a prime. The
$p$-adic density of quaternary quadratic forms $Q \in U(\Z_p)$ such that $Q$ is of rank $\leq 1$ modulo $p^a$ and is of
rank $\leq 2$ modulo $p^b$ is given by
$$p^{-3a-3b} \times \begin{cases} 1, & \text{ if $b = a$,} \\ 1+ p^{-2-4a}(1-p^{-7})^{-1}(1-p^{-3})(1-p^{7a-7 b}), & \text{ if $b > a$.} \end{cases}$$
\end{lemma}

We now have the following result.
\begin{proposition}\label{propperror}
For large enough $X$, we have the following bounds when $a$ is a good integer:
\begin{equation*}
\nu_p(T_a)
\leq\left\{
\begin{array}{rcl}
  \displaystyle \alpha_p|\J_\DA|_p|\six|_p^{10} |a|_p^6\tau_\DA(p)
  \int_{V_a(\Z_p)}\cM_p(\FF,f)df+O_\epsilon(p^{-\ell_p/60}X_\theta^\epsilon),
  & \mbox{if}& p<\log X_\theta,\; p^2 \nmid a;\\[.15in]
  \displaystyle \alpha_p|\J_\DA|_p|\six|_p^{10} |a|_p^6\tau_\DA(p)
  \int_{V_a(\Z_p)}\cM_p(\FF,f)df+O(p^{-\ell_p/60}|a|_p^6),
  &\mbox{if}& p<\log X_\theta , \; p^2 \mid a;\\[.15in]
|a|_p^6, &\mbox{if}& p\geq\log X_\theta \mbox{ and } p^2\nmid a;\\[.15in]
|a|_p^6+O(p^{-1}|a|_p^6), &\mbox{if}& p\geq\log X_\theta \mbox{ and } p^2\mid a.
\end{array}
\right.
\end{equation*}
When $a$ is a bad integer, we have that
\begin{equation*}
  \nu_p(T_a) =\left\{
\begin{array}{rcl}
  \displaystyle \alpha_p|\J_\DA|_p|\six|_p^{10} |a|_p^6\tau_\DA(p)
  \int_{V_a(\Z_p)}\cM_p(\FF,f)df,
  & \mbox{if}& p<\log X_\theta;\\[.15in]
  \displaystyle \alpha_p|\J_\DA|_p|a|_p^6
  \int_{V_a(\Z_p)}\cM_p(\FF,f)df,
  &\mbox{if}& p\geq \log X_\theta.
\end{array}
\right.
\end{equation*}
\end{proposition}
\begin{proof}
When $a$ is good and $p<\log X_\theta$ or when $a$ is bad, we first apply Lemma \ref{lemapproxbound}. The claim then follows from Theorem \ref{thJacgeneral} in a manner similar to \cite[Proposition 3.9]{MR3272925}.

Now assume that $a$ is good and that $p\geq \log X_\theta$.  The real number $X$ may be chosen
large enough to ensure that $|\six|_p=\alpha_p=|\J_\DA|_p=1$. When
$p\nmid a$, the result only claims an upper bound of $1$, which
follows from the fact that $\cM_{p,\ell}(\FF;a;\cdot)$ is bounded
above by $1$. When $p\parallel a$, the result follows since
$\cM_p^{(\log X_\theta,X_\theta)}(\FF;a,\cdot)=\cS_p(a;B)$, where
$\cS_p(a;B)$ is the characteristic function of the set of elements $B$
with $\F_p$-rank $\leq 1$. The $p$-adic density of these elements is precisely
$1/p^6$.

Finally, when $p^k \parallel a$, for $k\geq 2$, we note that once again,
$\cM_p^{(\log X_\theta, X_\theta)}(\FF;a,\cdot)=\cS_p(a;\cdot)$, the
characteristic function of the set of elements that are special at $p$.
Denote this set by $S$. By Proposition~\ref{prop-whatspecialmeans}, $S$ is contained in the finite union of sets of the form $p^\alpha \mc{L}_{k-2\alpha+\beta,k-\alpha-\beta}$, where $\alpha$ and $\beta$ range over nonnegative integers satisfying $2\alpha\leq k+\beta$ and $2\beta\leq \alpha$. It follows immediately from Lemma \ref{lemspecialden} that the $p$-adic density of $p^\alpha \mc{L}_{k-2\alpha+\beta,k-\alpha-\beta}$ is $p^{-6k-\alpha}(1+O(p^{-6}))$. Summing up over the possible values of $\alpha$ and $\beta$ gives a total $p$-adic density of $p^{-6k}(1+O(p^{-1}))$,
as necessary.
\end{proof}

We have the following immediate consequence of the above proposition:
\begin{corollary}\label{corlocalvolfinal}
For large enough $X$, we have that
\begin{equation*}
\alpha_\infty|\J_\DA|\Vol(\FF_\DA)a^6\nu\bigl(\cM^{(\log X_\theta,X_\theta)}
(\FF;a,\cdot)\bigr)
\leq 4\tau_\DA^{\log X_\theta} \prod_{p\leq \log X_\theta}\int_{V_a(\Z_p)}\cM_p(\FF,f)df
+O_\epsilon\bigl(X_\theta^{-1/120+\epsilon}\bigr)+E(a),
\end{equation*}
where $E(a) \defeq O\Bigl(\sum_{\substack{p^2\mid a\\p\geq \log X_\theta}}\frac{1}{p}\Bigr)$.
\end{corollary}
\begin{proof}
For good integers $a$, this is a consequence of Proposition~\ref{propperror}. For bad integers $a$, Proposition~\ref{propperror} implies the same bound with the truncated product over $p \leq \log X_\theta$ replaced by the product over all primes $p$ (and without the error terms!). But because $\cM_p(\FF,\cdot) \leq 1$ for all sufficiently large $p$, the product over all $p$ is at most the truncated product over $p \leq \log X_\theta$ for large enough $X$.
\end{proof}

Next, we have the following lemma, which is a consequence of Riemann
summation together with the fact that $\psi$ is compactly supported:
\begin{lemma}\label{lemRsum}
Let $Q<T_\theta$ for some $\theta<1$, and for each $p \leq Q$, let $\phi_p \colon V_a(\Z_p) \to \R$ be an acceptable function. Then we have that
\begin{equation*}
\sum_{a\ll_\psi T}\left(\Vol\bigl(\psi|_{\frac{a}{T}}\Bigr)
\prod_{p\leq Q}\int_{V_a(\Z_p)}\phi_p(f)df\right)=
T\Vol(\psi)\prod_{p\leq Q}\int_{V(\Z_p)}\phi_p(f)df+o(T),
\end{equation*}
where ``$a \ll_\psi T$'' means that $a$ is at most a positive constant times $T$, where the constant depends on $\psi$.
\end{lemma}

Recall the definition of $T_a$ in \eqref{eqTa}. Using Corollary
\ref{corlocalvolfinal} and Lemma \ref{lemRsum}, we find that
\begin{equation}\label{eqkeyineq}
\begin{array}{rcl}
&&\displaystyle\sum_{0 < |a| \ll_\psi \frac{X^{1/6}}{t^4}}
\alpha_\infty|\J_\DA|\Vol(\FF_\DA)
a^6\nu(T_a)\Vol\bigl(\psi|_{\frac{at^4}{X^{1/6}}}\bigr) \leq
\\[.2in]&&
\displaystyle \sum_{0 < |a| \ll_\psi \frac{X^{1/6}}{t^4}}
4\Vol\bigl(\psi|_{\frac{at^4}{X^{1/6}}}\bigr)\Bigl(\tau_\DA^{(\log X_\theta)}
\prod_{p\leq \log X_\theta}\int_{V_a(\Z_p)}\cM_p(\FF,f)df+O(X_\theta^{-1/120+\epsilon}+E(a))\Bigr)\leq
\\[.2in]&&\displaystyle
4\tau_\DA^{(\log X_\theta)}\frac{X^{1/6}}{t^4}
\Vol(\psi)\prod_{p\leq \log X_\theta}\int_{V(\Z_p)}\cM_p(\FF,f)df
+o\Bigl(\frac{X^{1/6}}{t^4}\Bigr)+O(E),
\end{array}
\end{equation}
where we estimate the error term using the following lemma:
\begin{lemma}
We have that
\begin{equation*}
\sum_{0 < |a| \ll_\psi \frac{X^{1/6}}{t^4}}\Vol\bigl(\psi|_{\frac{at^4}{X^{1/6}}}\bigr)(X_\theta^{-1/120+\epsilon}+E(a))\ll
(\log X_\theta)^{-2}\frac{X^{1/6}}{t^4}.
\end{equation*}
\end{lemma}
\begin{proof}
The first term is easily bounded: we have
\begin{equation*}
\sum_{0 < |a| \ll_\psi \frac{X^{1/6}}{t^4}} X_\theta^{-1/120+\epsilon}\ll \frac{X^{1/6}}{t^4X_\theta^{1/120-\epsilon}},
\end{equation*}
which is sufficient. For the second term, we interchange summation:
\begin{equation*}
\sum_{0 < |a| \ll_\psi \frac{X^{1/6}}{t^4}} E(a) \ll \sum_{p\geq \log X_\theta} \frac{1}{p} \times \frac{X^{1/6}}{p^2t^4}\ll (\log X_\theta)^{-2}\frac{X^{1/6}}{t^4}.
\end{equation*}
This concludes the proof of the lemma.
\end{proof}

In conjunction with Theorem \ref{thcIboundlbo}, the estimate \eqref{eqkeyineq}
yields the following.
\begin{equation*}
\cI^{(i)}\leq \frac{4\tau_\DA^{(\log X_\theta)}\Vol(\psi)}{\sigma_i\Vol(\Theta)}
\Vol(\FF_{\PGL})\prod_{p\leq \log X_\theta}\Vol(\cM_p(\FF,\cdot))X^{5/6}+o(X^{5/6}).
\end{equation*}
We know that $\Vol(\psi)/\Vol(\Theta)=|\J_{\PGL}|\Vol(\H^\pm)$ and that $\Vol(\cM_p(\FF,\cdot))=|\J_{\PGL}|_p\nu_p(\FF)\alpha_p$. Together
with the equalities
$1/\sigma_0+1/\sigma_{2+}=1/\sigma_1=\alpha_\infty$, we obtain
\begin{equation*}
  \cI^\pm\leq 8\tau_\DA^{(\log X_\theta)}\tau_\PGL^{(\log X_\theta)}\nu(\FF)\widehat{\H^\pm}(0)
  X^{5/6}+o(X^{5/6}).
\end{equation*}
Therefore, we have proven the following result.
\begin{theorem}
We have that
\begin{equation*}
\displaystyle\lim_{X\to\infty}\frac{1}{\nu(\FF)\widehat{\H^\pm}(0)X^{5/6}}
\sum_{E\in \FF^\pm} \sum_{\substack{\sigma\in\Sel_2(E)\\\sigma\neq 1}}\bigl(
|S_2(\sigma)|-2\bigr)
\H_X(\iota(E))\leq 8.
\end{equation*}
\end{theorem}
\begin{proof}
This follows immediately since $\tau_\DA^{(\log X_\theta)}$ and
$\tau_\PGL^{(\log X_\theta)}$ tend to $1$ as $X$ tends to infinity.
\end{proof}

Let $\chi^\pm\colon \Inv(\R)^\pm\to\R$ be the characteristic function of
the set of elements in $\Inv(\R)^\pm$ with height $\leq 1$.
For a sequence of real numbers $\epsilon\to 0$, consider smooth,
compactly supported functions $\H_\epsilon^{\pm}$ such that
\begin{enumerate}[itemsep=0pt]
\item For every $x\in\Inv(\R)^\pm$, we have that
  $\H_\epsilon^{\pm}(x)>\chi^\pm(x)$.
\item For every $\epsilon$, we have that
\begin{equation*}
\int_{\Inv(\R)}\bigl(\H_\epsilon^{\pm}(x)-\chi(x)\bigr)<\epsilon.
\end{equation*}
\end{enumerate}
Then it follows that
\begin{align*}
\displaystyle\frac{1}{\nu(\FF)\widehat{\chi^\pm}(0)X^{5/6}}
\sum_{\substack{E\in\FF^\pm\\H(E)<X}}|\Sel_2(E)|^2&=
\displaystyle 3+\frac{1}{\nu(\FF)\widehat{\chi^\pm}(0)X^{5/6}}
\sum_{\substack{E\in\FF^\pm\\H(E)<X}}(|\Sel_2(E)|-1)|\Sel_2(E)|+o(1)
\\ &=\displaystyle
3+\frac{1}{\nu(\FF)\widehat{\chi^\pm}(0)X^{5/6}}
\sum_{\substack{E\in\FF^\pm\\H(E)<X}}
\sum_{\substack{\sigma\in\Sel_2(E)\\\sigma\neq 1}}\bigl(
|S_2(\sigma)|-2\bigr) +
\\ &\displaystyle\hphantom{=3}
\frac{1}{\nu(\FF)\widehat{\chi^\pm}(0)X^{5/6}}
\sum_{\substack{E\in\FF^\pm\\H(E)<X}}
\sum_{\substack{\sigma\in\Sel_2(E)\\\sigma\neq 1}}2+o(1)
\\ &\leq \displaystyle
7+\frac{1}{\nu(\FF)\widehat{\chi^\pm}(0)X^{5/6}}
\sum_{E\in \FF^\pm}\sum_{\substack{\sigma\in\Sel_2(E)\\\sigma\neq 1}}\bigl(
|S_2(\sigma)|-2\bigr)\H^\pm_{\epsilon,X}(\iota(E))+o(1)
\\ &\leq \displaystyle
7+\frac{8\nu(\FF)\widehat{\H_\epsilon^\pm}(0)X^{5/6}}
{\nu(\FF)\widehat{\chi^\pm}(0)X^{5/6}}+o_\epsilon(1)
\\[0.1cm] &\leq \displaystyle
15+O(\epsilon)+o_\epsilon(1).
\end{align*}
Letting $X$ tend to infinity first, and then letting $\epsilon$ tend to $0$, we obtain Theorem~\ref{thm-main}.

\subsection*{Acknowledgments}

We are very grateful to Levent Alp\"{o}ge, Aaron Landesman, Peter Sarnak, Michael Stoll, and Jerry Wang for helpful conversations. The first-named author was supported by a Simons Investigator Grant and NSF grant~DMS-1001828. The second-named author was supported by an NSERC discovery grant and Sloan fellowship. The third-named author was supported by the NSF Graduate Research Fellowship.

\bibliographystyle{abbrv} \bibliography{bibfile}
\end{document}